\documentclass[11pt]{article}
\usepackage{amssymb, amsthm, amsmath, fontenc, psfrag, epsfig}
\usepackage{url}
\theoremstyle{plain}
\newtheorem{thm}{Theorem}[section]
\newtheorem{prop}[thm]{Proposition}
\newtheorem{cor}[thm]{Corollary}
\newtheorem{lem}[thm]{Lemma}

\theoremstyle{definition}
\newtheorem{defn}[thm]{Definition}

\theoremstyle{remark}

\title{Deformation spaces of Kleinian surface groups are not locally connected}
\author{Aaron D. Magid\footnote{The author was partially supported by NSF grants DMS-0902764, 
DMS-0554239, 
DMS-0602191.
}}
\date{21 March 2010}
\begin{document}
\maketitle  

\begin{abstract}
For any closed surface $S$ of genus $g \geq 2$, we show that the deformation space of marked hyperbolic 3-manifolds homotopy equivalent to $S$, $AH(S \times I)$, is not locally connected. This proves a conjecture of Bromberg who recently proved that the space of Kleinian punctured torus groups is not locally connected.  Playing an essential role in our proof is a new version of the filling theorem that is based on the theory of cone-manifold deformations developed by Hodgson, Kerckhoff, and Bromberg. 
\end{abstract}

\section{Introduction} \label{Introduction}

Given a compact, orientable $3$-manifold $N$, let $AH(N)$ denote the set of marked hyperbolic $3$-manifolds homotopy equivalent to $N$, equipped with the algebraic topology.  The interior of $AH(N)$, as a subset of the $PSL(2, \mathbb{C})$-character variety, is well-understood; however, little is known about the topology of entire deformation space and its dependence on the topology of $N$.  We show that there are points where $AH(N)$ fails to be locally connected when $N$ is homotopy equivalent to a closed surface. 

The work of Ahlfors \cite{Ahlfors}, Bers \cite{Bers}, Kra \cite{Kra}, Marden \cite{Marden open}, Maskit \cite{Maskit}, Sullivan \cite{Sullivan}, and Thurston \cite{Th2} shows that the components of the interior of $AH(N)$ are in one-to-one correspondence with the marked homeomorphism types of compact $3$-manifolds homotopy equivalent to $N$.  Using the theory of quasiconformal deformations and the measurable Riemann mapping theorem, each of these components can be parameterized by analytic information.  

Unfortunately, our understanding of the interior of $AH(N)$ does not extend to the entire space.  When the boundary of $N$ is incompressible, Anderson, Canary, and McCullough \cite{ACM} characterized when two components of the interior of $AH(N)$ have intersecting closures.  They called this phenomenon bumping.  For any genus $g \geq 2$ surface $S$, McMullen \cite{McMullen} showed that the interior of $AH(S \times I)$ self-bumps.  This means that there is a point $\rho \in AH(S \times I)$ such that whenever $U$ is a sufficiently small neighborhood of $\rho$, the intersection of $U$ and the interior of $AH(S \times I)$ is disconnected.  Bromberg and Holt \cite{BH} generalized this result by showing that whenever $N$ contains a primitive, essential annulus that is not homotopic into a torus boundary component of $N$ then the interior of $AH(N)$ self-bumps. 

Recent work by Agol \cite{Agol}, Calegari and Gabai \cite{CG}, Brock, Canary, and Minsky \cite{BCM}, and many others has led to a classification of hyperbolic manifolds up to isometry.  The existence of bumping and self-bumping points shows that the invariants used in this classification do not vary continuously at certain points on the boundary of the deformation space (see also \cite{Brock}).  Thus, further study of the local topology of $AH(N)$ near these points is necessary in order to fully understand these spaces of hyperbolic manifolds.    

Bromberg \cite{B3} recently showed that the space of Kleinian punctured torus groups is not locally connected.  The points where this deformation space fails to be locally connected are self-bumping points, but he also showed that the space is locally connected at other self-bumping points.  This indicates that bumping may be considerably more complicated than we previously thought.  He also conjectured that $AH(S \times I)$ would fail to be locally connected for any surface $S$, although his arguments in the punctured torus case made essential use of Minsky's \cite{Minsky} classification of punctured torus groups.  The results in \cite{Minsky} that Bromberg uses do not generalize to higher genus surfaces. 

The following theorem proves Bromberg's conjecture. 
\begin{thm} \label{main theorem}  Let $S$ be a closed surface of genus $g \geq 2$. Then $AH(S \times I)$ is not locally connected. 
\end{thm}

Using Bromberg's description of the space of punctured torus groups, one can show that many relative deformation spaces fail to be locally connected \cite{Magid}, although Theorem \ref{main theorem} provides the first examples of non-relative deformation spaces that fail to be locally connected. Non-relative means that no parabolicity condition is specified. See Section \ref{background deformation theory} for a discussion of pared manifolds and deformation spaces.

In our proof that $AH(S \times I)$ fails to be locally connected for any higher-genus surface $S$ we frequently refer to many of the arguments in \cite{B3}.  In particular, the results in Section \ref{local homeomorphism section} rely on and/or generalize the results of Section 3 of \cite{B3}.  However, since some of Minsky's results in \cite{Minsky} do not extend to higher genus surfaces, we make a significant departure from Bromberg's methods in Section \ref{closure not locally connected section}.  In this section, we make use of an improved version of the filling theorem, the key technical result of this paper.   

Before stating our version of the filling theorem, we will set up some of the notation.  Given a geometrically finite hyperbolic manifold $\hat{M}$ with a rank-$2$ cusp, the filling theorem provides sufficient conditions for one to ``Dehn-fill'' the cusp. That is, if $\hat{M}$ is homeomorphic to the interior of a compact manifold $\hat{N}$ with a torus boundary component corresponding to the cusp of $\hat{M}$, and $N$ is a Dehn-filling of $\hat{N}$, then the filling theorem provides conditions for one to construct a hyperbolic manifold $M$ homeomorphic to the interior of $N$ with the same conformal boundary as $\hat{M}$.   Assuming the hypotheses of the theorem are satisfied, one obtains a relationship between the metrics on $\hat{M}$ and $M$.  

 Suppose $T$ is a rank-$2$ cusp in $\hat{M}$ and $\beta$ is the slope in $T$ along which we are filling.  Let $L$ be the normalized length of $\beta$ in $T$, and let $A^2$ be the reciprocal of the normalized twist of the cusp.  Although we relegate the actual definitions of the normalized length and the normalized twist to Section \ref{filling theorem section}, we now describe these quantities with respect to a particular normalization of the cusp (the normalization that we will use throughout Sections \ref{local homeomorphism section} and \ref{closure not locally connected section}).  Suppose the rank-$2$ cusp $T$ of $\hat{M}$ is generated by parabolics $\begin{pmatrix} 1 & 2 \\ 0 & 1 \end{pmatrix}$ and $\begin{pmatrix} 1 & w \\ 0 & 1 \end{pmatrix}$, and that $\beta$ corresponds to $\begin{pmatrix} 1 & w \\ 0 & 1 \end{pmatrix}$.  If $Im(w)>0$ and $\frac{|w|^2}{2|Re(w)|} > 2$,  then $L^2$ and $A^2$ are given by:
 $$L^2 = \frac{|w|^2}{2Im(w)} \quad \text{and} \quad  A^2 = \frac{|w|^2}{2Re(w)}.$$

 For any curve $\gamma \subset M$, let $B \in PSL(2, \mathbb{C})$ denote the corresponding isometry in $\pi_1(M)$. The complex length of $\gamma$ is the value of $\mathcal{L} = l + i\theta$ such that $tr^2(B) =4\cosh^{2}\left(\frac{\mathcal{L}}{2} \right)$, $l \geq 0$, and $\theta \in (-\pi, \pi]$.  For a geodesic $\gamma$, the real part $l$ gives the length of $\gamma$ in $M$ which is the distance that $B$ translates along its axis.  The imaginary part $\theta$ is the amount $B$ rotates about its axis.

Let $\epsilon_3$ denote the Margulis constant for $\mathbb{H}^3$. If $\gamma$ is the core curve of the solid filling torus in $M$.   Then for any $\epsilon_3 \geq \epsilon > 0$, let $\mathbb{T}_\epsilon(T)$ (resp. $\mathbb{T}_\epsilon(\gamma)$) denote the $\epsilon$-Margulis tube about $T$ (resp. $\gamma$).  

\begin{thm} \label{filling}
Let $J>1$ and $\epsilon_3 \geq \epsilon > 0$.  There is some $K \geq 8(2\pi)^2$ such that the following holds:  suppose $\hat{M}$ is a geometrically finite hyperbolic $3$-manifold with no rank-$1$ cusps, $T$ is a rank-$2$ cusp in $\hat{M}$, and $\beta$ is a slope on $T$ such that the normalized length of $\beta$ is at least $K$ (i.e., $L^2 \geq K^2$), then   \\
$(i)$ the $\beta$-filling of $\hat{M}$, which we call $M$, exists;  \\
$(ii)$  the real part of the complex length $\mathcal{L} = l+i\theta$ of the core curve of the filling torus $\gamma$ in $M$ is approximately $\frac{2\pi}{L^2}$ with error bounded by
\[
\left| l - \frac{2\pi}{L^2} \right| \leq \frac{8(2\pi)^3}{L^4 - 16(2\pi)^4};
\] 
\\
$(iii)$  in particular, the length of $\gamma$ is bounded above by $\frac{2\pi}{L^2 - 4(2\pi)^2}$; \\
$(iv)$  there exists a $J$-biLipschitz diffeomorphism 
\[
\phi : \hat{M} - \mathbb{T}_\epsilon (T) \to M - \mathbb{T}_\epsilon (\gamma).
\]
$(v)$  If, in addition to $L^2 \geq K^2$, we have $|A^2| \geq 3$, then the imaginary part of the complex length $\mathcal{L} = l+i\theta$ of the core curve of the filling torus $\gamma$ in $M$ (chosen so $\theta \in (-\pi, \pi]$) is approximately $\frac{2\pi}{A^2}$ with error bounded by
\[
 \left| \theta - \frac{2\pi}{A^2} \right| \leq \frac{5(2\pi)^3}{(L^2 - 4(2\pi)^2)^2}.
\] 
\end{thm}

The proof of Theorem \ref{filling} is contained in Section \ref{filling theorem section}.   Although our version may be stated differently, parts $(i)-(iii)$ can be found in the work of Hodgson and Kerckhoff \cite{HK1, HK2} on cone-manifold deformations which was generalized to geometrically finite manifolds by Bromberg \cite{B1, B2}.  Part $(iv)$ follows from the drilling theorem of Brock and Bromberg \cite{BB}.   The most original part of this version of the filling theorem is the estimate in part $(v)$, although its proof also relies on the Hodgson-Kerckhoff cone-manifold technology.  Some of the background cone-manifold deformation theory and a summary of the work of Bromberg, Hodgson, and Kerckhoff on cone-manifolds can be found in Section \ref{background cone}.

We now outline the proof of Theorem \ref{main theorem}.  We begin by parameterizing a 
 subset of $AH(S \times I)$.  If $P \subset S \times \{1\}$ is a pants decomposition, then $MP(S \times I, P)$ denotes the subset of the boundary of $AH(S \times I)$ consisting of the marked hyperbolic 3-manifolds that are homeomorphic to the interior of $S \times I$, are geometrically finite, have a rank-$1$ cusp associated to each component of $P$, and contain no other cusps (see Section \ref{background deformation theory} for this notation).  
 
 We define a subset $\mathcal{A} \subset MP(S \times I, P) \times \hat{\mathbb{C}}^{3g-3}$ and a map 
\[
\Phi: \mathcal{A} \to AH(S \times I), 
\] and we show there is some $\sigma_0 \in MP(S \times I, P)$ and a neighborhood $U$ of $(\sigma_0, \infty, \ldots, \infty)$ in $\mathcal{A}$ such that $\Phi(\sigma_0, \infty, \ldots, \infty) = \sigma_0$ and $\Phi \vert_U : U \to \Phi(U)$ is a homeomorphism.  The map $\Phi$ can roughly be described as follows. Let $d = 3g-3$.  If $(\sigma, \infty, \ldots, \infty) \in \mathcal{A}$ then we define $\Phi(\sigma, \infty, \ldots, \infty) = \sigma$.  If $(\sigma, w_1, \ldots, w_d) \in \mathcal{A}$ for some $(w_1, \ldots, w_d) \in \mathbb{C}^d$, we use the $w$-coordinates to define a marked hyperbolic $3$-manifold with $d$ rank-$2$ cusps.  To each rank-$2$ cusp, one can associate a conformal structure on a torus, and $w_i$ acts as a Teichm\"uller parameter for the $i$th cusp.  We then use the filling theorem (Theorem \ref{filling}) to fill in these cusps and obtain a marked hyperbolic $3$-manifold in the interior of $AH(S \times I)$.  
We define $\mathcal{A}$ to exclude points in $MP(S \times I, P) \times \hat{\mathbb{C}}^{3g-3}$  where some, but not all, of the $w$-coordinates are $\infty$.

This parameterization of the subset $\Phi(U) \subset AH(S\times I)$ is a straightforward generalization of the results in Section 3 of Bromberg \cite{B3}.  We set up the necessary background in Section \ref{background deformation theory} and describe the parameterization in Section \ref{local homeomorphism section}.  This parameterization is an application of parts $(i)-(iv)$ of Theorem \ref{filling} and Corollary \ref{multiple fillings}, which is a generalization of the filling theorem for multiple cusps.  

In Section \ref{A not loc conn section}, we use results of Section 4 of Bromberg \cite{B3} to show that $\mathcal{A}$ is not locally connected at $(\sigma_0 ,\infty, \ldots, \infty)$.  Moreover, we find that in some neighborhood $U$ of $(\sigma_0 ,\infty, \ldots, \infty)$ in $\mathcal{A}$, there exists $\delta>0$ and subsets $C_n \subset U$ accumulating at $(\sigma_0 ,\infty, \ldots, \infty)$ such that for any $(\sigma, w_1, \ldots, w_d) \in \overline{C_n}$ and any $(\sigma', w_1', \ldots, w_d') \in \overline{U-C_n}$, we have $|w_1 - w_1'| > \delta$ for all $n$ (see Lemma \ref{lowerboundw}).  Heuristically, we think of the sets $C_n$ as being components of $U$ that are bounded apart from the rest of $U$ by a lower bound that is independent of $n$.  In actuality, these sets are likely collections of components.  

Finally, in Section \ref{closure not locally connected section}, we show that $AH(S\times I)$ is not locally connected at $\sigma_0$.  By Lemma \ref{lowerboundw}, there is a lower bound on the distance between the first $w$-coordinate (\textit{i.e.}, the first coordinate of the $\hat{\mathbb{C}}^{3g-d}$ factor of $\mathcal{A}$) of a point in $\overline{C_n}$ and the first $w$-coordinate of a point in $\overline{U-C_n}$. We then use the filling theorem to estimate the complex length of a curve in $\Phi(\sigma, w_1, \ldots, w_d) \in AH(S \times I)$ based on $(w_1, w_2, \ldots, w_d)$.  The control on the $w_1$-coordinate from Lemma \ref{lowerboundw} and the quality of the estimates in the filling theorem show that for all but finitely many $n$, $\overline{\Phi(C_n)}$ and $\overline{\Phi(U-C_n)}$ must be disjoint.  Hence, $\overline{\Phi(U)}$ has infinitely many components that accumulate at $\sigma_0$.  It follows from the Density Theorem (Theorem \ref{density theorem}) that $\overline{\Phi(U)}$ contains a neighborhood of $\sigma_0$ in $AH(S\times I)$; hence, $AH(S\times I)$ is not locally connected at $\sigma_0$.


These results originally formed the content of my Ph.D. Thesis at the University of Michigan \cite{MagidThesis}.  I owe my advisor, Dick Canary, many thanks for introducing me to this problem, teaching me about deformation spaces and bumping, and helping me to correct errors and to clarify my work.  Initial progress on this work began during the Teichm\"uller Theory and Kleinian Groups program at MSRI during the Fall of 2007, and I'd like to thank Juan Souto, Steve Kerckhoff, and Ken Bromberg for their direction and many helpful conversations.  The results in this paper generalize several different result of Ken Bromberg and he has been very helpful in explaining his work, as well as some of the background and intuition underlying it.


\section{Background Deformation Space Theory} \label{background deformation theory}

In this section, we recall the definition of a pared $3$-manifold $(N,P)$ and define the relative deformation space $AH(N,P)$.  This is a space of hyperbolic $3$-manifolds homotopy equivalent to $N$ with cusps associated to annuli and tori in $P$. We will review the Ahlfors-Bers parameterization that describes the interior of $AH(N,P)$ and set up some of the notation that will be used later.  For more information about pared manifolds and deformation spaces, see Chapters 5 and 7 of \cite{CM} respectively.  For a survey of the Density Theorem and bumponomics, see \cite{survey}.   

\subsection{Kleinian Groups and Pared Manifolds} \label{KM combination}

Before turning to deformation spaces, we begin with a brief review of Kleinian group theory in order to set up some notation.  A Kleinian group is a discrete subgroup of $PSL(2, \mathbb{C}) \cong Isom^+(\mathbb{H}^3)$.  We will assume that all of our Kleinian groups are finitely generated, torsion-free, and not virtually abelian.  The action of a Kleinian group $\Gamma$ on $\mathbb{H}^3$ extends to an action on $\partial \mathbb{H}^3 \cong \hat{\mathbb{C}}$ by M\"obius transformations.  The domain of discontinuity $\Omega(\Gamma)$ is the largest open $\Gamma$-invariant subset of $\hat{\mathbb{C}}$ on which the action of $\Gamma$ is properly discontinuous.  The quotient $\Omega(\Gamma)/\Gamma$ is called the conformal boundary of $M_\Gamma = \mathbb{H}^3/\Gamma$.  The limit set, $\Lambda(\Gamma)$, is the complement of $\Omega(\Gamma)$ in $\hat{\mathbb{C}}$, and the convex core of $M_\Gamma$ is the quotient of the convex hull (in $\mathbb{H}^3$) of the limit set by the group action.

\textbf{Pared Manifolds}

A pared 3-manifold is a pair $(N,P)$ where $N$ is a compact, oriented, hyperbolizable 3-manifold that is not a 3-ball, and $P \subset \partial N$ is a disjoint collection of incompressible annuli and tori satisfying the following properties: \vspace{-0.2in}
\begin{enumerate}
\item $P$ contains all of the tori in $\partial N$, and 
\item every $\pi_1$-injective map $(S^1 \times I, S^1 \times \partial I) \to (N,P)$ is homotopic, as a map of pairs, into $P$.
\end{enumerate} \vspace{-0.2in}

To avoid some degenerate cases in the statements that follow, we will assume throughout this paper that $\pi_1(N)$ is not virtually abelian.  This will ensure that any Kleinian group isomorphic to $\pi_1(N)$ is non-elementary.

\textbf{Geometrically Finite Kleinian Groups}

A hyperbolic $3$-manifold $M_\Gamma = \mathbb{H}^3/\Gamma$ is geometrically finite if and only if the union of $M_\Gamma$ with its conformal boundary, $(\mathbb{H}^3 \cup \Omega(\Gamma))/\Gamma$, is homeomorphic to $N-P$ for some pared $3$-manifold $(N,P)$.  A Kleinian group $\Gamma$ is geometrically finite if and only if the corresponding $3$-manifold $M_\Gamma$ is geometrically finite. There are many other equivalent notions of geometric finiteness.  See Bowditch \cite{Bowditch} for a more complete discussion.

\textbf{Thick-Thin Decomposition}

Let $M$ be a hyperbolic manifold.  For any $\epsilon > 0$, we define the $\epsilon$-thin part of $M$ to be the set of points $x \in M$ where the injectivity radius is at most $\epsilon$ and denote this set by $M^{\leq \epsilon}$. 
By the Margulis lemma, there is some constant $\epsilon_3$ (depending only on the dimension) such that for any $\epsilon_3 \geq \epsilon>0$, the $\epsilon$-thin part of $M$ consists of a disjoint union of metric collar neighborhoods of short geodesics and cusps.  We use the notation $\mathbb{T}_\epsilon (\gamma)$ to denote the component of $M^{\leq \epsilon}$ associated to a geodesic $\gamma$ and $\mathbb{T}_\epsilon(T)$ to denote the Margulis $\epsilon$-thin region associated to a rank-$2$ cusp $T$.  We let $\mathbb{T}_\epsilon^{par}$ denote the union of the Margulis $\epsilon$-thin regions associated to parabolics (\textit{i.e.}, the rank-$1$ and rank-$2$ cusps).

\textbf{Klein-Maskit Combination} 

Let $H$ be a subgroup of $\Gamma$.  A subset $B \subset \hat{\mathbb{C}}$ is precisely invariant under $H$ in $\Gamma$ if  \vspace{-0.2in}
\begin{enumerate}
\item for all $h \in H$, $h(B) = B$, and 
\item for all $\gamma \in \Gamma - H$, $\gamma (B) \cap B = \emptyset$.  \end{enumerate} \vspace{-0.2in}

For example, if $H$ is the infinite cyclic group generated by $\begin{pmatrix} 1 & 2 \\ 0 & 1 \end{pmatrix}$ and $\Gamma$ is a geometrically finite group containing $H$ with a rank-$1$ cusp corresponding to $H$ (\textit{i.e.}, the largest abelian subgroup of $G$ containing $H$ is $H$), then there is some $R$ such that the two sets
\[
 B_R^+ = \{ z \in \mathbb{C} \; : \; Im(z) > R \} \; \text{and} \; B_R^- = \{ z \in \mathbb{C} \; : \; Im(z) < -R \}
 \] are precisely invariant under $H$ in $\Gamma$ (\textit{e.g.}, see p. 125 of \cite{Marden}).  

Precisely invariant sets are useful for constructing Kleinian groups via a process known as Klein-Maskit combination.  We will use statements similar to those in \cite{AM}, but one should also refer to \cite{Maskit original, Maskit paper, Maskit paper 2}.  

Suppose $G_1, G_2$ are two geometrically finite Kleinian groups with $G_1 \cap G_2 = H$.  Here, $H$ could be any subgroup, but we will only be interested in the case that $H$ is the infinite cyclic parabolic subgroup of the previous example.  If there is a Jordan curve $c$ bounding two open discs $B_1, B_2$ in $\hat{\mathbb{C}}$ such that $B_i$ is precisely invariant under $H$ in $G_i$, then the group $G$ generated by $G_1$ and $G_2$ is geometrically finite and isomorphic to the amalgamated free product $G_1 *_H G_2$.  In this case, we say that the group $G$ is obtained from $G_1$ and $G_2$ by type I Klein-Maskit combination along the subgroup $H$. 

We now describe type II Klein-Maskit combination.  Let $G$ be a geometrically finite Kleinian group containing $H$.  Let $f \in PSL(2, \mathbb{C})$ such that $fHf^{-1} \subset G$. Suppose there is a Jordan curve $c$ bounding a disc $B \subset \hat{\mathbb{C}}$ such that 
\vspace{-0.2in}
\begin{enumerate}
\item $B$ is precisely invariant for $H$ in $G$, 
\item $\hat{\mathbb{C}} - f(\overline{B})$ is precisely invariant for $f H f^{-1} \subset G$, and
\item  $gB \cap \left( \hat{\mathbb{C}} - f(\overline{B}) \right) = \emptyset$ for all $g \in G$.
\end{enumerate} \vspace{-0.2in}
Then the group $\Gamma$ generated by $G$ and $f$ is geometrically finite and isomorphic to the HNN extension $G *_{\left< f \right>}$.  

Again, while type II Klein-Maskit combination can be applied in a more general setting, consider a geometrically finite group $G$ containing $H$ as above, and consider $f = \begin{pmatrix} 1 & w \\ 0 & 1 \end{pmatrix}$. Note that $f H f^{-1}= H$.  There is some $R$ such that $B_R^-$ and $B_R^+$ are precisely invariant under $H$ in $G$.  Moreover, we can assume that for all $g \in G$, $g B_{R}^- \cap B_R^+ = \emptyset$.  Then if $Im(w) \geq 2R$, the group $G *_{\left< f \right>}$ is geometrically finite. 

\subsection{Deformation Spaces}

We define the relative representation variety 
\[
\mathcal{R}(N,P) = Hom_P (\pi_1(N), PSL(2, \mathbb{C}))
\] to be the set of representations $\rho: \pi_1(N) \to PSL(2, \mathbb{C})$ such that $\rho(g)$ is parabolic or the identity whenever $g \in \pi_1(P)$.  We then define the relative character variety $X(N,P)$ to be the Mumford quotient of the relative representation variety
\[
X(N,P) = \mathcal{R}(N,P)//PSL(2, \mathbb{C}).
\]
Although the Mumford quotient is defined algebraically, non-radical points in the character variety can be identified with conjugacy classes of representations (\textit{i.e.}, points in the topological quotient $Hom_P (\pi_1(N), PSL(2, \mathbb{C}))/PSL(2, \mathbb{C})$)   
(see p. 62 of \cite{Kap}).  
Since we have assumed $\pi_1(N)$ is not virtually abelian, $\rho(\pi_1(N))$ is non-elementary (and therefore non-radical) for any discrete, faithful representation $\rho$. For these representations, 
we will make no distinction between conjugacy classes of representations and points in $X(N,P)$.  See also Section 1 of \cite{CS}.

Let $AH(N,P)$ denote the subset of $X(N,P)$ consisting of the conjugacy classes of representations that are discrete and faithful. Thus $AH(N,P)$ inherits a topology from the character variety known as the algebraic topology.  Results of Chuckrow \cite{Chuckrow} and J\o rgensen \cite{Jorgensen} show that $AH(N,P)$ is a closed subset of $X(N,P)$ with respect to this topology.  Since $\pi_1(N)$ is not virtually abelian, a neighborhood of $AH(N,P)$ is a smooth complex manifold, and the topology on $AH(N,P)$ is the same as the topology when considered as a subset of the topological quotient of $Hom_P (\pi_1(N), PSL(2, \mathbb{C}))$ by $PSL(2, \mathbb{C})$ acting by conjugation (Chapter 4 of \cite{Kap}).  

The space $AH(N,P)$ is a deformation space of hyperbolic $3$-manifolds in the following sense.  Given $\rho \in AH(N,P)$, the image group $\rho(\pi_1(M))$ defines a hyperbolic manifold $M_\rho = \mathbb{H}^3/\rho(\pi_1(N))$. Moreover (since $N$ is aspherical) the representation determines a homotopy equivalence $f_\rho: N \to M_\rho$, defined up to homotopy.  So points in $AH(N,P)$ can be identified with equivalence classes of marked hyperbolic $3$-manifolds $(M,f)$ where $f: N \to M$ is a homotopy equivalence such that $f(P)$ is homotopic into the cusps of $M$.  Two pairs $(M_1, f_1)$ and $(M_2, f_2)$ correspond to the same point of $AH(N,P)$ if there is an orientation preserving isometry $g: M_1 \to M_2$ such that $f_2 \simeq g \circ f_1$. 

We say that $\rho \in AH(N,P)$ is minimally parabolic if $\rho(g)$ is parabolic if and only if $g \in \pi_1(P)$.  A representation $\rho \in AH(N,P)$ is geometrically finite if $\rho(\pi_1(N))$ is a geometrically finite subgroup of $PSL(2, \mathbb{C})$. Results of Marden \cite{Marden open} and Sullivan \cite{Sullivan} show that when $\partial N - P \neq \emptyset$, the interior of $AH(N,P)$ consists of precisely the conjugacy classes of representations that are both geometrically finite and minimally parabolic, and we denote this set by $MP(N,P)$.  We now describe what is known as the Ahlfors-Bers parameterization of each of the components of $MP(N,P)$ in the case that $\partial N - P$ is incompressible.  See Chapter 7 of \cite{CM} for a more complete description of this parameterization including when $\partial(N,P)$ is compressible.

To enumerate the components of $MP(N,P)$, we first define $A(N,P)$ to be the set of marked pared homeomorphism types.  More precisely, $A(N,P)$ is the following set of equivalence classes:
\begin{align*}
A(N,P) = \{ [(N',P'), &h] \; : \; (N',P') \;  \text{is a compact, oriented, pared 3-manifold},  \\ 
& h:(N,P) \to (N',P') \;  \text{is a pared homotopy equivalence} \}/\sim
\end{align*} where $[(N_1, P_1), h_1] \sim [(N_2, P_2), h_2]$ if there exists an orientation preserving pared homeomorphism $j : (N_1, P_1) \to (N_2, P_2)$ such that $j \circ h_1$ is pared homotopic to $h_2$.    

Recall that we can identify $\rho \in AH(N,P)$ with a marked hyperbolic 3-manifold $(M_\rho, f_\rho)$.  Any 3-manifold with finitely generated fundamental group admits a compact core \cite{Scott}.  A relative compact core $C$ of $M_\rho$ is a compact core for $M_\rho - \mathbb{T}_\epsilon^{par}$ such that $\partial C$ meets every non-compact component of the boundary of $M_\rho - \mathbb{T}_\epsilon^{par}$ in an incompressible annulus and contains every toroidal boundary component of $M_\rho - \mathbb{T}_\epsilon^{par}$. The existence of a such a core is given in \cite{KuSh, McCullough}.  This definition naturally imparts a pared structure on any relative compact core whose paring locus consists of the tori and annuli that intersect $\partial \mathbb{T}_\epsilon^{par}$.  When $\rho$ is geometrically finite, we can construct a relative compact core $C$ by intersecting the convex core of $M_\rho$ with $M_\rho - \mathbb{T}_\epsilon^{par}$.  We will refer to this as \emph{the} relative compact core of $M_\rho$.  If $\rho \in MP(N,P)$ then the marking $f_\rho$ is homotopic to a pared homotopy equivalence from $(N,P)$ to the relative compact core of $M_\rho$.  So we can define a map $F: MP(N,P) \to A(N,P)$ by sending $(M_\rho, f_\rho)$ to the relative compact core of $M_\rho$ (still marked by $f_\rho$).  The map $F$ establishes a bijection between the components of $MP(N,P)$ and the elements of the set $A(N,P)$.  That is, $F(\rho_1) = F(\rho_2)$ if and only if $\rho_1$ and $\rho_2$ are in the same component of $MP(N,P)$.  

Let $B$ be the component of $MP(N,P)$ determined by $F^{-1}([(N',P'), h])$.  For $\rho \in B$, we have that $M_\rho$ is geometrically finite and minimally parabolic, and $f_\rho \circ h^{-1}$ is homotopic to a pared homeomorphism from $(N', P')$ to the relative compact core of $M_\rho$.  Using $f_\rho \circ h^{-1}$, we can mark each component of the conformal boundary of $M_\rho$ with a component of $\partial N' - P'$.  This gives us a map 
\[
\mathcal{AB} : B \to \mathcal{T}(\partial N' - P')
\] where $\mathcal{T}(\partial N' - P')$ denotes the Teichm\"uller space of $\partial N' - P'$.  Recall that the Teichm\"uller space of a disconnected surface is the product of the Teichm\"uller spaces of its components.  

\begin{thm}[Ahlfors \cite{Ahlfors}, Bers \cite{Bers}, Kra \cite{Kra}, Marden \cite{Marden open}, Maskit \cite{Maskit}, Sullivan \cite{Sullivan}, and Thurston \cite{Th2}]  When $\partial N - P$ is incompressible, the map $\mathcal{AB}$ is a homeomorphism on each component of $MP(N,P)$.  
\end{thm}

Throughout the rest of this paper, we will be primarily concerned with the case $N = S \times I$ where $S$ is a closed surface of genus at least two.  In this case, the previous theorem is known as Bers' simultaneous uniformization \cite{Bers2}.  The interior of $AH(N)$ (in this case $P = \emptyset$) is $MP(N)$ and is connected.  The Ahlfors-Bers map defines a homeomorphism 
\[
\mathcal{AB} : MP(N) \to \mathcal{T}(S) \times \mathcal{T}(S).
\]
Although we will continue to use the term minimally parabolic when $N = S \times I$, representations in $MP(N)$ contain no parabolics.  

Generally the Ahlfors-Bers parameterization does not extend over $AH(N,P)$; however, the recent resolution the Bers-Sullivan-Thurston Density Conjecture guarantees that every representation in $AH(N,P)$ can be expressed as the algebraic limit of geometrically finite and minimally parabolic representations.  We refer to this as the Density Theorem.  In the case that $(N,P) =  (S \times I, \emptyset)$, Brock, Canary, and Minsky obtained this result as Corollary 10.1 of the Ending Lamination Theorem \cite{BCM2}, using results of Ohshika \cite{Ohshika2} and Thurston \cite{Th4}.  

\begin{thm}[Brock-Canary-Minsky \cite{BCM2}] \label{density theorem} The closure of $MP(S \times I,\emptyset)$, as a subset of the character variety $X(N,P)$, is $AH(S\times I, \emptyset)$.  
\end{thm}

The more general result that $AH(N,P) = \overline{MP(N,P)}$ for any pared manifold $(N,P)$ follows from work of Brock-Bromberg \cite{BB}, Brock-Canary-Minsky \cite{BCM}, Bromberg \cite{Br4}, Bromberg-Souto \cite{BS}, Kim-Lecuire-Ohshika \cite{KLO}, Kleinedam-Souto \cite{KlSo}, Lecuire \cite{Lecuire}, Namazi-Souto \cite{NS}, Ohshika \cite{Ohshika}, and Thurston \cite{Th1}.   See \cite{survey} for a more complete discussion of the Density Theorem.

\section{Background Cone-Manifold Deformation Theory} \label{background cone}

\subsection{Geometrically Finite Cone-Manifolds}

Let $N$ be a compact $3$-manifold.  A hyperbolic cone-metric on the interior of $N$ with singular locus consisting of a link $\Sigma \subset int(N)$ is an incomplete hyperbolic metric (constant sectional curvature equal to $-1$) on the interior of $N-\Sigma$ whose metric completion determines a singular metric on $int(N)$ with singularities along $\Sigma$.  The link is totally geodesic, and in cylindrical coordinates around a component of $\Sigma$, the metric has the form
\[
dr^2 + \sinh^2(r) d\theta^2 + \cosh^2(r) dz^2
\] where $\theta$ is measured modulo $\alpha>0$.  We require $\alpha$ to be constant on each connected component of $\Sigma$, and we say $\alpha$ is the cone angle about that component of the singular locus.  See Section 1 of \cite{HK1} or Section 4 of \cite{B1} for more details.  When the cone angle on each component of $\Sigma$ is $\alpha = 2\pi$, this is equivalent to having a complete hyperbolic metric on the interior of $N$ (\textit{i.e.}, in the above definition, we require the metric on $int(N-\Sigma)$ to be complete in every end of $int(N-\Sigma)$ not associated to a component of $\Sigma$).   From now on, we will only consider cone-manifolds whose singular locus is connected. 

Let $M_\alpha$ be a hyperbolic cone-manifold homeomorphic to the interior of $N$ with cone angle $\alpha$ about $\Sigma$. We now define what it means for $M_\alpha$ to be a geometrically finite hyperbolic cone-manifold (see also Section 3 of \cite{B1}).  To do so, we first define a geometrically finite end.  Let $S$ be a closed surface of genus at least two, and let $Y = S \times [0,\infty)$ be a hyperbolic manifold with boundary $S \times \{0\}$.  That is, there is a smooth immersion $D: \tilde{Y} \to \mathbb{H}^3$ and representation $\rho: \pi_1(S) \to PSL(2, \mathbb{C})$ such that for all $x \in \tilde{Y}$ and $\gamma \in \pi_1(S)$,  $D(\gamma x) = \rho(\gamma) D(x)$. We say $D$ is the developing map for $Y$ and $\rho$ is the holonomy map.  We say $Y$ is a geometrically finite end if $D$ can be extended to a local homeomorphism $\tilde{S} \times [0, \infty] \to \mathbb{H}^3 \cup \hat{\mathbb{C}}$ such that $D(\tilde{S}\times \{\infty\}) \subset \hat{\mathbb{C}}$.  In this case, $S \times \{\infty\}$ inherits a conformal structure from the charts defined into $\hat{\mathbb{C}}$.  In fact, $Y$ has a projective structure at infinity since $PSL(2, \mathbb{C})$ acts by M\"obius transformations, although we will only use the fact that the transition maps are conformal.

Given a hyperbolic cone-manifold $M_\alpha$ with cone singularity $\Sigma$, we note that $M_\alpha - \Sigma$ has a (possibly incomplete) hyperbolic metric with no cone singularities.  Although one could consider hyperbolic cone-manifolds in greater generality, we have defined our cone-manifolds $M_\alpha$ to be homeomorphic to the interior of $N$ and hence topologically tame.   The ends of $M_\alpha - \Sigma$ (\textit{i.e.}, the complement of a compact core) are of three types (see p. 160 of \cite{B1}). There will be one end homeomorphic to $T^2 \times [0, \infty)$ associated to $\Sigma$, some number of ends associated to the rank-$2$ cusps of $M_\alpha$, also homeomorphic to $T^2 \times [0, \infty)$, and some number of ends homeomorphic to $S_i \times [0, \infty)$ associated to the higher genus surfaces $S_i$ in the boundary of the compact core.  We say $M_\alpha$ is geometrically finite if each of the ends not associated to a rank-$2$ cusp or to $\Sigma$ is geometrically finite.  We will not be considering hyperbolic cone-manifolds with rank-$1$ cusps.

We want to provide a meaningful way of interpreting a hyperbolic manifold with a rank-$2$ cusp as a hyperbolic cone-manifold with cone angle $\alpha = 0$ about the cone-singularity.  The convergence results below will allow us to do this more formally.  See also Section 3 of \cite{HK2} and Section 6 of \cite{B2}.

\begin{defn} A sequence of metric spaces with basepoints $\{(X_i, x_i)\}$ converges to $(X_\infty,x_\infty)$ geometrically if, for each $R>0$, $K>1$, there exists an open neighborhood $U_\infty$ of the
radius $R$ neighborhood of $x_\infty$ in $X_\infty$ and some $i_0$ such that for all $i > i_0$, there is a map $f_i : (U_\infty, x_\infty) \to (X_i, x_i)$ that is a $K$-biLipschitz diffeomorphism onto its image.   
\end{defn}

We say that a sequence $X_i \to X_\infty$ geometrically if there exist basepoints such that $(X_i, x_i) \to (X_\infty, x_\infty)$ geometrically.  For a more detailed discussion of geometric convergence in Kleinian group theory, see Chapter E of \cite{BP}, Chapter I of \cite{CEG}, or Chapter 8 of \cite{Kap}.  

The following is Theorem 6.11 of \cite{B2}, although a finite volume analogue was proven in Section 3 of \cite{HK2}.

\begin{thm}[Bromberg \cite{B2}] \label{geometric limit}
Let $\{ M_\alpha \}$ be a family of geometrically finite hyperbolic cone-manifolds defined for $\alpha \in (0, \alpha_0)$, with fixed conformal boundary, $\alpha_0 \leq 2\pi$, and suppose there is an embedded tubular neighborhood about the cone-singularity of radius $\geq \sinh^{-1}(\sqrt{2})$ in $M_\alpha$ for all $\alpha \in (0, \alpha_0)$.  Then 
\begin{enumerate}
\item as $\alpha \to 0$, the manifolds $M_\alpha$ converge geometrically to a complete hyperbolic manifold $M_0$ homeomorphic to the interior of $N-\Sigma$ with a rank-$2$ cusp in the end associated to $\Sigma$ and the same conformal boundary as $M_\alpha$. \\
\item as $\alpha \to \alpha_0$, the manifolds $M_\alpha$ converge geometrically to a hyperbolic cone-manifold $M_{\alpha_0}$ with cone angle $\alpha_0$ along $\Sigma$ and the same conformal boundary components as $M_\alpha$. 
\end{enumerate}
\end{thm}

This theorem serves two purposes.  First, we can interpret a manifold with a rank-$2$ cusp as a limit of a family of cone-manifolds. Second, a one-parameter family of cone-manifolds $M_\alpha$ with fixed conformal boundary, defined for some interval $[0,\alpha_0)$, can be extended to a one-parameter family defined over $[0, \alpha_0]$.

\subsection{Deformations of Cone-Manifolds}

Let $X$ denote the interior of $N -\Sigma$ and suppose $M_t$ is a one-parameter family of hyperbolic cone-manifolds homeomorphic to the interior of $N$ with singular locus $\Sigma$. By restricting the metric on $M_t$ to $X$, we obtain a one-parameter family of hyperbolic metrics on $X$. Up to precomposition by isotopies of $X$ and postcomposition by isometries of $\mathbb{H}^3$ this determines a one-parameter family of developing maps 
\[
D_t : \tilde{X} \to \mathbb{H}^3.
\]  We will assume $M_t$ is smooth in the sense that $D_t$ is a smooth one-parameter family of diffeomorphisms.  Recall that $D_t$ being a developing map means that for each $t$ there is a corresponding holonomy representation $\rho_t : \pi_1(X) \to PSL(2, \mathbb{C})$ such that for any $x \in \tilde{X}$ and $\beta \in \pi_1(X)$, $D_t(\beta x) = \rho_t(\beta)D_t(x)$ 

At each fixed $t$, we obtain a vector field $v$ on $\tilde{X}$ defined by setting $v(x)$ to be the pullback (via $D_t$) of the tangent vector to the path $t \mapsto D_t(x)$.   Let $\tilde{E} = \tilde{X} \times sl_2(\mathbb{C})$ be the flat bundle of Killing fields on $\tilde{X}$. Recall that a Killing field is a vector field whose associated flow $\phi_t: \tilde{X} \to \tilde{X}$ is an isometry for all sufficiently small $t$.  Killing fields on $\mathbb{H}^3$ are parameterized by $sl_2(\mathbb{C})$ by taking the derivative $\left. \frac{d}{dt}\right|_{t=0} \phi_t$.  In general, unless the deformation is trivial, $v$ will not be a Killing field on $\tilde{X}$; however, we can associate to $v$ the Killing field, or equivalently the section $s_v : \tilde{X} \to \tilde{E}$ of the bundle $\tilde{E}$, that best approximates $v$ at $x$. To define $s_v$ we first need to analyze the natural complex structure $\tilde{E}$ inherits from $sl_2(\mathbb{C})$. 

At each point $x \in \tilde{X}$, the fiber of $\tilde{E}$ decomposes as a direct sum $\tilde{\mathcal{P}} \oplus \tilde{\mathcal{K}}$ where $\tilde{\mathcal{P}}$ consists of the infinitesimal isometries which are pure translations and $\tilde{\mathcal{K}}$ consists of the infinitesimal isometries that are pure rotations at $x$.  (This decomposition is defined to be orthogonal with respect to the metric we are going to put on the fibers of $\tilde{E}$.  Note this decomposition is not preserved by the flat connection on $\tilde{E}$. See p. 13-19 of \cite{HK1}.) The sub-bundle $\tilde{\mathcal{P}}$ is naturally identified with $T \tilde{X}$, and using the complex structure of $\tilde{E}$ one sees that $\tilde{\mathcal{K}} = i \tilde{\mathcal{P}}$, so we can decompose $\tilde{E} \cong T\tilde{X} + i T\tilde{X}$.  With this notation, we can define the canonical lift of the vector field $v$ to be the section  $s_v: \tilde{X} \to \tilde{E}$ given by 
\[
s_v (x)= v(x) - i curl(v)(x)
\] Here we are using twice the usual curl, which is normally defined by $curl(v) = \frac{1}{2} (*\hat{d}\hat{v})$ where $\hat{d}$ is exterior differentiation and $\hat{v}$ is the $1$-form dual to $v$.   Under the identification of Killing fields with $sl_2(\mathbb{C})$, this curl operator on vector fields acts like multiplication by $i$ on sections of $\tilde{E}$. Hence, $s_v$ is the section whose real part agrees with $v$ and such that the real part of $curl(s)$ agrees with $curl(v)$.  See \cite{HK1} or \cite{B1} for more details.

\subsection{$E$-valued Differential Forms}

Now we want to view the canonical lift $s_v$ as an $\tilde{E}$-valued $0$-form and obtain a $1$-form via exterior differentiation. Recall, an $\tilde{E}$-valued $k$-form (on $\tilde{X}$) is a section of the bundle $\wedge^k T\tilde{X}^* \otimes \tilde{E} \to \tilde{X}$, and the exterior derivative 
$
d: \wedge^k T\tilde{X}^* \otimes \tilde{E}  \to \wedge^{k+1} T\tilde{X}^* \otimes \tilde{E}
$ is defined using the flat connection on $\tilde{E}$.  Since a form $\omega \in \wedge^k T\tilde{X}^* \otimes \tilde{E}$ can be thought of as having values in $\tilde{E}$,  the complex structure of $\tilde{E}$ gives us real and imaginary parts $D$ and $T$ of the exterior derivative $d$.  That is, if $d\omega = \eta_1 + i \eta_2$ then $D \omega = \eta_1$ and $T\omega = \eta_2$.

Define $E$ to be the quotient of $\tilde{E}$ by $\pi_1(X)$ where $\pi_1(X)$ acts on $\tilde{X}$ by covering transformations and on $sl_2(\mathbb{C})$ by the adjoint representation.  This gives $E \to X$ the structure of a flat bundle, and the action of $\pi_1(X)$ also preserves the complex structure of $\tilde{E}$.  We define $E$-valued $k$-forms similarly to $\tilde{E}$-valued forms, and using the exterior derivative, we can define the cohomology groups $H^k(X;E)$ to be the closed forms modulo the exact forms.  We will use the notation $\Omega^k(X;E)$ to denote the set of closed forms.

If $s_v$ is the canonical lift of a vector field $v$ as defined above, then $ds_v$ is an equivariant closed $1$-form and thus descends to an element $\omega \in \Omega^1(X;E)$.  Moreover, the cohomology class in $H^1(X;E)$ defined by $\omega$ is independent of the choice of developing maps $D_t$.     This is proven in Section 2 of \cite{HK1} (see also Section 2 of \cite{B1}).  As Hodgson and Kerckhoff explain on p. 375 of \cite{HK2}, altering $D_t$ by postcomposing isometries of $\mathbb{H}^3$ has no effect on $\omega$ and precomposing isotopies of $\tilde{X}$ only has the effect of changing $\omega$ by an exact form. Thus the cohomology class determined by $\omega$ is well-defined by the one-parameter family of metrics on $X$.

Conversely, given a cohomology class $\omega_{t_0}$, we can describe the infinitesimal change in the metric on $X$ in the following sense.  Given any $\gamma \in \pi_1(X)$, we can compute 
\begin{align} \label{integrate}
\int_\gamma \omega_{t_0} = \left. \frac{d}{dt} \rho_t(\gamma)\rho_0(\gamma)^{-1} \right|_{t = t_0}.
\end{align}  We first choose a closed $1$-form representing the cohomology class $\omega_{t_0}$, also denoted $\omega_{t_0}$, and integrate the form along $\gamma$.  The integral only depends on the homotopy class of $\gamma$ and gives us an element of $sl_2(\mathbb{C})$.  See p. 12-13 of \cite{HK1}. 

\subsection{Harmonic Forms}

We now define an $L^2$-norm on the set of closed $E$-valued $k$-forms, which will allow us to pick a nice closed form to represent each cohomology class.   This will be the Hodge representative we define below.  

  Given $x \in \tilde{X}$, define an inner product on the fiber of $\tilde{E}$ over $x$, identified with $sl_2(\mathbb{C})$ by 
  \[
  \left< v, w \right>_x = \left< v(x), w(x) \right>_x + \left< iv(x), iw(x) \right>_x
  \] where both of the inner products on the right are the inner product on $T_x \tilde{X}$.  Recall $v, w, iv, iw \in sl_2(\mathbb{C})$ correspond to Killing fields in $T\tilde{X}$ and $v(x), w(x), iv(x), iw(x)$ are the vectors in $T_x \tilde{X}$ determined by these Killings fields. One can check that this inner product descends to an inner product on the fibers of $E$.

Using this inner product, we have an isomorphism $\sharp : E \to E^*$.  Recall that if $\omega$ is an $E$-valued $k$-form, then $\omega = \alpha \otimes s$ for some real $k$-form $\alpha$, so we define $\sharp(\omega) = \alpha \otimes \sharp(s)$.   The Hodge star operator can be defined for $E$-valued forms by $*\omega = *\alpha \otimes s$.  Then define an inner product on $E$-valued $k$-forms by
 \[
 (\omega_1, \omega_2) = \int_X  \omega_1 \wedge (\sharp * \omega_2) .
 \] 
The square of the norm of $\omega$ is $(\omega, \omega)$, and we say $\omega$ is in $L^2$ if $(\omega, \omega)$ is finite.

\textit{Remark.}  This is equivalent to defining $(\omega_1, \omega_2) = \int_X \left< \omega_1, \omega_2 \right>_x$ where the pointwise inner product of two $E$-valued forms $\omega_1 = \alpha_1 \otimes s_1$ and $\omega_2 = \alpha_2 \otimes s_2$ is the product of $\left< s_1 , s_2 \right>_x$ defined above and the standard inner product on real-valued forms $\left< \alpha_1, \alpha_2 \right>_x$.  
 
Define adjoints $\delta$, $D^*$, and $T^*$ to $d, D, T$ by 
\begin{align*}
 \delta &=   (-1)^{n(k+1) + 1} * (D-T) * \\
 D^* &= (-1)^{n(k+1) + 1} * D * \\
T^* &= (-1)^{n(k+1) + 1} * T *.
\end{align*}
One can check that if $\omega_1$ is a $k+1$-form and $\omega_2$ is a $k$-form then $(\delta \omega_1 , \omega_2) = (\omega_1, d\omega_2)$, whenever the inner product is finite.  See p. 13-14 of \cite{Bromberg Thesis} or Section 1 of \cite{HK1} for a broader development of this theory.

A $k$-form is closed if $d\omega = 0$ and co-closed when $\delta \omega =0$.  A form $\omega$ is harmonic of $\Delta \omega = (d \delta + \delta d) \omega = 0$.  When $X$ is a closed manifold, the Hodge theorem says that any cohomology class in $H^k(X; E)$ can be represented by a closed, co-closed (and hence harmonic) form.  While we are not dealing with closed manifolds, we will be able to use similar results of Hodgson, Kerckhoff, and Bromberg to find nice representatives for the cohomology classes in which we are interested.  See Section \ref{hodge subsection} below.

\subsection{Standard Form} \label{standard subsection}

Hodgson and Kerckhoff calculated the effects of two particular $E$-valued $1$-forms $\omega_m$ and $\omega_l$ in a neighborhood of $\Sigma$ (p. 31-33, \cite{HK1}).  We now review the definitions of $\omega_m$ and $\omega_l$.  Using these forms, some of the results in Section 3 of \cite{HK1} will allow us to put an arbitrary $E$-valued $1$-form into a standard form within a cohomology class of $H^1(X;E)$. 

  Let $M_\alpha$ be a cone-manifold, and let $U$ be a metric collar neighborhood of $\Sigma$ in $M_\alpha$.  We give $U$ the cylindrical coordinates $(r, \theta, z)$ where $r$ is the distance from $\Sigma$.  Recall $X = M_\alpha - \Sigma$, so $U \cap X$ is the set of points in $U$ with $r>0$.  Using the complex structure of $E$, the real and imaginary parts of an $E$-valued $1$-form can be identified with $TX$-valued $1$-forms, or in other words, sections of $Hom(TX, TX)$.  Thus we will define $\omega_m$ and $\omega_l$ as complex valued sections of $Hom(TX, TX)$.  With respect to the basis $\left(\frac{\partial}{\partial r},  \frac{1}{\sinh(r)} \frac{\partial }{\partial\theta}, \frac{1}{\cosh(r)}\frac{\partial}{\partial z} \right)$ on $TX$, we can define the forms $\omega_m$ and $\omega_l$ at any point $(r, \theta, z)$ of $U \cap X$ by the following matrices
 \[
\omega_m = \begin{pmatrix} \frac{-1}{\cosh^2(r) \sinh^2(r)} & 0 & 0 \\ 0 & \frac{1}{\sinh^{2}(r)} & \frac{-i}{\cosh(r) \sinh(r)} \\ 0  &  \frac{-i}{\cosh(r) \sinh(r)} &  \frac{-1}{\cosh^2(r)} \end{pmatrix}
\]
\[
\omega_l =  \begin{pmatrix} \frac{-1}{\cosh^2(r)} & 0 & 0 \\ 0 & -1 & \frac{-i\sinh(r)}{\cosh(r)} \\ 0  &  \frac{-i\sinh(r)}{\cosh(r)} &  \frac{\cosh^2(r) + 1}{\cosh^2(r)} \end{pmatrix}.
\]
If $\alpha \to 0$, the neighborhood $U$ limits to a rank-$2$ cusp.  We define $\omega_m$ and $\omega_l$ in a rank-$2$ cusp as limits of the $1$-forms defined above as $r \to \infty$, although to make this precise, we would need to describe how the coordinates on $U$ limit to coordinates on a rank-$2$ cusp (see Section 3.6 of \cite{MagidThesis} for this description).

\begin{defn} A closed $E$-valued $1$-form $\omega$ is \textit{in standard form} if there is a neighborhood $U_1$ of the singular locus and neighborhoods $U_2, \ldots, U_n$ of each rank-$2$ cusps such that in $U_i$, $\omega$ equals a complex linear combination of $\omega_m$ and $\omega_l$.  
\end{defn}

Note that the complex coefficients of $\omega_m$ and $\omega_l$ will generally be different for each $U_i$.  The following lemma (Lemma 3.3 of \cite{HK1}) shows that every cohomology class can put into standard form.

\begin{lem}[Hodgson-Kerckhoff \cite{HK1}]  \label{standard}
Given any closed $E$-valued $1$-form $\phi$, there is a cohomologous form $\omega_0$ that is in standard form. \end{lem}

Note that the form $\omega_0$ is not unique since there is no control over $\omega_0$ outside the union of the neighborhoods $U_i$.  Nevertheless, we will use a form $\omega_0$ that is in standard form to represent our cohomology class when we want to understand the infinitesimal deformation near the cone singularities and near the cusps.  

\subsection{Hodge Forms} \label{hodge subsection}

In the previous section, we showed that for any cohomology class in $H^1(X;E)$, there is a closed $1$-form representing that class that is in standard form in some neighborhood of $\Sigma$ and some neighborhood of each rank-$2$ cusp.  The following Hodge theorem for cone-manifolds shows that there is a harmonic representative as well, and that its difference from the standard form is bounded. 

\begin{defn} A $1$-form $\omega \in \Omega^1(X;E)$ is a Hodge form if $\omega$ is closed, co-closed, and locally $\omega$ can be expressed as $ds$ where $s$ is the canonical lift of a divergence-free, harmonic vector field. 
\end{defn}

Before stating Theorem 4.3 of \cite{B1} which generalizes Theorem 2.7 of \cite{HK1}, we need to define what it means for a $1$-form to be conformal at infinity.   By Lemma 3.2 of \cite{B1}, there is an isomorphism $\Pi_* : H^1 (X;E) \to H^1(\partial_c X; E_\infty)$ where $\partial_c X$ is the conformal boundary of $X$ and $E_\infty$ is the bundle of germs of projective vector fields on $\partial_c X$. A cohomology class $\omega_\infty \in H^1(\partial_c X; E_\infty)$ is conformal if it can be expressed as $ds_\infty$ where $s_\infty$ is the canonical lift of an automorphic, conformal vector field on $\partial_c X$.  A cohomology class $\omega \in H^1 (X;E)$ is conformal at infinity if $\Pi_* (\omega)$ is conformal.  

We will only be concerned with $1$-forms on $X$ that arise from one-parameter deformations of hyperbolic cone manifolds $M_t$ whose conformal boundary is fixed throughout the deformation.  These $1$-forms will be conformal at infinity with respect to the definition given above. 

\begin{thm}[Bromberg \cite{B1}, Hodgson-Kerckhoff \cite{HK1}]   \label{hodge}
Let $M$ be a geometrically finite hyperbolic cone-manifold, and let $\omega_0$ be an $E$-valued $1$-form  on $X = M - \Sigma$ that is conformal at infinity and in standard form in a neighborhood $U$ of $\Sigma$.  Then there exists a unique Hodge form $\omega$ such that the following holds:\\
(1) $\omega$ is cohomologous to $\omega_0$, \\
(2) there exists an $L^2$ section $s$ of $E$ such that $ds = \omega_0 - \omega$ \\
(3) $\omega_0 - \omega$ has finite $L^2$ norm on the complement of $U$. 
\end{thm}  

Using their analysis of Hodge forms and infinitesimal deformations, Bromberg generalized the local rigidity results of Hodgson and Kerckhoff in \cite{HK1} to show that the (possibly incomplete) hyperbolic metric on the interior of a cone-manifold $M_\alpha$ is completely determined by the cone angle $\alpha$ and the conformal boundary components associated to each of the geometrically finite ends of $M_\alpha$.  See Theorem 5.8 of \cite{B1} which is restated as Theorem 1.1 of \cite{B2}.  We will use the following consequence of Bromberg's result.

\begin{thm}[Bromberg \cite{B1}] \label{open}  Let $M_{\alpha_0}$ be a geometrically finite hyperbolic cone-manifold with cone angle $\alpha_0 \in [0, 2\pi]$ about the cone singularity $\Sigma$.  Suppose there is an embedded tubular neighborhood about $\Sigma$ in $M_{\alpha_0}$ of radius $\geq \sinh^{-1}(\sqrt{2})$.  Then there exists an open neighborhood $W$ of $\alpha_0$ in $[0, 2\pi]$ such that the one-parameter family $M_\alpha$, defined by varying the cone angle and keeping the conformal boundary of $M_{\alpha_0}$ fixed, is defined for all $\alpha \in W$. 
\end{thm}

\subsection{Complex Length} \label{complex length subsection}

 Let $X$ be the interior of $N - \Sigma$ and $\rho: \pi_1(X) \to PSL(2, \mathbb{C})$.  If $\gamma \in \pi_1(X)$, then the complex length of $\gamma$, denoted $\mathcal{L} = l+i\theta$ or $\mathcal{L}(\rho(\gamma))$, is defined by the formula 
\[
tr^2(\rho(\gamma)) = 4 \cosh^2 \left( \frac{\mathcal{L}}{2} \right) 
\]  
and the normalizations $l \geq 0$ and $\theta \in (-\pi, \pi]$.  If $\rho(\gamma)$ is a loxodromic element, then $l$ is the length of the geodesic representative of $\gamma$ in $X$ (equivalently, the translation length of $\rho(\gamma)$ along its axis in $\mathbb{H}^3$), and $\theta$ gives the amount $\rho(\gamma)$ twists along its axis. If $\rho(\gamma)$ is parabolic, then the complex length is zero.  

If $M_\alpha$ is a cone-manifold homeomorphic to the interior of $N$ with cone singularity $\Sigma$, and $U$ is a tubular neighborhood of the cone-singularity, then $\partial U$ has a well-defined meridian $\beta$.  This is the homotopy class of a curve on $\partial U$ that bounds a disk (with a cone-point) in $M_\alpha$.  When the cone angle is $\alpha \in (0, 2\pi)$, then the meridian will be sent to an elliptic element that rotates by $\alpha$ about its axis.  In this case, we say the meridian has (purely imaginary) complex length $i \alpha$.  In our situation, when $\alpha= 0$ we will have $\rho(\beta)$ be parabolic, but when $\alpha = 2\pi$, $\rho(\beta)$ will be the identity.   

If $\alpha \in (0, 2\pi]$ and $U$ is a metric collar neighborhood, then the torus $\partial U$ inherits a Euclidean metric as a subset of $M_\alpha$, so we can pick the shortest longitude $\lambda$ on $\partial U$ (by a longitude, we mean any curve that intersects the meridian once) and normalize the complex length of $\lambda$ to be $l + i\theta$ for some $\theta \in \left(-\frac{\alpha}{2}, \frac{\alpha}{2} \right]$.  Then any other longitude will have complex length $l +i\theta + im \alpha$ for some $m \in \mathbb{Z}$.  The only time the choice of $\lambda$ is not well-defined is when there are two shortest longitudes on $\partial U$ in which case we pick one and assign it the complex length $l + i \frac{\alpha}{2}$ by convention.  
We say that the complex length of the cone axis $\Sigma$ is the complex length of any longitude since these are all homotopic to $\Sigma$. Since the complex length of $\Sigma$ is only well-defined up to the addition of multiples of $i\alpha$, we work with the complex length of a particular longitude instead.

\section{The Drilling and Filling Theorems} \label{filling theorem section}

\subsection{Drilling
}
If $M$ is a geometrically finite hyperbolic manifold and $\gamma$ is a disjoint collection of simple closed geodesics in $M$, then Kojima showed, using an argument he attributes to Kerckhoff, that $M-\gamma$ admits a unique complete hyperbolic metric such that the natural inclusion of $(M-\gamma) \subset M$ extends to a conformal map between the conformal boundaries of $M-\gamma$ and $M$ (see Thoerem 1.2.1 of \cite{Kojima}).  We call this process drilling (\textit{i.e.}, finding a new metric on $M-\gamma$).  If the curves in $\gamma$ are sufficiently short, the Brock-Bromberg drilling theorem bounds the difference between the original metric on the complement of $\gamma$ in $M$ and the new complete metric on $M-\gamma$ \cite{BB}. That is, they show the metrics are close (in a biLipschitz sense) on the complement of a neighborhood of the drilled curves.  Their work makes use of the cone-manifold machinery developed by Hodgson, Kerckhoff, and Bromberg \cite{HK1, HK2, B1, B2} that we outlined in the previous section.

 More precisely, if $M$ is a geometrically finite hyperbolic manifold without rank-$1$ cusps, and $\gamma_1, \ldots, \gamma_n$ is a disjoint collection of simple closed curves in $M$, let $\hat{M}$ denote the geometrically finite hyperbolic manifold homeomorphic that is homeomorphic $M- \cup \gamma_i$ and has the same conformal boundary components.  We say $\hat{M}$ is the $\cup \gamma_i$-drilling of $M$. The following is Theorem 6.2 of \cite{BB}. 

\begin{thm}[Brock-Bromberg \cite{BB}] \label{drilling} Given any $J > 1$, $\epsilon_3 \geq \epsilon>0$, there exists some $l_0 > 0$ such that the following holds:  if $M$ is a geometrically finite manifold with no rank-$1$ cusps, and $\gamma_1, \ldots, \gamma_n$ is a collection of geodesics in $M$ with 
\[
\sum_{i=1}^n l(\gamma_i) < l_0,
\] then there exists a $J$-biLipschitz diffeomorphism
\[
\phi : \hat{M} - \cup_{i=1}^n \mathbb{T}_\epsilon (T_i) \to M - \cup_{i=1}^n \mathbb{T}_\epsilon (\gamma_i)
\]   where $\hat{M}$ is the $\cup \gamma_i$-drilling of $M$, and $T_i$ is the cusp corresponding to the drilling of $\gamma_i$.  
\end{thm}

\textit{Remark.}  In the drilling theorem, Brock and Bromberg also conclude that the map $\phi$ in the drilling theorem is level-preserving on cusps in the following sense. Suppose $T_{\hat{M}}$ is a rank-$2$ cusp in $\hat{M} - \cup_{i=1}^n \mathbb{T}_\epsilon (T_i)$.  Then the drilling map $\phi$ sends $T_{\hat{M}}$ to a cusp $T_M$ in $M$ in such a way that for any $0< \epsilon' \leq \epsilon$, $\phi( \partial \mathbb{T}_{\epsilon'} (T_{\hat{M}})) = \partial \mathbb{T}_{\epsilon'} (T_{M})$. See Theorem 6.12 of \cite{BB} (in particular, see Lemma 6.17 of \cite{BB} which was used to prove 6.12).

\subsection{Filling}

The filling theorem provides an inverse construction.  A geometrically finite hyperbolic manifold with a rank-$2$ cusp is homeomorphic to the interior of a compact manifold with a torus boundary component.  One can Dehn-fill this compactification along any boundary slope and attempt to hyperbolize the interior of the filled manifold.  Under certain conditions on the boundary slope, the filling theorem provides a way of doing this hyperbolic Dehn-filling while preserving the conformal boundary components.  Moreover, we obtain estimates on the complex length of the core curve of the filling torus in the new metric.

Now let $\hat{M}$ be a geometrically finite hyperbolic manifold with $n$ rank-$2$ cusps.  We want to describe a way of filling in these cusps to obtain a hyperbolic manifold $M$ with the same conformal boundary but no rank-$2$ cusps.  Although methods developed by Hodgson, Kerckhoff, and Purcell \cite{Purcell} can be used to fill multiple cusps simultaneously, this introduces some unnecessary complications. We will proceed by describing the filling theorem for one cusp, which up to renumbering we can assume is the first cusp.  Then we derive the multiple cusp case by filling one cusp at a time (see Corollary \ref{multiple fillings}).
  
Let $\hat{N}$ be a compact $3$-manifold with interior homeomorphic to $\hat{M}$.  On the first torus boundary component of $\hat{N}$ fix a slope $\beta$.  Let $N$ be the manifold obtained by Dehn-filling $\hat{N}$ along $\beta$.  If possible, we hyperbolize the interior of $N$ to obtain a hyperbolic manifold $M$ with the same conformal boundary as $\hat{M}$ and one fewer cusp.  If it exists, we call $M$ the $\beta$-filling of $\hat{M}$.  Let $\gamma$ be the geodesic representative in $M$ of the core curve of the solid torus used to Dehn-fill $\hat{N}$.  

The following theorem gives sufficient conditions for the $\beta$-filling of $\hat{M}$ to exist, and when these conditions are satisfied, gives information about the complex length of the geodesic $\gamma$ in $M$.  Before stating the theorem, we define the normalized length and normalized twist of the slope $\beta$ used in the filling; henceforth called the meridian.

The boundary of the cusp $T$ we are filling, $\partial \mathbb{T}_{\epsilon_3}(T)$, inherits a Euclidean metric.  Let $\mu$ be a geodesic in the homotopy class of $\beta$ on this flat torus. Let $m$ denote the length of $\mu$. Fix a point $x \in \mu$, and let $\nu$ be a geodesic ray perpendicular to $\mu$ at $x$.  Let $y$ denote the next point on $\nu$ where $\nu$ meets $\mu$ (after $x$).  Orient $\mu$ so that if $\vec{\mu}$ and $\vec{\nu}$ denote the tangent vectors to $\mu$ and $\nu$ at $x$ then $\vec{\mu} \times \vec{\nu}$ points into the cusp $T$.  Define $b$ to be the value in $\left(-\frac{m}{2}, \frac{m}{2}\right]$ such that $|b|$ is the distance between $x$ and $y$, and after orienting $\mu$ as described, the sign of $b$ gives the orientation of the shortest path beginning at $x$ and ending at $y$ realizing this distance. See Figure \ref{torus4}.  When $y$ is exactly half-way around $\mu$ from $x$ then there are two shortest paths.  In this case, we choose the positively oriented one so $b = \frac{m}{2}$. The value $b$ is the twist associated to the flat structure on $\partial \mathbb{T}_{\epsilon_3}(T)$ with meridian $\beta$.

\begin{figure}[htbp]
\begin{center}
\psfrag{y}[][]{\small $\nu$}
\psfrag{b}[][]{\small $b$}
\psfrag{x}[][]{\small $\lambda$}
\psfrag{h}[][]{\small $\mu$}
\psfrag{z}[][]{\small $x$}
\psfrag{f}[][]{\small $y$}
\includegraphics[width=2in]{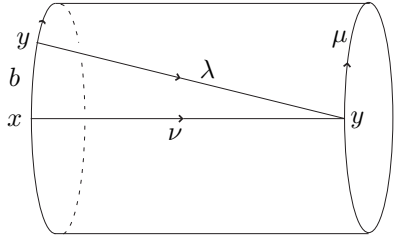}
\caption[Torus]{\label{torus4} A flat torus obtained by identifying the boundary circles of a cylinder with a twist of $b$.  In this picture, $b$ is positive.}
\end{center}
\end{figure}

\begin{defn} 
The normalized length $L$ of $\beta$ is the length of $\mu$ divided by the square root of the area of $\partial \mathbb{T}_{\epsilon_3}(T)$:
\[
L = \frac{m}{\sqrt{Area(\partial \mathbb{T}_{\epsilon_3}(T))}}.
\]
\end{defn}   

\begin{defn} \label{defn_norm_twist} The normalized twist of the flat metric on $\partial \mathbb{T}_{\epsilon_3}(T)$ with meridian $\beta$ is the ratio $\frac{b}{m}$.  We let $A^2 = \frac{m}{b}$ denote the reciprocal of the normalized twist.  
\end{defn}

\bigskip \noindent \textit{Remark.}  The normalized length and normalized twist are well-defined invariants of the cusp $T$ with meridian $\beta$.  That is, they depend only on the point in $\mathcal{T}(T^2)/\mathbb{Z}$ determined by the flat metric on $\partial \mathbb{T}_{\epsilon_3}(T)$. Here we consider the quotient of the Teichm\"uller space of the torus by the subgroup generated by Dehn twists about $\beta$. In particular, we could define the normalized length or twist using any cross section of the cusp $\partial \mathbb{T}_{\epsilon}(T)$, where $\epsilon_3 \geq \epsilon>0$. 

Also note that in spite of the ``square" notation, the quantity $A^2 \in (-2, 2]$ could be negative, and we make no use of any quantity $A$ in this paper.  We use this notation to emphasize that $A^2$ is a counterpart to $L^2$ in the following sense.  

Assume $A^2 \neq 2$ so there is a unique shortest longitude $\lambda$ on $T^2 \times \{0\}$.  
If we conjugate $\pi_1(\hat{M})$ in $PSL(2, \mathbb{C})$ so that the isometry corresponding to $\lambda$ is $\begin{pmatrix} 1 & 2 \\ 0 & 1 \end{pmatrix}$, and the isometry corresponding to the meridian $\beta$ is $\begin{pmatrix} 1 & w \\ 0 & 1 \end{pmatrix}$ for some $w \in \mathbb{C}$ with $Im(w) > 0$, then 
\[
L^2 = \frac{|w|^2}{2Im(w)} \quad \text{and} \quad A^2 = \frac{|w|^2}{2Re(w)}.
\]

Now that we have defined the quantities $L^2$ and $A^2$, we are ready to state the filling theorem.

\noindent \textbf{Theorem \ref{filling}.} \emph{Let $J>1$ and $\epsilon_3 \geq \epsilon > 0$.  There is some $K \geq 8(2\pi)^2$ such that the following holds:  suppose $\hat{M}$ is a geometrically finite hyperbolic $3$-manifold with no rank-$1$ cusps, $T$ is a rank-$2$ cusp in $\hat{M}$, and $\beta$ is a slope on $T$ such that the normalized length of $\beta$ is at least $K$ (i.e., $L^2 \geq K^2$), then   \\
$(i)$ the $\beta$-filling of $\hat{M}$, which we call $M$, exists;  \\
$(ii)$  the real part of the complex length $\mathcal{L} = l+i\theta$ of the core curve of the filling torus $\gamma$ in $M$ is approximately $\frac{2\pi}{L^2}$ with error bounded by
\[
\left| l - \frac{2\pi}{L^2} \right| \leq \frac{8(2\pi)^3}{L^4 - 16(2\pi)^4};
\] 
\\
$(iii)$  in particular, the length of $\gamma$ is bounded above by $\frac{2\pi}{L^2 - 4(2\pi)^2}$; \\
$(iv)$  there exists a $J$-biLipschitz diffeomorphism 
\[
\phi : \hat{M} - \mathbb{T}_\epsilon (T) \to M - \mathbb{T}_\epsilon (\gamma).
\]
$(v)$  If, in addition to $L^2 \geq K^2$, we have $|A^2| \geq 3$, then the imaginary part of the complex length $\mathcal{L} = l+i\theta$ of the core curve of the filling torus $\gamma$ in $M$ (chosen so $\theta \in (-\pi, \pi]$) is approximately $\frac{2\pi}{A^2}$ with error bounded by
\[
 \left| \theta - \frac{2\pi}{A^2} \right| \leq \frac{5(2\pi)^3}{(L^2 - 4(2\pi)^2)^2}.
\] 
}

\bigskip

\noindent \textbf{Outline of the Proof and Prior Results}
\bigskip

As the proof is rather technical and spans multiple subsections, we begin with an outline.  Many of the results of the filling theorem follow directly from results of Brock, Bromberg, Hodgson, and Kerckhoff, so first we delineate these contributions. Part $(i)$ was shown by Bromberg in \cite{B2}.  This was a  generalization of the work of Hodgson and Kerckhoff in \cite{HK2} to geometrically finite manifolds.  We intend to revisit these arguments in order to reprove $(i)$ and draw the conclusions stated in $(ii)$ and $(v)$.  Part $(iii)$ follows from $(ii)$, although both are proven in Lemma \ref{length lemma}. Part $(iv)$ follows from parts $(i)$, $(iii)$ and the drilling theorem. That is, once we show that the filling exists and $\gamma$ is sufficiently short in $M$, we can then apply the drilling theorem to $M$ to recover $\hat{M}$ and get the $J$-biLipschitz map from the drilling theorem.
  
We view the geometrically finite manifold $\hat{M}$ with $n$ cusps as a geometrically finite hyperbolic cone-manifold with $n-1$ rank-$2$ cusps and cone singularity $\Sigma = \gamma$ with cone angle $\alpha = 0$ about $\Sigma$. Part $(i)$ amounts to showing that we can increase the cone angle from $0$ to $2\pi$.  We will use the parameterization $t = \alpha^2$.  
By Theorem \ref{open} the one-parameter family is defined in some interval $[0, t')$.  At any $t \in [0, t')$, Proposition \ref{derivative_estimates} provides estimates on the derivative of the complex length of any peripheral curve. The bounds here depend on the radius of an embedded tube about the cone-singularity.   Provided $L^2$ is sufficiently large, Lemma \ref{radius lemma} shows that for all $t \in [0, t')$, the radius remains bounded below; therefore, we can extend the one-parameter family to $M_{t'}$ using Theorem \ref{geometric limit}, and in Lemma \ref{radius lemma 2} we show that if $L^2 \geq 8(2\pi)^2$ then we can continue the deformation to $t=(2\pi)^2$.  Having defined the one-parameter family $M_t$ for all $t \in [0, (2\pi)^2]$, we can integrate the estimates from Proposition \ref{derivative_estimates} to prove Lemma \ref{length lemma}.  This gives us the estimate on the real part of the length of $\gamma$ we claimed in part $(ii)$.  

To get the estimate on the imaginary part of the complex length of $\gamma$ in $M$ that we claimed in part $(v)$, we consider a longitude on $\partial U_1$ homotopic to $\gamma$.  Actually, it is more convenient to work with a metric collar neighborhood $V_t$  of $\gamma$ in $M_t$.  The torus $\partial V_t$ inherits a flat metric and we can define the twist $b(t)$ and the normalized twist $\frac{b(t)}{m(t)}$ associated to $\partial V_t$ in the same way that we defined the twist $b$ and normalized twist $\frac{b}{m}$ for $\partial V_0$ (the boundary of the cusp we are filling in $\hat{M}$).   Since $|A^2| > 2$, there is a unique shortest longitude $\lambda$ on $\partial V_0$. For any $t$ when $\lambda$ is shortest, the imaginary part of the complex length which we  denote by $\theta(\lambda, t)$ is well-defined as a real number (see Section \ref{complex length subsection}). In Lemma \ref{boundv}, we define the quantity $v(\lambda', t) = \frac{\theta(\lambda', t)}{\alpha}$ for any longitude $\lambda'$ and bound $\frac{dv}{d\alpha}$ at any fixed time $t$.  In Lemma \ref{longitude}, we show $\lim_{t\to 0} v(\lambda, t) = \frac{1}{A^2}$.  In Lemmas \ref{theta lemma} and \ref{part 5} we use this limiting value of $v(\lambda, t)$, the bound on the derivative of $v$ found in Lemma \ref{boundv}, and the hypothesis $|A^2| \geq 3$ to show that $\lambda$ remains the shortest longitude on $\partial V_t$ for all $t \in [0, (2\pi)^2]$ and obtain the estimate on $\theta =\theta(\lambda, 2\pi)$ in part $(v)$ of Theorem \ref{filling}.

\subsection{Existence of the Deformation and Derivative Estimates}

We now begin the proof of Theorem \ref{filling}.  In this section, our goal is to show that the one-parameter family $M_t$, defined by setting $M_0 = \hat{M}$ and increasing the cone angle with with the parameterization $t = \alpha^2$, can be defined for all $t \in [0, (2\pi)^2]$. In Lemma \ref{radius lemma 2}, we show that if $L^2 \geq 8(2\pi)^2$, then we can do this.

\begin{proof}  Let $M_0 = \hat{M}$, and define a one-parameter family of hyperbolic cone-manifolds $M_t$ by increasing the cone angle about $\gamma$ according to the parameterization $t = \alpha^2$ and keeping the conformal boundary fixed.  This is well defined for some interval $[0, t')$ by the local rigidity results of Bromberg (see Theorem 5.8 of \cite{B1}, restated as Theorem \ref{open} in the previous section).  At any $t\in[0,t')$ we can estimate the derivative of the complex length of any curve on the boundary of a neighborhood of the cone singularity. 

\begin{prop} \label{derivative_estimates}  Suppose the one-parameter family of cone manifolds $M_t$ has been defined for $t \in [0, t')$, and suppose there is an embedded tube $U_1$ of radius $R$ about $\gamma$ in $M_t$. If
 $\mathcal{L} = l+ i\theta$ denotes the complex length of any curve on $\partial U_1$, then the derivative of $\mathcal{L}$ is given by 
 \begin{align} \label{derivative}
 \frac{d\mathcal{L}}{dt} = \frac{-1}{4\alpha^2} (-2\mathcal{L}) + (x+iy) (2l)
 \end{align} and therefore 
 \begin{align} \label{derivatives}
 \frac{dl}{d\alpha} = \frac{l}{\alpha}(1 + 4\alpha^2 x) \quad \text{and} \quad \frac{d\theta}{d\alpha} = \frac{\theta}{\alpha} + 4\alpha y l 
 \end{align} where $x$ and $y$ satisfy
 \begin{align} \label{boundx}
\frac{-1}{\sinh^2(R)} \left( \frac{2 \sinh^2 (R) + 1}{2 \sinh^2(R) + 3 } \right) \leq 4 \alpha^2 x \leq \frac{1}{\sinh^2(R)}
\end{align} and 
\begin{align} \label{boundy}
|4\alpha^2 y|  \leq \frac{2}{ \sinh^2(R)}\frac{\cosh^2(R)}{(2\cosh^2(R) + 1)}.
\end{align}
\end{prop}

\begin{proof}  Let $X = M_t - \Sigma$. To each time $t$ for which the one-parameter deformation has been defined we can associate a class in $H^1(X;E)$ that is conformal at infinity.   We can represent this cohomology class in two ways.  First, by Lemma \ref{standard}, we can choose $\omega_0$ to be in standard form in a neighborhood $U_1$ of the singular locus and in each of the rank-$2$ cusps $U_i$ $(i = 2, \ldots, n)$ of $X$.  By our choice of parameterization $t = \alpha^2$, in a neighborhood $U_1$ we must have 
\begin{align} \label{omega zero}
\omega_0 = \frac{-1}{4\alpha^2} \omega_m + (x+iy) \omega_l
\end{align}  for some constants $x$ and $y$ since the derivative of the complex length of the meridian is determined by the coefficient of $\omega_m$ and the complex length of the meridian is $i\sqrt{t}$ at any time $t$  (see equation (4) of \cite{HK2}).  By integrating $\omega_m$ and $\omega_l$, Hodgson and Kerckhoff compute the effect of $\omega_m$ and $\omega_l$ on the infinitesimal change in the holonomy of any path on $\partial U_1$ on p. 32-33 of \cite{HK1} (see also Lemma 2.1 of \cite{HK2}).  Thus (\ref{derivative}) follows from (\ref{omega zero}) by their work. 

 From (\ref{derivative}), the real part gives us $\frac{dl}{dt}$ and the imaginary part gives us $\frac{d\theta}{dt}$.  Given the parameterization $t = \alpha^2$, the formulas for $\frac{dl}{d\alpha}$ and $\frac{d\theta}{d\alpha}$ given in (\ref{derivatives}) follow.

Now we want to derive the bounds on $x$ and $y$.  By Theorem \ref{hodge}, we can find a Hodge form $\omega$ in the same cohomology class as $\omega_0$ such that $\omega_c := \omega - \omega_0$ has finite $L^2$-norm outside $U_1$. Lemma 3.4 of \cite{HK1} shows that $\omega_c$ does not effect the holonomy of any of the peripheral elements.  By Proposition 2.6 of \cite{HK1}, we can write $\omega$ as  
 \[
\omega = \eta + i * D \eta
\] such that $D^* \eta = D*D\eta + \eta = 0$, and both $\eta$ and $D\eta$ are symmetric and traceless $TX$-valued forms. 

Since $\omega = \omega_0 + \omega_c$, we can decompose the real part of $\omega$ as $\eta = \eta_0 + \eta_c$ where $\eta_c$ does not effect the holonomy of the peripheral elements.  

We can find a smoothly embedded convex surface $S$ cutting off the geometrically finite ends of $X$ such that (after possibly shrinking the $U_i$ neighborhoods of the cusps) $U_i$ and $S$ are all pairwise disjoint.  Note that such a surface exists by Theorem 4.3 of \cite{B1} and may have multiple components (one for each end). 

By Lemma 2.3 of \cite{HK2}, we have that for any compact submanifold $N \subset X$ with $\partial N$ oriented with an inward pointing normal
\[
\int_N \| \omega \|^2 = \int_{\partial N}  * D \eta \wedge \eta.
\]  

Letting $N$ be the complement of the union $\cup_{i=1}^n U_i$ and the ends cut off by $S$, we can decompose the boundary integral:
\[
\int_N \| \omega \|^2 - \int_{S} * D \eta \wedge \eta =  \sum_{i=1}^n \int_{\partial U_i}* D \eta \wedge \eta 
\]  Let $b_i(\alpha, \beta) = \int_{\partial U_i}* D \alpha \wedge \beta$.

Lemma 2.5 of \cite{HK2} says that $b_i (\eta, \eta) = b_i(\eta_0, \eta_0) + b_i (\eta_c, \eta_c)$ and Lemma 2.6 of \cite{HK2} says that $b_i(\eta_c, \eta_c)$ is non-positive.  So we have 
\[
\int_N \| \omega \|^2 - \int_S * D \eta \wedge \eta  \leq   \sum_{i=1}^n b_i(\eta_0, \eta_0).
\]
Since this holds for any compact submanifold $N$, we can apply this to a family of compact submanifolds obtained by shrinking the neighborhood of the geometrically finite ends (but keeping the torus boundary components of $N$ the same as before).  That is, let $S_T$ be the surface obtained by taking a parallel copy of $S$ a distance $T$ further out the ends.  When we do this, $\int_{S_T}  * D \eta \wedge \eta  \to 0$ as $T \to \infty$, while the other term on the left, $\int_{N_T} \| \omega \|^2$, can only get larger (Theorem 4.6 of \cite{B1}).  So we must have
\[
\sum_{i=1}^n b_i(\eta_0, \eta_0) \geq 0
\] with equality if and only if the deformation is trivial (see eq. (16) of \cite{Purcell}).

We also note that Lemmas 3.1 and 3.2 in \cite{Purcell} imply that for $i \geq 2$,
\[
-b_i (\eta_0, \eta_0)= 2 | \zeta_i'(t)|^2 Area(\partial U_i) 
\] where $\zeta_i(t)$ is the path in the Teichm\"uller space (using the Teichm\"uller metric) of $\partial U_i$ throughout the deformation.  
  In particular $b_i(\eta_0, \eta_0) \leq 0$ for all $i\geq 2$.  It follows that $b_1(\eta_0, \eta_0) \geq 0$.  We now are in a position to follow the calculations on p. 382-384 of \cite{HK2} identically (for more steps in the calculations see \cite{HK2} or Chapter 4 of \cite{MagidThesis}).  

In order to bound $x$ and $y$, we begin by computing $b_1(\eta_0, \eta_0)$ in terms of $x$ and $y$ and some constants that only depend on $R$ and $\alpha$ (within this proposition, $R$ and $\alpha$ are fixed so we refer to $a_R, b_R, c_R$ as constants).  This is done on p. 382-383 of \cite{HK2}.  
\[
\frac{b_1(\eta_0, \eta_0)}{area(\partial V)}= a_R (x^2 +y^2) + b_R x + c_R
\] where 
\begin{align*}
a_R &= -\tanh(R) \frac{2\cosh^2(R) + 1}{\cosh^2(R)}  \\
b_R &= \frac{\tanh(R)}{2\alpha^2 \cosh^2(R) \sinh^2(R)}    \\
c_R &=  \frac{\tanh(R) + \tanh^3(R)}{16\alpha^4 \sinh^4(R)}.
\end{align*}

Using the fact that $b_1 (\eta_0, \eta_0) \geq 0$, we get
\[
\left( x+ \frac{b_R}{2a_R} \right)^2 + y^2 \leq \frac{b_R^2-4a_Rc_R}{4a_R^2}. 
\]


Since both of the terms $\left( x^2 + \frac{b_R}{2a_R} \right)^2$ and $y^2$ are positive, we get the inequalities \[
\left(x - \frac{1}{(4\alpha^2 \sinh^2(R))(2\cosh^2(R) + 1)} \right)^2 \leq \left(\frac{\cosh^2(R)}{(2\alpha^2 \sinh^2(R))(2\cosh^2(R)+1)}\right)^2
\] and \[
y^2 \leq\left(\frac{\cosh^2(R)}{(2\alpha^2 \sinh^2(R))(2\cosh^2(R)+1)}\right)^2.
\]
The inequalities (\ref{boundx}) and (\ref{boundy}) follow immediately.  This completes the proof which gives us a generalized version of Theorem 2.7 in \cite{HK2} for geometrically finite manifolds with rank-$2$ cusps.  Moreover, we obtain an additional bound on $y$ which will be useful later.
\end{proof}

Because the hypotheses in Proposition \ref{derivative_estimates} require the existence of an embedded tube $U_1$ of radius $R$ about $\gamma$, we now show there is some interval on which there is a lower bound to the size of such a tube.  From now on we will use $R_t$ to denote the maximum radius of an embedded tubular neighborhood about $\gamma$ in $M_t$, and let $V_t$ denote this $R_t$-neighborhood of $\gamma$. Note that this replaces the neighborhood $U_1$ we were using earlier because we are now interested in the parameter $t$ and no longer care about the other neighborhoods $U_i$, $i \geq 2$. Although not necessarily the optimal lower bound, it will be convenient to show $R_t \geq \sinh^{-1}(\sqrt{2})$.  In particular, this will allow us to invoke Theorem 1.2 of \cite{B2} in the proof of Lemma \ref{radius lemma 2}.  (See the comments preceding Theorem 3.5 on p. 796 of \cite{B2}.)

\begin{lem} \label{radius lemma} Suppose $M_t$ is defined for some interval $[0,t') \subset [0, (2\pi)^2]$, and let $R_t$ denote the maximal radius such that if $V_t$ is an $R_t$-neighborhood of $\Sigma$ in $M_t$ then $V_t$ is embedded.  Suppose $L^2 \geq 8(2\pi)^2$ where $L$ is the normalized length of the meridian of $\partial V_0$.  Then $R_t > \sinh^{-1}(\sqrt{2})$ for all $t \in [0, t')$.
\end{lem}

\begin{proof}
When $t = 0$, $M_0$ has a rank-$2$ cusp.  We can interpret $V_0$ as this rank-$2$ cusp and $R_0 = \infty$.  As we vary the metric, $R_t$ varies continuously.  Suppose there was some first time $t'' < t'$ such that $R_{t''}  = \sinh^{-1}(\sqrt{2})$.  Let $l(t)$ denote the length of $\gamma$ in $M_t$.  This is the same as the length of any curve on $\partial V_t$ homotopic to $\gamma$ so we can apply the bounds on $\frac{dl}{dt}$ in Proposition \ref{derivative_estimates}.   
We will find a contradiction by showing that $R_{t''}$ is bounded below by a function of $l(t'')$ and estimate $l(t'')$ using $L^2$. 

The area, $A_t$, of $\partial V_t$ satisfies
\[
A_t \geq 1.69785 \frac{\sinh^2 (R_t)}{\cosh(2R_t)}
\] by Theorem 4.4 of \cite{HK2} (see also Proposition 3.4 of \cite{B2} in the geometrically finite setting).  The value $1.69785$ is an approximation of $2 \sqrt{6} \sinh^{-1} \left(\frac{1}{2\sqrt{2}}\right)$ but we won't need this precision.   Define 
\[
h(r) = 1.69785 \frac{\tanh(r)}{\cosh(2r)}.
\]

\bigskip \noindent \textit{Remark.}  In Section 5 of \cite{HK2}, Hodgson and Kerckhoff define $h(r) = 3.3957 \frac{\tanh(r)}{\cosh(2r)}$.  Our definition differs by a factor of $2$ since we are allowing our manifold to have multiple cusps and geometrically finite ends (see Theorem 4.4 of \cite{HK2}).  This also allows us to directly apply Propositions 3.2 and 3.4 in \cite{B2}. 
\bigskip

Since $A_t = \alpha l(t) \sinh(R_t) \cosh(R_t)$ (p. 403 of \cite{HK2}), the maximal radius $R_t$ satisfies
\[
\alpha l(t) \geq h(R_t).
\]
In Lemma 5.2 of \cite{HK2}, Hodgson and Kerckhoff show that $h$ has a unique maximum, $h_{max}\approx 0.5098$, when $r \approx 0.531$ and is decreasing for all $r \geq 0.531$.  For any $0 \leq a \leq h_{max}$ we can define an inverse function $h^{-1}(a)$ to be the value of $r$ such that $r \geq 0.531$ and $h(r) = a$. One can easily see from the definition of $h(r)$ that $\lim_{r \to \infty} h(r) = 0$, so we interpret $h^{-1}(0) = \infty$.

Then we have  
\[
R_t \geq h^{-1}(\alpha l(t))
\] whenever $\alpha l(t) \leq h_{max}$ and $R_t \geq 0.531$.  We are assuming $R_t \geq \sinh^{-1}(\sqrt{2}) \approx 1.4622$ for all $0 \leq t \leq t''$ so the condition that $R_t \geq 0.531$ is immediately satisfied for all $t$ in this interval.  


If $\alpha l(t) \leq h_{max}$, set $\rho(t) = h^{-1}(\alpha l(t))$ which is clearly bounded above by $R_t$.   Note that $\alpha$ and $l(t)$ both start at zero when $t = 0$ so the condition that $\alpha l(t) \leq h_{max}$ holds in some interval around $t = 0$.


 Substituting $\rho$ for $R_t$ in the inequality (\ref{boundx}) gives us
\[
\frac{-1}{\sinh^2(\rho)} \left( \frac{2 \sinh^2 (\rho) + 1}{2 \sinh^2(\rho) + 3 } \right) \leq 4 \alpha^2 x \leq \frac{1}{\sinh^2(\rho)}.
\]   Proposition 5.5 of \cite{HK2} shows the lower bound is an increasing function of $R$ and the upper bound is a decreasing function.  Now set 
\[
u(t) = \frac{\alpha}{l(t)}.
\]  
Differentiating with respect to $t$, we find that 
\begin{align*}
\frac{du}{dt}= \frac{-1}{2\alpha l} (4\alpha^2 x).
\end{align*}
 and after the substitution $z = \tanh(\rho)$, 
\begin{align} \label{z}
- \frac{1 + z^2}{3.3956 (z^3)} \leq \frac{du}{dt}  \leq \frac{(1 + z^2)^2}{3.3956(z^3)(3 - z^2)}.
\end{align}

As long as $\alpha l \leq h_{max}$ we have $0.531 \leq \rho = h^{-1}(\alpha l)$ and therefore $0.48 \leq  z \leq 1$.  Since $\frac{1 + z^2}{3.3957 (z^3)}$ and $\frac{(1 + z^2)^2}{3.3957(z^3)(3 - z^2)}$ are both decreasing functions of $z$ over this interval, we can replace $z$ with $0.48$ to obtain the somewhat liberal bound
\begin{align} \label{boundu}
\left| \frac{du}{dt} \right| \leq 4.
\end{align}
Eq. (37) of \cite{HK2} shows that $\lim_{\alpha \to 0} u = L^2$ which implies that as long as $\alpha l(t) \leq h_{max}$ and $0 \leq t \leq t''$ we have 
\[
\left| \frac{\alpha}{l(t)} - L^2 \right| \leq 4t.
\] Since $L^2 \geq 8 (2\pi)^2$, we have that $L^2 \pm 4t$ is positive for any $t \leq (2\pi)^2$, so  
\begin{align} \label{boundl}
\frac{\alpha}{L^2 + 4t} \leq l(t) \leq \frac{\alpha}{L^2 - 4t}.
\end{align}
Multiplying by $\alpha$ and substituting $t = \alpha^2$ we get 
\begin{align} \label{boundalphal}
\alpha l(t) \leq \frac{t}{L^2 - 4t}.
\end{align} 

Since $L^2 \geq 8(2\pi)^2$ we have $L^2 - 4(2\pi)^2 \geq 4(2\pi)^2$.  Thus 
\[
\frac{(2\pi)^2}{L^2 - 4(2\pi)^2} \leq \frac{1}{4} < h_{max}.
\]   This implies that for any $0 \leq t \leq (2\pi)^2$, 
\[
\frac{t}{L^2 - 4t} \leq \frac{(2\pi)^2}{L^2 - 4(2\pi)^2}  < h_{max}
\] which in particular implies that $\alpha l(t) < h_{max}$ for all $t \in [0, t'']$. It also follows that   \[
R_{t''} \geq h^{-1}(\alpha l(t''))  \geq  h^{-1}\left( \frac{t''}{L^2 - 4t''} \right) \geq  h^{-1}\left( \frac{(2\pi)^2}{L^2 - 4(2\pi)^2} \right) \geq h^{-1}\left(\frac{1}{4} \right).
\]
  Since $\frac{1}{4} < h(\sinh^{-1}(\sqrt{2})) \approx 0.27725$, $R_{t''} > \sinh^{-1}(\sqrt{2})$. This contradicts that $R_{t''} = \sinh^{-1}(\sqrt{2})$ for any time $t'' < t'$, and so we have $R_{t} > \sinh^{-1}(\sqrt{2})$ for all $0 \leq t < t'$.  
\end{proof}

\textit{Remark.}  Although many of these estimates appear in \cite{HK2} and \cite{B2}, we will be using them to produce estimates in following subsection which do not appear in the existing literature. 
\bigskip

Using Theorems \ref{geometric limit} and \ref{open}, we can extend the one-parameter family $M_t$ to be defined for all $t \in [0, (2\pi)^2]$. The following lemma proves part $(i)$ of the filling theorem.  See Theorem 1.2 of \cite{B2} for the proof. 

\begin{lem} \label{radius lemma 2} If $L^2 \geq 8(2\pi)^2$, then the one-parameter family is defined for all $t \in [0, (2\pi)^2]$.
\end{lem}

\subsection{Complex Length Estimates}

Now that we have defined the one-parameter family for all $t \in [0, (2\pi)^2]$, we can integrate the estimates we found for $\frac{dl}{d\alpha}$ and $\frac{d\theta}{d\alpha}$ in Proposition \ref{derivative_estimates}.  This will allow us to estimate the complex length of any longitude on $\partial V_t$ at any $t$.  When $t = (2\pi)^2$, this produces estimates on the complex length of $\gamma$ in $M$.

First we consider the real part of the complex length of $\gamma$.  As in the proof of Lemma \ref{radius lemma}, we consider $u(t) = \frac{\alpha}{l(t)}$, which approaches $L^2$ as $t \to 0$. 
We can integrate the bounds on $\frac{du}{dt}$ in (\ref{boundu}) to estimate the length of $\gamma$ in $M$.  In other words, the inequalities in (\ref{boundl}) hold for all $t \in [0, (2\pi)^2]$. Thus we have shown

\begin{lem} \label{length lemma} If $L^2 \geq 8 (2\pi)^2$, the length of $\gamma$ in $M$ is given by $l((2\pi)^2)$ which satisfies 
\[
\frac{2\pi}{L^2 + 4(2\pi)^2} \leq l((2\pi)^2) \leq \frac{2\pi}{L^2 - 4(2\pi)^2}.
\]
\end{lem}

This immediately implies parts $(ii)$ and $(iii)$ of the filling theorem.  

Next we consider the imaginary part of the complex length of $\gamma$.  Again we consider any longitude on $\partial V_t$.  
Recall from Section \ref{complex length subsection} that for a fixed $t$ we can choose a shortest longitude on $\partial V_t$ and thus identify the imaginary part of the complex length of any longitude on $\partial V_t$ with a real number as opposed to modulo $\alpha$.   We begin by using the bounds from Proposition \ref{derivative_estimates} to bound the change in $\frac{\theta(t)}{\alpha}$ for any longitude on $\partial V_t$.

\begin{lem} \label{boundv} Suppose $L^2 \geq 8(2\pi)^2$.  For any fixed $t$ such that $0 < t \leq (2\pi)^2$ and any longitude on $\partial V_t$, let $\mathcal{L} = l + i \theta$ denote the complex length of that longitude.  Define $$v = \frac{\theta}{\alpha}.$$  Then 
\[
\left| \frac{dv}{d\alpha} \right| \leq  \frac{5(2\pi)}{(L^2 - 4(2\pi)^2)^2}.
\]
\end{lem}

\bigskip \noindent \textit{Remark.} Because we are considering the derivative at a fixed time $t$, we use $l = l(t)$ to denote the length at time $t$ and $\theta = \theta (t)$.  Also note that the role of $v$ is similar to that of the reciprocal of $u$, rather than $u$ itself. 
\bigskip

\begin{proof}  
Because $L^2 \geq 8(2\pi)^2$, the one-parameter family $M_t$ can be defined for all $t \in [0, (2\pi)^2]$ by Lemma \ref{radius lemma 2}.   Lemma \ref{radius lemma} implies that $R_t \geq \sinh^{-1}(\sqrt{2})$ for all $t$, and so we can apply the results of Proposition \ref{derivative_estimates} with a lower bound on $R$.

Recall from (\ref{derivatives}) in the statement of Proposition \ref{derivative_estimates} that $\frac{d\theta}{d\alpha} = \frac{\theta}{\alpha}   + 4\alpha y l$, so differentiating $v$ with respect to $\alpha$ gives us 
\[
\frac{dv}{d\alpha} = \frac{\left(\frac{d\theta}{d\alpha}  \right)}{\alpha} -\frac{ \theta}{\alpha^2} = 4yl. 
\]  
In order to obtain a bound on this quantity, we rewrite $4yl$ in the following way since we can bound $|4\alpha^2 y|$ using (\ref{boundy}). 
\begin{align} \label{dvdalpha}
\left| \frac{dv}{d\alpha} \right| = \left| 4yl \right| =  \left| \frac{l^2}{\alpha^2 l} 4\alpha^2 y \right|= (l)\left( \frac{1}{u} \right) \left( \frac{1}{\alpha l}\right) \left| 4 \alpha^2 y \right|.
\end{align}  We will bound each of these four quantities separately.   First, an upper bound for $l$ at any time $t$ is given in (\ref{boundl}).  For any $t\leq  (2\pi)^2$ this upper bound satisfies the uniform bound
\begin{align} \label{bound_on_l}
l \leq \frac{\alpha}{L^2 - 4t} \leq \frac{2\pi}{L^2 - 4(2\pi)^2}.
\end{align}

  Recall that $u = \frac{\alpha}{l}$ and so the quantity $\frac{1}{u} = \frac{l}{\alpha}$ can bounded using (\ref{boundu}).  Since $u$ approaches $L^2$ as $\alpha \to 0$ and $\left| \frac{du}{dt} \right| \leq 4$, we have that $| u - L^2| \leq 4\alpha^2$.  Therefore at any $t \leq (2\pi)^2$, a lower bound for $u$ is given by $u \geq L^2 - 4(2\pi)^2$ which implies 
\begin{align} \label{bound_on_u}
\frac{1}{u} \leq \frac{1}{L^2 - 4 (2\pi)^2}. 
\end{align}

Making the same changes of variables $\rho = h^{-1}(\alpha l)$ and $z = \tanh(\rho)$ that we made in the proof of the inequality (\ref{z}) in Proposition \ref{derivative_estimates}, we see that (as in eq. (38) of \cite{HK2})
\begin{align} \label{one_over_alpha_l}
\frac{1}{\alpha l} =\frac{1+ z^2}{1.69785(z)(1-z^2)}.  
\end{align}

Finally we must bound $4\alpha^2 y$ in terms of $z$.  As in the proof of Proposition \ref{derivative_estimates}, we can replace $R$ by $\rho$ in the inequality (\ref{boundy}) to get
\[
|4\alpha^2 y|  \leq \frac{2}{ \sinh^2(\rho)}\frac{\cosh^2(\rho)}{(2\cosh^2(\rho) + 1)}.
\] Using $\sinh^2 (\rho) = \frac{z^2}{1-z^2}$ and $\cosh^2(\rho) = \frac{1}{1-z^2}$ gives us
\begin{align} \label{four_alpha_squared_y}
|4\alpha^2 y| \leq \frac{2(1-z^2)}{3z^2 - z^4}.
\end{align}

Now we combine the bounds on $l$, $\frac{1}{u}$, $\frac{1}{\alpha l}$, and $|4\alpha^2y|$ in (\ref{bound_on_l}), (\ref{bound_on_u}), (\ref{one_over_alpha_l}), and (\ref{four_alpha_squared_y}) to get an estimate replacing eq. (\ref{dvdalpha}):
\[
\left| \frac{dv}{d\alpha} \right| \leq \left(\frac{2\pi}{L^2 - 4(2\pi)^2} \right) \left( \frac{1}{L^2 - 4(2\pi)^2} \right) \left(\frac{1+ z^2}{1.69785(z)(1-z^2)} \right) \left( \frac{2(1-z^2)}{3z^2 - z^4}\right).
\]

Since  $L^2 \geq 8(2\pi)^2$ we know (as in the proof of Lemma \ref{radius lemma}) that $\alpha l$ remains bounded above by $h_{max}$ for all $t$ and therefore $\rho \geq 0.531$.  This implies $1 \geq z \geq  0.4862$ throughout the deformation so we can bound the following function of $z$ by its value when $z = 0.48$ since it is decreasing on $[0.48, 1]$.
\[
\frac{2(1+ z^2)(1-z^2)}{1.69785(z)(1-z^2)(3z^2 - z^4)} \leq \frac{2(1+ (0.48)^2)(1-(0.48)^2)}{1.69785(0.48)(1-(0.48)^2)(3(0.48)^2 - (0.48)^4)} .
\]  This upper bound is approximately $4.73191$, so for any $z \in [0.48, 1]$, 
\[
\frac{2(1+ z^2)(1-z^2)}{1.69785(z)(1-z^2)(3z^2 - z^4)}  \leq  5.
\]
Thus
\[
\left| \frac{dv}{d\alpha} \right| \leq  \frac{5(2\pi)}{(L^2 - 4(2\pi)^2)^2}.
\]
\end{proof}

The quantity $v = v(\lambda', t)$ depends on both a longitude $\lambda'$ and the parameter $t$.  The previous Lemma shows that if $L^2$ is sufficiently large, then $v$ remains roughly constant.  Thus for any longitude $\lambda'$, $v(\lambda', (2\pi)^2)$ can be approximated by $\lim_{t\to 0} v(\lambda', t)$.  We claim that since $|A^2|>2$, there is a longitude $\lambda$ such that this limit exists and is equal to $\frac{1}{A^2}$, the normalized twist of the cusp we are filling $\hat{M}$.  

Let $b(t)$ and $m(t)$ be the twist and length of the meridian on the flat torus $V_t$, as in Definition \ref{defn_norm_twist}.  Since $M_t \to M_0$ geometrically, the flat tori $\partial V_t$ converge to $\partial V_0$.  Hence $\frac{b(t)}{m(t)}$ converges to $\frac{b(0)}{m(0)} = \frac{1}{A^2}$ unless
$\frac{b}{m} = \frac{1}{2}$ (in which case $\lim \frac{b(t)}{m(t)}$ could be $\frac{1}{2}$, $-\frac{1}{2}$, or not exist).  

\begin{lem} \label{bmlimit} Suppose that $|A^2| > 2$.  Then
\[
\lim_{t \to 0} \frac{b(t)}{m(t)} \to \frac{1}{A^2}.
\]
\end{lem}

When $\frac{b(t)}{m(t)} \neq \frac{1}{2}$, then there is a unique shortest longitude $\lambda$ on $\partial V_t$.  In particular, if $|A^2| >2$, there is a unique shortest longitude $\lambda$ on $\partial V_0$. 

\begin{lem} \label{longitude} If $|A^2|>2$, then there is some $\delta>0$ such that the imaginary part of the complex length of $\lambda$ lies in the interval 
$\theta(\lambda,t) \in \left( -\frac{\alpha}{2}, \frac{\alpha}{2} \right)$ for all $0 < t < \delta$.  Moreover, 
\[
\lim_{t\to 0} v(\lambda, t) = \lim_{t\to 0} \frac{\theta(\lambda,t)}{\alpha} = \frac{1}{A^2}.
\]
\end{lem}

\begin{proof}  Since $\left| \frac{1}{A^2}\right| < \frac{1}{2}$, there is some $\delta>0$ such that for all $t \in [0, \delta)$, $\frac{b(t)}{m(t)} \neq \frac{1}{2}$ by the previous lemma.  Hence, the longitude $\lambda$ that is shortest on $\partial V_0$ is the unique shortest longitude on $\partial V_t$ for all $t \in [0, \delta)$.  Thus, the imaginary part of the complex length satisfies $\theta(\lambda, t) \in \left(-\frac{\alpha}{2}, \frac{\alpha}{2} \right)$ for all $t \in (0, \delta)$.

For any $t \in (0,\delta)$, we have $\alpha > 0$, so it is clear from the definitions that 
\[
\frac{\theta(\lambda, t)}{\alpha} = \frac{b(t)}{m(t)}.
\]   
Thus 
\[
\lim_{t\to 0} v(\lambda, t) = \lim_{t\to 0} \frac{\theta(\lambda,t)}{\alpha}  = \lim_{t\to 0} \frac{b(t)}{m(t)} = \frac{1}{A^2}.
\] 
\end{proof}

We now use the bounds on $\frac{dv}{d\alpha}$ and the fact that $\lim_{t\to 0} v(\lambda,t) = \frac{1}{A^2}$ to estimate $v(\lambda, (2\pi)^2)$. 

\begin{lem} \label{theta lemma}
If $|A^2| > 2$ and $L^2 \geq 8(2\pi)^2$, then the complex length $l(\lambda, (2\pi)^2) + i\theta(\lambda, (2\pi)^2)$ of the longitude $\lambda$ on $\partial V_{(2\pi)^2}$ satisfies:
\[
\left| \theta(\lambda, (2\pi)^2) - \frac{2\pi}{A^2} \right| \leq \frac{5(2\pi)^3}{(L^2 - 4(2\pi)^2)^2}.
\]
\end{lem}

\begin{proof}  Recall that for any longitude, we obtained the bound 
\[
\left| \frac{dv}{d\alpha} \right| \leq  \frac{5(2\pi)}{(L^2 - 4(2\pi)^2)^2}
\] in Lemma \ref{boundv}.   For the longitude $\lambda$, we have $\lim_{t\to 0} v(\lambda,t) = \frac{1}{A^2}$, so we can integrate to find that, for any $t \leq (2\pi)^2$, we have
\[
\left| \frac{\theta(\lambda, t)}{\alpha} - \frac{1}{A^2} \right| \leq \frac{5(2\pi)(t)}{(L^2 - 4(2\pi)^2)^2}.
\]
Since $L^2 \geq 8(2\pi)^2$ we can define the deformation for all $t \in [0, (2\pi)^2]$, and therefore setting $t = (2\pi)^2$ gives us 
\begin{align} \label{theta inequality}
\left| \theta(\lambda, (2\pi)^2) - \frac{2\pi}{A^2} \right| \leq \frac{5(2\pi)^3}{(L^2 - 4(2\pi)^2)^2}.
\end{align}
\end{proof}

We can now prove part $(v)$ of Theorem \ref{filling}.  
\begin{lem} \label{part 5}  Let $\gamma$ be the core curve of the filling torus in $M = M_{(2\pi)^2}$.  If $L^2 \geq 8(2\pi)^2$ and $|A^2| \geq 3$, then the imaginary part of the complex length $\mathcal{L}(\gamma) = l(\gamma) + i\theta(\gamma)$, normalized so that $\theta(\gamma) \in (-\pi, \pi]$, satisfies:
\begin{align} \label{theta gamma}
\left| \theta(\gamma) - \frac{2\pi}{A^2} \right| \leq \frac{5(2\pi)^3}{(L^2 - 4(2\pi)^2)^2}.
\end{align} 
\end{lem}

\begin{proof}  Since any longitude on $\partial V_{(2\pi)^2}$ is homotopic to $\gamma$, the complex length $l(\gamma) + i \theta(\gamma)$ of $\gamma$ in $M_{(2\pi)^2}$ is given by the complex length of $\lambda$ on $V_{(2\pi)^2}$; although, the imaginary part of the complex length of $\lambda$ may not lie in the interval $(-\pi, \pi]$.  The conditions $|A^2| \geq 3$ and $L^2 \geq 8(2\pi)^2$, together with the inequality in (\ref{theta inequality}) imply \[
\left| \theta(\lambda, (2\pi)^2) \right| \leq \frac{5}{16(2\pi)} + \frac{2\pi}{3} < \pi.
\]  Hence $\theta(\gamma) = \theta(\lambda, (2\pi)^2) \in (-\pi, \pi)$, and inequality (\ref{theta gamma}) holds. 
\end{proof}

We now complete the proof of the filling theorem by summarizing what we have done to prove parts $(i), (ii), (iii)$, and deriving part $(iv)$.  Part $(i)$ was proven in Lemma \ref{radius lemma 2} when we showed that one can increase the cone angle from $0$ to $2\pi$.  Parts $(ii)$ and $(iii)$ were completed in Lemma \ref{length lemma}.  Part $(iv)$ follows from parts $(i)$, $(iii)$, and the drilling theorem (Theorem \ref{drilling}) in the following way.  Part $(i)$ provides the existence of $M$ (\textit{i.e.}, the $\beta$-filling of $\hat{M}$), and by the drilling theorem, there exists some $l_0$ such that if $l(\gamma) < l_0$ then there is a $J$-biLipschitz diffeomorphism 
\[
\phi: \hat{M} - \mathbb{T}_\epsilon (T) \to M - \mathbb{T}_\epsilon (\gamma).
\] By part $(iii)$ of the filling theorem, there exists some $K$ such that if the normalized length, $L$, of $\beta$ is at least $K$ then $l(\gamma) \leq \frac{2\pi}{L^2 - 4(2\pi)^2} < l_0$. Thus we can apply the drilling theorem to reverse the filling and obtain the desired biLipschitz map.  
\end{proof}

\bigskip \noindent \textit{Remark.}  Note that in parts $(i), (ii), (iii), (v)$ of the filling theorem, we only used the uniform bounds $L^2 \geq 8 (2\pi)^2$ and $|A^2| \geq 3$.  These four parts do not depend on the condition that $L\geq K$.  The constant $K$ depends on $J$ and $\epsilon$ and is therefore only necessary to conclude that if $L \geq K$, then the map $\phi$ is $J$-biLipschitz outside a Margulis $\epsilon$-thin region about the cusp $T$.  We also remark that since the filling map $\phi$ is obtained by applying the drilling theorem, we can assume that $\phi$ is level-preserving on cusps (see the remark following Theorem \ref{drilling}). 
\bigskip

 Before moving on to the next section, we remark that one could eliminate the hypothesis that $|A^2| \geq 3$ by not requiring the normalization $\theta(\gamma) \in (-\pi, \pi]$.  For instance, if we only assume $|A^2| > 2$, then (\ref{theta inequality}) still holds. If $A^2 = 2$, then Lemma \ref{boundv} could still be used to estimate the imaginary part of the complex length of $\gamma$ using an altered definition of the normalized twist
 (see the remarks in Section 4.3 of \cite{MagidThesis}).

 \subsection{Consequences and Generalizations}

In this section, we explain some of the consequences of the drilling and filling theorems.  First, we draw the following corollary of Theorems \ref{drilling} and \ref{filling} that will allow us to fill a manifold with multiple cusps.  If $\hat{M}$ has multiple cusps, we can fill them one at a time, using the filling theorem each time to obtain bounds on the lengths of the core curves of the filling tori.  In the statement of the following corollary, we will suppose our manifold has $d$ cusps of which $n\leq d$ are being filled.  We will assume they have been ordered so that the first $n$ are filled.  We label the cusps in $\hat{M}$ by $T_i$ and we label the core curves of the solid tori by $\gamma_i$.  

\begin{cor} \label{multiple fillings}  Let $J>1$, $l_0>0$, $\epsilon_3 \geq \epsilon>0$, and $n$ be given. There exists some $K$ such that the following holds:  suppose $\hat{M}$ is a geometrically finite manifold with $d \geq n$ cusps.  Suppose $\beta_i$ is a slope on the $i$th cusp of $\hat{M}$, $1 \leq i \leq n \leq d$. If the normalized length of $\beta_i$ is at least $K$ for each $i$, then \\
 $(i)$  we can fill in the $n$ cusps with labeled meridians, obtaining a manifold $M$ such that each $\beta_i$ bounds a disk in $M$ (in other words, the $\cup \beta_i$-filling of $\hat{M}$ exists);\\ 
$(ii)$  if $\gamma_i$ is the core curve of the torus used to fill the $i$th cusp, then $\sum_{i=1}^n l(\gamma_i) < l_0$;   \\
 $(iii)$ there exists a $J$-biLipschitz diffeomorphism
 \[
\phi : \hat{M} -  \cup_{i=1}^n \mathbb{T}_\epsilon (T_i) \to M -  \cup_{i=1}^n \mathbb{T}_\epsilon (\gamma_i).
\]  
\end{cor}

\begin{proof}  If the filled manifold $M$ exists, the drilling theorem says that there exists some $l_0'$, depending only on $J$ and $\epsilon$, such that if $\sum_{i=1}^n l(\gamma_i) < l_0'$, then there is 
a $J$-biLipschitz diffeomorphism
 \[
\phi : \hat{M} -  \cup_{i=1}^n \mathbb{T}_\epsilon (T_i) \to M -  \cup_{i=1}^n \mathbb{T}_\epsilon (\gamma_i).
\]  
Let $\ell= \min\{ l_0, l_0'\}$ where $l_0$ is the constant given in the statement of the Corollary.  We will show there exists a $K$ such that if the normalized length of $\beta_i$ is at least $K$ for $1 \leq i \leq n$, then the filled manifold $M$ exists, and the length of $\gamma_i$ in $M$ is less than $\frac{\ell}{n}$.  

We want to fill the cusps one at a time.  Let $M^0 = \hat{M}$, and if it exists let $M^i$ be the $\beta_i$ filling of $M^{i-1}$.  We will use the notation $l_{M^i} (\gamma_j)$ to denote the length of $\gamma_j$ in $M^i$.  In this notation, we want to show that $M^n = M$ exists and that $l_{M^n} (\gamma_i) < \frac{\ell}{n}$ for all $1 \leq i \leq n$. 

By the filling theorem, there exists some $K'$ (independent of $i$) such that if the normalized length of $\beta_i$ in $M^{i-1}$ is at least $K'$ then we have the following: 
\begin{itemize}
\item the $\beta_i$-filling of $M^{i-1}$, which we will call $M^i$, exists; 
\item the length of $\gamma_i$ in $M^i$ satisfies $l_{M^i}(\gamma_i) < \frac{\ell}{n2^{n}}$; 
\item  there is a $2$-biLipschitz map $\phi_i : M^{i-1} - \mathbb{T}_{\epsilon'} (T_{i}) \to M^i - \mathbb{T}_{\epsilon'} (\gamma_i)$ for some $\epsilon_3 \geq \epsilon' >0$.
\end{itemize}

Now let $K > 4^n K'$.  If the normalized length of $\beta_1$ is at least $K$, we can do the first filling to obtain $M^1$.  We now apply the filling theorem inductively. In the $i$th filling, the length of $\gamma_j$ (for any $1 \leq j < i$) does not increase by more than factor of $2$ since $\phi_i$ is $2$-biLipschitz. We next claim that the normalized length of $\beta_j$ (for any $i <j \leq n$) does not decrease by more than factor of $4$.

Fix a torus cross-section $T=T^2 \times \{1\}$ of the cusp $T_j \cong T^2 \times [0, \infty)$ contained in $M^{i-1} - \mathbb{T}_\epsilon (T_i)$.  Let $\mu$ be a curve on $T$ homotopic to the meridian $\beta_j$.  The normalized length of $\beta_j$ in $M^{i-1}$ is $\frac{l(\mu)}{\sqrt{area(T)}}$, where $l(\mu)$ is the length of a geodesic representative of $\mu$ on $T$ with respect to the induced Euclidean metric on $T$.  Since $\phi_i$ is $2$-biLipschitz, the length of a geodesic representative of $\phi_i(\mu)$ on $\phi_i(T)$ is bounded by $l(\phi_i(\mu)) >\frac{ l(\mu)}{2}$, and also $area(\phi_i(T)) < 4 (area(T))$. Thus 
\[
\frac{l(\phi_i(\mu))}{\sqrt{area(\phi_i(T))}}  > \frac{ l(\mu)/2}{\sqrt{ 4 (area(T))}}.
\]
By Theorem 6.12 of \cite{BB} (see the remark following Theorem \ref{drilling}), we can assume that $\phi_i(T)$ is a flat cross-section of the $j$th cusp in $M^i$.  Thus the ratio $\frac{l(\phi_i(\mu))}{\sqrt{area(\phi_i(T))}}$ in the left-hand side of the inequality above is the normalized length of $\beta_j$ in $M^i$.  This completes the proof that the normalized length of $\beta_j$ does not shrink by more than a factor of $4$ during the $i$th filling. 

Thus, for any $1 \leq i \leq n$, the normalized length of $\beta_i$ in $M^{i-1}$ is at least $4^{n-i} K'$.  So we can apply the filling theorem $n$ times to get $M$, the $\cup \beta_i$-filling of $\hat{M}$.  This completes part $(i)$.  Since the length of $\gamma_i$ in $M^i$ is less than $\frac{\ell}{n2^n}$, the length of $\gamma_i$ in $M$ is less than $\frac{\ell}{n} \frac{2^{n-i}}{2^n} \leq \frac{\ell}{n}$.  Hence $\sum_{i=1}^n l(\gamma_i) < \ell$.  Since $\ell = \min\{l_0, l_0'\}$ this completes parts $(ii)$ and $(iii)$.    
\end{proof}

Now suppose that $\hat{M}$ is the $\gamma$-drilling of $M$, and let $T$ denote the new cusp of $\hat{M}$.   Recall that if $\gamma$ is sufficiently short, the drilling theorem provides a biLipschitz diffeomorphism $\phi: \hat{M} - \mathbb{T}_{\epsilon_3}(T) \to M - \mathbb{T}_{\epsilon_3}(\gamma)$.  There is a unique slope $\beta$ on $\partial  \mathbb{T}_{\epsilon_3}(T)$ such that $\phi(\beta)$ bounds a disk in $\mathbb{T}_{\epsilon_3}(\gamma) \subset M$, but $\beta$ does not bound a disk in $\hat{M}$. 
Equivalently, the $\beta$-filling of $\hat{M}$ (if it exists) is $M$.  We say that $\beta$ is the meridian of $\hat{M}$.   If $\gamma$ is sufficiently short, then one can bound the normalized length of $\beta$ in $\hat{M}$ from below. This is stated without proof in part (2) of Theorem 2.4 of \cite{B3}.  
 
 \begin{prop} \label{normalized length corollary}  Let $K>0$.  There exists $l_0>0$ such that the following holds.  Let $M$ be a geometrically finite manifold containing a geodesic $\gamma$ of length less than $l_0$. If $\hat{M}$ is the $\gamma$-drilling of $M$ and $\beta$ is the meridian in $\hat{M}$ that bounds a disk in $M$, then the normalized length $L$ of the meridian $\beta$ in $\hat{M}$ is at least $K$.  
\end{prop}

\begin{proof}  By the drilling theorem, there exists $l_1$ such that if 
$l(\gamma) < l_1$ then there exists a $2$-biLipschitz map 
\[
\phi: \hat{M} - \mathbb{T}_{\epsilon_3} (T) \to M - \mathbb{T}_{\epsilon_3}(\gamma).
\]

Suppose that $R$ is the distance between $\gamma$ and $\partial \mathbb{T}_{\epsilon_3}(\gamma)$. The area of the boundary of this Margulis tube in $M$ is equal to 
\[
A(\partial \mathbb{T}_{\epsilon_3}(\gamma)) = A= 2\pi l(\gamma) \sinh(R) \cosh(R).
\]  (For example, see p. 403 of \cite{HK2}.)  Here $2\pi \sinh(R)$ gives the length of the meridian on $\partial \mathbb{T}_{\epsilon_3}(\gamma)$ in $M$.  
\
Define the normalized length of $\beta$ in $M$ to be $L_{M}(\beta) = \frac{2\pi\sinh(R)}{\sqrt{A}}$.  This is the length of (a geodesic representative of) $\phi(\beta)$ on $\partial \mathbb{T}_{\epsilon_3}(\gamma)$ divided by the square root of the area of $\partial \mathbb{T}_{\epsilon_3}(\gamma)$.  Unlike the normalized length of $\beta$ in $\hat{M}$, which is the length of $\beta$ on $\partial \mathbb{T}_{\epsilon_3}(T)$ divided by the square root of the area of $\partial \mathbb{T}_{\epsilon_3}(T)$, the normalized length of $\beta$ in $M$ depends on the size of the Margulis tube (in this case $\epsilon_3$). 

Now $L_{M}(\beta)= \frac{2\pi\sinh(R)}{\sqrt{A}} =\frac{ \sqrt{A} }{l(\gamma) \cosh(R)}$. Thus 
\[
L_{M}(\beta) = \sqrt{\frac{2\pi  \tanh(R)}{l(\gamma)}}.
\]  

The estimates in Brooks-Matelski \cite{BM} imply that given any $R_0$, there is some $l_2'$ such that if $l(\gamma) < l_2'$ then $R > R_0$.  Hence, there is some $l_2$ such that if $l(\gamma) < l_2$, then $L_M (\beta)> 4K$.

Now let $l_0 = \min\{ l_1, l_2\}$. This implies the filling map restricts to a $2$-biLipschitz diffeomorphism on the boundary tori: $\phi^{-1} : \partial \mathbb{T}_{\epsilon_3}(\gamma) \to \partial \mathbb{T}_{\epsilon_3}(T)$.  Hence, as we saw in the proof of the previous corollary, the normalized length of $\beta$ in $\hat{M}$ is no less than $\frac{1}{4}$ times the normalized length of $\beta$ on $ \partial \mathbb{T}_{\epsilon_3}(\gamma)$.  Thus, the normalized length of the meridian in $\hat{M}$ (which we are simply denoting by $L$) is at least 
\[
L \geq \frac{1}{4} L_M(\beta) > K.
\]
\end{proof}

\bigskip \noindent \textit{Remark.}  One can also prove Proposition \ref{normalized length corollary} using the tools developed in the proofs of Proposition \ref{derivative_estimates}, Lemma \ref{radius lemma}. One defines a one-parameter family of cone-manifolds by decreasing the cone-angle about $\gamma$ from $2\pi$ to $0$, showing that the maximal radius of a neighborhood of $\gamma$ does not become too small.  There is some $l_1$ such that if the length of $\gamma$ is less than $l_1$ then the one-parameter family of cone-manifolds can be defined for all $t \in [0, (2\pi)^2]$, and 
the estimate 
\[
\left| \frac{du}{dt} \right| \leq 4
\]  can be applied for all $t$.  Recall that $u(t) \to L^2$ as $t\to 0$ where $L^2$ is the square of the normalized length of $\beta$ in $\hat{M}$, and $u((2\pi)^2) = \frac{2\pi}{l(\gamma)}$ when $t = (2\pi)^2$.  From this, one can see that given any $K$, there exists some $l_0$ such that if $l(\gamma) < l_0$ then $L^2>K^2$.

\section{Constructing a Local Model of the Deformation Space}      \label{local homeomorphism section}

Let $S$ be a closed surface of genus at least two, set $N = S \times I$, and define the paring locus $P\subset \partial N$ to be a collection of annuli forming a pants decomposition of $S \times \{1\}$.  We will define a space $\mathcal{A}$ and show that $\mathcal{A}$ locally models a dense subset of the deformation space $AH(N)$.  We do this by constructing a map $\Phi : \mathcal{A} \to MP(N) \cup MP(N,P)$, and showing that there is some open set $U\subset \mathcal{A}$ and some neighborhood $V$ of $\sigma$ in $MP(N) \cup MP(N,P)$ such that $\Phi \vert_{U} : U \to V$ is a homeomorphism. The definition of $\Phi$ and the proof that it is continuous is in Section \ref{subsection: definition of phi}.  Then, in Sections \ref{subsection: inverse to phi} and \ref{subsection: local homeomorphism}, we show that $\Phi$ is a local homeomorphism $MP(N) \cup MP(N,P) \subset AH(N)$.  In order to find a point where $MP(N) \cup MP(N,P)$ is not locally connected, we must first show there exists such a point in $\mathcal{A}$.  

We begin by studying lower dimensional analogues to $\mathcal{A}$.   Let $S_{1,1}$ and $S_{0,4}$ denote the punctured torus and four-punctured sphere respectively. In Section \ref{subsection: punctured torus}, we define spaces $\mathcal{A}_{1,1}$ and $\mathcal{A}_{0,4}$ which, by Bromberg's results \cite{B3}, locally model the deformation spaces $AH(S_{1,1} \times I, \partial S_{1,1} \times I)$ and $AH(S_{0,4} \times I, \partial S_{0,4} \times I)$ respectively.  We define $\mathcal{A}$ in Section \ref{subsection: model space} similarly, and relate $\mathcal{A}$ to the lower dimensional versions by showing there exists a continuous surjection $\Pi: \mathcal{A} \to \mathcal{A}_{1,1}$ in Section \ref{subsection: projections}.  We use this in Section \ref{A not loc conn section}, along with the fact that $\mathcal{A}_{1,1}$ is not locally connected \cite{B3}, to show that $\mathcal{A}$ is not locally connected.  In fact, in Section \ref{A not loc conn section}, we find a point of $\mathcal{A}$ with the property that any sufficiently small neighborhood of this point contains infinitely many components that are bounded apart from each other.  In Section \ref{closure not locally connected section}, we will use this description of the components of a neighborhood $U \subset \mathcal{A}$ and the filling theorem, which is used in the definition of $\Phi$, to show that there is a point $\sigma_0 \in MP(N,P)$ such that for any sufficiently small neighborhood $\sigma_0 \in V \subset MP(N) \cup MP(N,P)$, the closure of $V$ has infinitely many components. Along with the Density Theorem, this will be used to show that $AH(N)$ is not locally connected.

\subsection{The Punctured Torus and Four-Punctured Sphere}  \label{subsection: punctured torus}

Define $N_{1,1} = S_{1,1} \times I$ and  $P_{1,1} = \partial S_{1,1} \times I$.  Let $P_{1,1}'$ be the union of $P_{1,1}$ with a non-peripheral annulus in $S_{1,1} \times \{ 1\}$ about a curve $b \times \{1\}$ (see Figure \ref{punctured_surface}). Let $\hat{N}_{1,1}$ be the manifold obtained by removing an open tubular neighborhood of $b \times \{ \frac{1}{2} \}$ from $N_{1,1}$, and let $\hat{P}_{1,1}$ be the union of $P_{1,1}$ with the toroidal boundary component of $\hat{N}_{1,1}$.

\begin{figure}[htbp]
\begin{center}
\psfrag{a}[][]{\small $b$}
\psfrag{b}[][]{ \small $a$}
\psfrag{c}[][]{\small $c$}
\psfrag{T}[][]{\small $S_{1,1}$}  
\psfrag{S}[][]{$S_{0,4}$}
\includegraphics[width=3in]{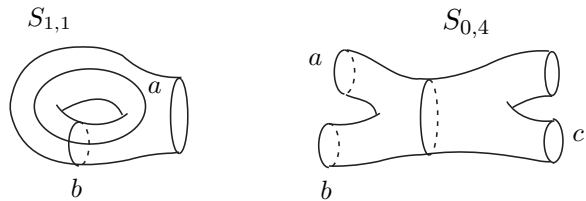} 
\caption[Punctured torus and four-punctured sphere]{\label{punctured_surface}  Orient the curves $a,b$ on $S_{1,1}$ and identify a presentation $\pi_1(S_{1,1}) = \left< a,b\right>$. Similarly, orient $a,b,c$ on $S_{0,4}$ and identify a presentation $\pi_1(S_{0,4})= \left< a,b,c\right>$.    }
\end{center}
\end{figure}

Recall from Section \ref{background deformation theory} that the components of $MP(N_{1,1}, P_{1,1}')$ are in one-to-one correspondence with the marked pared homeomorphism types of manifolds pared homotopy equivalent to $(N_{1,1}, P_{1,1}')$.  Fix an orientation on $N_{1,1}$ and let $\overline{N}_{1,1}$ denote $N_{1,1}$ with the opposite orientation.  Then there are two components of $MP(N_{1,1}, P_{1,1}')$ corresponding to $F^{-1}( [(N_{1,1}, P_{1,1}'), id] )$ and $F^{-1}( [(\overline{N}_{1,1}, P_{1,1}'), id] )$.  We denote the former by $MP_0(N_{1,1}, P_{1,1}')$. 

Given any $z \in \mathbb{C}$, one can define a representation $\sigma_z : \pi_1(N_{1,1}) \to PSL(2, \mathbb{C})$ by
\[
\sigma_z(a) = \begin{pmatrix} iz & i \\ i & 0 \end{pmatrix} \quad \sigma_z(b) = \begin{pmatrix} 1 & 2 \\ 0 & 1 \end{pmatrix}. 
\]
For any $\sigma \in MP(N_{1,1}, P_{1,1}')$ there is a unique $z$ such that $\sigma_z$ is in the conjugacy class of $\sigma$ (see Section 6.3 of \cite{Kra2} or \cite{KS}).  This defines an embedding of $MP(N_{1,1}, P_{1,1}')$ into $\mathbb{C}$.  The Maskit slice, $\mathcal{M}^+$, denotes the set of all $z \in \mathbb{C}$ such that $\sigma_z \in MP_0(N_{1,1}, P_{1,1}')$, and its mirror image in the lower half plane will be denoted by $\mathcal{M}^-$ (this corresponds to the other component of $MP(N_{1,1}, P_{1,1}')$).  Let $\mathcal{M} = \mathcal{M}^+ \cup \mathcal{M}^-$.  

\bigskip \noindent \textit{Remark.} We are following the Keen-Series convention by defining $\mathcal{M}^+$ and $\mathcal{M}^-$ to be open sets, whereas Bromberg uses $\mathcal{M}^\pm$ to denote the closures of these sets in Section 4 of \cite{B3}.  See Proposition 4.4 of \cite{B3} or \cite{KS} for more on the Maskit slice and its closure. The set $\mathcal{M}^+$ is also known as the Maskit embedding of the Teichm\"uller space of punctured tori.
\bigskip

Given $w \in \mathbb{C}$ and a conjugacy class of representations $\sigma \in MP_0(N_{1,1}, P_{1,1}')$, one can define a representation $\sigma_{z,w}$ of $\pi_1(\hat{N}_{1,1})=  \left< a,b,c \; | \; [b,c]=1\right>$ in the following way. There is a unique $z \in \mathcal{M}^+$ such that the representation $\sigma_z$ represents the conjugacy class $\sigma$.  Define $\sigma_{z,w}(a) = \sigma_z(a)$, $\sigma_{z,w}(b) = \sigma_z(b)$ and $\sigma_{z,w}(c) = \begin{pmatrix} 1 & w \\ 0 & 1 \end{pmatrix}$. Note that $\sigma_{z,w}$ is an actual representation of $\pi_1(\hat{N}_{1,1})$, but we will also use $\sigma_{z,w}$ to refer to the conjugacy class.

Define $\mathcal{A}_{1,1}$ to be 
\begin{align*}
\mathcal{A}_{1,1} = \{ (\sigma ,w) \in MP_0(N_{1,1}, P_{1,1}') \times \hat{\mathbb{C}} \; : \; &w = \infty, \;  \text{or} \; \\
&\sigma_{z,w} \in MP(\hat{N}_{1,1}, \hat{P}_{1,1}) \; \text{and} \; Im(w) > 0 \}.
\end{align*} Note that this is what Bromberg calls $\mathring{\mathcal{A}}$ in \cite{B3}.

One defines $\mathcal{A}_{0,4}$ similarly.  Let $N_{0,4} = S_{0,4} \times I$, let $P_{0,4} = \partial S_{0,4} \times I$, and define $P_{0,4}'$ to be the union of $P_{0,4}$ with a non-peripheral annulus in $S_{0,4} \times \{ 1\}$ whose core curve is $ab \times \{1\}$ (see Figure \ref{punctured_surface}). Let $\hat{N}_{0,4}$ be the manifold obtained by removing an open tubular neighborhood of $ab \times \{ \frac{1}{2} \}$ from $N_{0,4}$, and define $\hat{P}_{0,4}$ to be the union of $P_{0,4}$ with the toroidal boundary component of $\hat{N}_{0,4}$.

Given any $\sigma \in MP(N_{0,4}, P_{0,4}')$, there is a unique $z \in \mathbb{C}$ such that the representation $\sigma_z : \pi_1(N_{0,4}) \to PSL(2, \mathbb{C})$ defined by 
\[
\sigma_z(a) = \begin{pmatrix} -3 & 2 \\ -2 & 1 \end{pmatrix}, \quad \sigma_z(b) = \begin{pmatrix} 1 & 0 \\ 2 & 1 \end{pmatrix}, \quad \sigma_z (c) = \begin{pmatrix} -1 +2z & -2z^2 \\ 2 & -1-2z \end{pmatrix}
\] represents the conjugacy class of $\sigma$ (Section 6.1 of \cite{Kra2}). Note that for any $z$, one can check that $\sigma_z(ab) = \begin{pmatrix} 1 & 2 \\ 0 & 1 \end{pmatrix}$.

The Maskit slice for $S_{0,4}$, denoted $\mathcal{M}_{0,4}^+$, is the set of all $z \in \mathbb{C}$ such that $\sigma_z \in MP_0(N_{0,4}, P_{0,4}')$, and its mirror image in the lower half plane will be denoted by $\mathcal{M}_{0,4}^-$.  Again, we let $\mathcal{M}_{0,4} = \mathcal{M}_{0,4}^+ \cup \mathcal{M}_{0,4}^-$.  Kra shows that $z \in \mathcal{M}_{0,4}$ if and only if $2z \in \mathcal{M}^+$ (p. 558 of \cite{Kra2}).

Given $w \in \mathbb{C}$ and a conjugacy class of representations $\sigma \in MP_0(N_{0,4}, P_{0,4}')$, one can define a representation $\sigma_{z,w}$ of $\pi_1(\hat{N}_{0,4})=  \left< a,b,c,d \; | \; [ab, d]=1\right>$ in the following way. There is a unique $z \in \mathcal{M}_{0,4}^+$ such that the representation $\sigma_z$ represents the conjugacy class $\sigma$.  Define $\sigma_{z,w}(a) = \sigma_{z}(a)$, $\sigma_{z,w}(b) = \sigma_z(b)$, $\sigma_{z,w}(c) = \sigma_z(c)$, and $\sigma_{z,w}(d) = \begin{pmatrix} 1 & w \\ 0 & 1 \end{pmatrix}$. 

Define $\mathcal{A}_{0,4}$ to be 
\begin{align*}
\mathcal{A}_{0,4} = \{ (\sigma ,w) \in MP_0(N_{0,4}, P_{0,4}') \times \hat{\mathbb{C}} \; : \; &w = \infty, \;  \text{or} \;  \\
&\sigma_{z,w} \in MP(\hat{N}_{0,4}, \hat{P}_{0,4}) \; \text{and} \; Im(w) > 0 \}.
\end{align*}

Given $\sigma \in MP_0(N_{1,1}, P_{1,1}')$ and $w \in \mathbb{C}$, Bromberg characterizes when $(\sigma, w) \in \mathcal{A}_{1,1}$.

\begin{lem} \label{existence of n}  (i) \emph{(Bromberg \cite{B3})} Let $\sigma_z \in MP_0(N_{1,1}, P_{1,1}')$ and $w \in \mathbb{C}$ with $Im(w) > 0$.  Then $\sigma_{z,w} \in MP(\hat{N}_{1,1}, \hat{P}_{1,1})$ if and only if there exists an integer $n$ such that $z-nw \in \mathcal{M}^+$ and $z-(n+1)w \in \mathcal{M}^-$.
\\ (ii) Let $\sigma_z \in MP_0(N_{0,4}, P_{0,4}')$ and $w \in \mathbb{C}$ with $Im(w)>0$.  Then $\sigma_{z,w} \in MP(\hat{N}_{0,4}, \hat{P}_{0,4})$ if and only if there exists an integer $n$ such that $z-nw \in \mathcal{M}_{0,4}^+$ and $z-(n+1)w \in \mathcal{M}_{0,4}^-$.
\end{lem}

\textit{Remark.}  Part $(i)$ is Proposition 4.7 of \cite{B3}.  The proof of $(ii)$ is nearly identical to that of $(i)$, and we refer the reader to Lemma 5.1 of \cite{MagidThesis} for the necessary adaptations.

Although in the definitions of $\mathcal{A}_{1,1}$ and $\mathcal{A}_{0,4}$ we only require $Im(w) > 0$, the following lemma allows us to give a positive lower bound.  This will be used to obtain Corollary \ref{bound imaginary part} which is used in the proof of Lemma \ref{continuity2}.

\begin{lem} \label{lower bound on w (one dimensional)}  If $(\sigma, w) \in \mathcal{A}_{1,1}$, then $w= \infty$ or $Im(w)>2$. 
  If $(\sigma, w) \in \mathcal{A}_{0,4}$ then $w = \infty$ or $Im(w) > 1$. 
  \end{lem}

\begin{proof}
If $(\sigma, w) \in \mathcal{A}_{1,1}$, then there is some $z$ such that the conjugacy class of $\sigma_z$ is $\sigma$. If $w \neq \infty$, then by Lemma \ref{existence of n}$(i)$ there is some integer $n$ such that $z-nw \in \mathcal{M}^+$ and $z-(n+1)w \in \mathcal{M}^-$.  Thus $Im(w)$ is at least as large as the vertical distance between the components of $\mathcal{M}$.  
Wright shows that if $Im(z) \leq 1$ then $z \notin \mathcal{M}^+$ \cite{Wright} (see also p. 534, 558 of \cite{Kra2} and the comments after Proposition 2.6 of \cite{KS}). Since $\zeta \in \mathcal{M}^-$ if and only if $-\zeta \in \mathcal{M}^+$, the distance between these two components of the Maskit slice is at least $2$.  It follows that if $(\sigma, w) \in \mathcal{A}_{1,1}$ and $w \neq \infty$ then $Im(w) > 2$.  

The four-puntured sphere case is similar.  Suppose $(\sigma, w) \in \mathcal{A}_{0,4}$ and $w \neq \infty$. Let $z$ be such that the conjugacy class of $\sigma_z$ is  $\sigma$.   By Lemma \ref{existence of n}$(ii)$, $Im(w)$ is at least the vertical distance between the two components of $\mathcal{M}_{0,4}$.  Since $\zeta \in \mathcal{M}_{0,4}$ if and only if $2\zeta \in \mathcal{M}$ (p. 558 \cite{Kra2}), we must have $Im(w) > 1$. 
\end{proof}

\subsection{The Model Space $\mathcal{A}$}  \label{subsection: model space}

For surfaces with higher dimensional Teichm\"uller spaces, the construction of $\mathcal{A}$ takes more bookkeeping but is otherwise similar to $\mathcal{A}_{1,1}$ and $\mathcal{A}_{0,4}$.  

Recall $N = S \times I$.  Let $\{ \gamma_i\}_{i=1}^{d}$ be a pants decomposition of $S$ (recall from Section \ref{Introduction} that we will abbreviate $d = 3g-3$).  Although fixing any pants decomposition would be acceptable, we will make some choices that make it more convenient to apply results of \cite{B3}.  We will choose $\gamma_2$ to be a curve that separates $S$ into a punctured torus and a punctured genus $g-1$ surface, and let $\gamma_1$ be a curve in the punctured torus component of $S - \gamma_2$.  Also, fix an orientation on each $\gamma_i$ to distinguish $\gamma_i$ from its inverse in $\pi_1(S)$.
\begin{figure}[htbp]
\begin{center}
\psfrag{a}[][]{\small $\gamma_1$}  \psfrag{b}[][]{ \small $\gamma_2$}  \psfrag{S}[][]{$S$}
\includegraphics[width=3in]{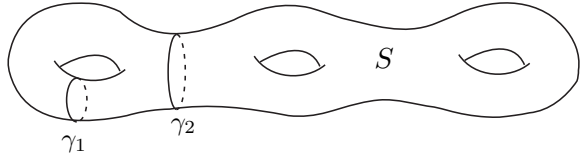} 
\caption[Pants decomposition]{\label{pants}  Part of the pants decomposition that we will fix throughout the rest of the argument.}
\end{center}
\end{figure} 

For each $i$, define a homeomorphism $G_i$ from either $S_{1,1}$ or $S_{0,4}$ to the component of $\left( S - \cup_{j\neq i} \gamma_j \right)$ containing $\gamma_i$. Moreover, using the markings of $S_{1,1}$ and $S_{0,4}$ from Figure \ref{punctured_surface}, define $G_i$ on $S_{1,1}$ so that $(G_i)_*(b) = \gamma_i$ and define $G_i$ on $S_{0,4}$ so that $(G_i)_*(ab) = \gamma_i$.

Let $P$ be a collection of $d$ disjoint annuli in $S\times \{1\}$ such that $\gamma_i \times \{1\}$ is core curve of the $i$th annulus of $P$ (see Figure \ref{pants}).  Then $MP(N,P)$ has $\binom{d}{2}$ components corresponding to whether $\gamma_i$ is parabolic to one side or the other.  For two of these components, all of the parabolics will be on the same side of $N$.  In other words, for any $\rho$ in one of these components, $M_\rho$ will have exactly one closed component of its conformal boundary homeomorphic to $S$. Using the notation from Section \ref{background deformation theory}, these two components can be identified with $F^{-1}[(N,P), id]$ and $F^{-1}[(\overline{N},P), id]$ where $\overline{N}$ denotes $N$ with the opposite orientation.   Label the former component by $MP_0(N,P)$.  These will be the marked hyperbolic manifolds with a rank-$1$ cusp associated to each $\gamma_i$ such that the cusps all occur to the ``top'' of the manifold.  

We now elaborate on what we mean by the ``top'' of a hyperbolic manifold. This discussion will be useful in distinguishing whether or not a point $\rho \in MP(N,P)$ lies in $MP_0(N,P)$.  Given $\rho \in MP(N,P)$, 
there exists an embedding $f: S \to M_\rho = \mathbb{H}^3/\rho(\pi_1(N))$ such that $f_* = \rho$.  The orientation on $S$ induces an orientation on $f(S)$.  This orientation, together with a normal direction to $f(S)$, defines an orientation on $M_\rho$, and for only one of the two normal directions will this orientation be compatible with the orientation induced on $M_\rho$ as a quotient of $\mathbb{H}^3$.  This distinguishes a top side of $f(S)$, and we say the top of the manifold $M_\rho$ with respect to $f(S)$ is the component of $M_\rho - f(S)$ that lies to the top side of $f(S)$.  If $\rho \in MP(N,P)$, then there are $d$ rank-$1$ cusps associated to each of the components of $P$.  If each of these cusps lies in the top of $M_\rho - f(S)$, then we say $\rho \in MP_0(N,P)$.  Since any two embeddings $f: S \to M_\rho$ such that $f_* = \rho$ are isotopic \cite{Waldhausen}, this notion is independent of the map $f$.  Likewise, if $\rho \in MP(N,P)$ and $X$ is a conformal boundary component of $M_\rho$, then we can distinguish whether $X$ lies on the top or bottom side of $M_\rho$.

Let $\hat{N}$ be obtained by drilling a set of $d$ curves out of $N$.  Specifically, let $\gamma_i \times \{1/2\}$ be a collection of $d$ curves in $S \times \{ 1/2\}$.  Let $U_i$ be an open collar neighborhood of $\gamma_i \times \{1/2\}$ such that the elements of the collection $\cup U_i$ are pairwise disjoint. Let \[
\hat{N} = N - \bigcup_{i=1}^{d} U_i .
\]
Define $\hat{P} = \cup \partial U_i$. 
Note that $\pi_1(\hat{N})$ is generated by $\pi_1(N)$ and $d$ new elements $\beta_i$ corresponding to meridians of $\partial U_i$.  Since the meridian of $\partial U_i$ commutes with any longitude of $\partial U_i$, there are new relations.  We will use $\gamma_i$ to denote both a curve and the element of $\pi_1(N)$ corresponding to that curve. Then $[\beta_i, \gamma_i]=1$.  We write 
\[
\pi_1(\hat{N}) = \left< \pi_1(N), \beta_1, \ldots, \beta_d \: | \: [\beta_i, \gamma_i]=1\right>
\] with the understanding that $\pi_1(N)$ has generators besides $\gamma_i$ and some of its own relations.

Given $\sigma \in MP_0(N,P)$ and $w=(w_1, \ldots, w_d) \in \mathbb{C}^d$, we describe a process for constructing a representation $\sigma_{w} : \pi_1(\hat{N}) \to PSL(2, \mathbb{C})$.  We can find a representative in the conjugacy class determined by $\sigma$ (which by an abuse of notation, we still refer to as $\sigma$) such that 
\[
\sigma(\gamma_1) = \begin{pmatrix} 1 & 2 \\ 0 & 1 \end{pmatrix}
\] and
\[
\sigma(\gamma_2) = \begin{pmatrix} -3 & -2 \\ 2 & 1 \end{pmatrix}. 
\]  
We choose these matrices to parallel the construction of $\mathcal{A}_{1,1}$ in the previous section.
Recall that we fixed some homeomorphism $G_1$ from $S_{1,1}$ to the subsurface of $S$ bounded by $\gamma_2$ that contains $\gamma_1$.  There is a unique $z \in \mathcal{M}^+$ such that $\sigma\circ (G_1)_*$ is conjugate to $\sigma_z$.  Note that for any $z$, $\sigma_z (b) = \begin{pmatrix} 1 & 2 \\ 0 & 1 \end{pmatrix} = \sigma \circ (G_1)_* (b)$ and $\sigma_z (b^{-1}a^{-1} b a) =  \begin{pmatrix} -3 & -2 \\ 2 & 1 \end{pmatrix} = \sigma \circ (G_1)_* (b^{-1} a^{-1} ba)$.  Thus, specifying $\sigma$ on $\gamma_1$ and $\gamma_2$ determines a well-defined representation in the conjugacy class of $\sigma$ that restricts to $\sigma_z$ for some $z$ on $G_1(S_{1,1})$.

  We define $\sigma_{w_1}: \pi_1(N - U_1) \to PSL(2, \mathbb{C})$ by $\sigma_{w_1}(\alpha) = \sigma(\alpha)$ for all $\alpha \in \pi_1(N)$ and \[
\sigma_{w_1}(\beta_1) = \begin{pmatrix} 1 & w_1 \\ 0 & 1 \end{pmatrix}. 
\]  We then inductively define $\sigma_{w_1, \ldots, w_i} : \pi_1(N - \cup_{j=1}^i U_j) \to PSL(2, \mathbb{C})$ by conjugating $\sigma_{w_1, \ldots, w_{i-1}}$ so that 
\[
\sigma_{w_1, \ldots, w_{i-1}}(\gamma_i) = \begin{pmatrix}1 & 2 \\ 0  & 1 \end{pmatrix}
\] (there is some ambiguity here that will be clarified below).  Then define $\sigma_{w_1, \ldots, w_{i}}$ by $\sigma_{w_1, \ldots, w_{i}}(\alpha) = \sigma_{w_1, \ldots, w_{i-1}}(\alpha)$ for all $\alpha \in \pi_1(N - \cup_{j=1}^{i-1} U_j)$, and
\[
\sigma_{w_1, \ldots, w_i}(\beta_i) = \begin{pmatrix}1 & w_i \\ 0  & 1 \end{pmatrix}.
\]  
As we indicated above, specifying that we should conjugate $\sigma_{w_1, \ldots, w_{i-1}}$ so that $\gamma_i$ is sent to $\begin{pmatrix}1 & 2 \\ 0 & 1 \end{pmatrix}$ does not determine a unique representation, but we now show how a unique representation can be specified.  Each curve $\gamma_i$ lies in either a four-punctured sphere or punctured torus component of 
\[
S - \bigcup_{j \neq i} \gamma_j
\] that we have marked by a homeomorphism $G_i$ from $S_{1,1}$ or $S_{0,4}$. 
If $\gamma_i$ lies in a punctured-torus component bounded by some curve $\gamma_j$ then we conjugate $\sigma_{w_1, \ldots, w_{i-1}}$ such that 
\[
\sigma_{w_1, \ldots, w_{i-1}} \circ (G_i)_* (b) = \begin{pmatrix} 1 & 2 \\ 0 & 1 \end{pmatrix} \quad \text{and} \quad \sigma_{w_1, \ldots, w_{i-1}} \circ (G_i)_* (b^{-1}a^{-1}ba) = \begin{pmatrix} -3 & -2 \\ 2 & 1 \end{pmatrix}.
\]
Since $G_i$ was chosen so that $(G_i)_*(b)  = \gamma_i$, this ensures $\sigma_{w_1, \ldots, w_{i-1}}(\gamma_i) = \begin{pmatrix}1 & 2 \\ 0  & 1 \end{pmatrix}$, and the condition that $\sigma_{w_1, \ldots, w_{i-1}}(\gamma_j) =  \begin{pmatrix} -3 & -2 \\ 2 & 1 \end{pmatrix}$ specifies $\sigma_{w_1, \ldots, w_{i-1}}$ uniquely. 

 If $\gamma_i$ lies in a four-punctured sphere component, then we conjugate $\sigma_{w_1, \ldots, w_{i-1}}$ so that 
 \[
\sigma_{w_1, \ldots, w_{i-1}} \circ (G_i)_* (a) = \begin{pmatrix} -3 & 2 \\ -2 & 1 \end{pmatrix} \quad \text{and} \quad \sigma_{w_1, \ldots, w_{i-1}} \circ (G_i)_* (b) = \begin{pmatrix} 1 & 0 \\ 2 & 1 \end{pmatrix}.
\]
It follows that 
\[
\sigma_{w_1, \ldots, w_{i-1}} \circ (G_i)_* (ab) = \begin{pmatrix} 1 & 2 \\ 0 & 1 \end{pmatrix}.
\]
 
After $d$ steps, we get a well-defined representation $\sigma_{w_1, \ldots, w_d}$ which we also denote by $\sigma_{w}: \pi_1(\hat{N}) \to PSL(2, \mathbb{C})$.  By construction, for each $i$ there exists some representation in the conjugacy class of $\sigma_{w}$ such that the generators $\gamma_i, \beta_i$ of $\pi_1(\partial U_i)$ are sent to $\begin{pmatrix} 1 & 2 \\ 0 & 1 \end{pmatrix}$ and $\begin{pmatrix} 1 & w_i \\ 0 & 1 \end{pmatrix}$ respectively.  

Given $\sigma \in MP_0(N,P)$, not every choice of $w= (w_1, \ldots, w_d) \in \mathbb{C}^d$ will result in $\sigma_{w} \in MP(\hat{N}, \hat{P})$.  Thus we are led to the following definition:
\begin{align*}
\mathcal{A} = \{ (\sigma, w) \in MP_0(N,P) \times \hat{\mathbb{C}}^d \; : \; & w= (\infty, \ldots, \infty), \; \text{or} \\
& Im(w_i) > 0 \; \text{and} \;  \sigma_{w} \in MP(\hat{N}, \hat{P})  \}. 
\end{align*}

\subsection{Projections of $\mathcal{A}$ to $\mathcal{A}_{1,1}$ and $\mathcal{A}_{0,4}$}  \label{subsection:  projections}

Now that we have defined the model space $\mathcal{A}$, we want to use the fact that $\mathcal{A}_{1,1}$ is not locally connected \cite{B3} to show that $\mathcal{A}$ is not locally connected.  In this section, we show there is a continuous surjection $\Pi: \mathcal{A} \to \mathcal{A}_{1,1}$, and in the sequel we explain how this can be used to relate the local connectivity of $\mathcal{A}$ and $\mathcal{A}_{1,1}$.

By our choice of pants decomposition (see Fig. \ref{pants}), the annulus $\gamma_2 \times [0,1]$ cuts $N$ into two pieces, one of which is homeomorphic to $N_{1,1}$ (and so we will refer to this component as $N_{1,1}$).  Given $\sigma : \pi_1(N) \to PSL(2, \mathbb{C})$, the restriction of $\sigma$ to this punctured torus defines a representation $\sigma\vert_{\pi_1(N_{1,1})} : \pi_1(N_{1,1}) \to PSL(2, \mathbb{C})$.  

\begin{lem} \label{restriction}
If $\sigma \in MP_0(N,P)$, then $\sigma\vert_{\pi_1(N_{1,1})} \in MP_0(N_{1,1}, P'_{1,1})$.
\end{lem}

\begin{proof}  If this restriction was not discrete, faithful, geometrically finite, and minimally parabolic with respect to $P_{1,1}' = P \cap N_{1,1}$ then $\sigma$ would not be in $MP(N,P)$.  Note that we are using the fact (attributed to Thurston) that finitely generated subgroups of geometrically finite Kleinian groups with nonempty domain of discontinuity are geometrically finite (Proposition 7.1 of \cite{Morgan}).  Hence, $\sigma\vert_{\pi_1(N_{1,1})} \in MP(N_{1,1}, P'_{1,1})$.

 Suppose $\sigma \vert_{\pi_1(N_{1,1})}$ is in $MP(N_{1,1}, P'_{1,1}) \backslash MP_0(N_{1,1}, P'_{1,1})$. Then for any proper embedding of $f_{1,1} : int(S_{1,1}) \to \mathbb{H}^3/\sigma \vert_{\pi_1(N_{1,1})} (\pi_1(N_{1,1}))$ inducing $\sigma \vert_{\pi_1(N_{1,1})}$, the cusp corresponding to $\gamma_1$ will lie to the bottom side of $f_{1,1}(int(S_{1,1}))$.  Since $\sigma \vert_{\pi_1(N_{1,1})}$ is defined as the restriction of $\sigma$, this implies that the cusp corresponding to $\gamma_1$ in $\mathbb{H}^3/\sigma(\pi_1(N))$ lies on the bottom with respect to any embedding $f : S \to \mathbb{H}^3/\sigma(\pi_1(N))$.  This contradicts that $\sigma \in MP_0(N,P)$.  Thus $\sigma \vert_{\pi_1(N_{1,1})} \in  MP_0(N_{1,1}, P'_{1,1})$. 
\end{proof}

Lemma \ref{restriction} allows us to define the projection map in the following lemma.

\begin{lem} \label{projection} The map $\Pi: \mathcal{A} \to MP_0(N_{1,1}, P_{1,1}') \times \hat{\mathbb{C}}$ defined by 
\[
\Pi(\sigma, w_1, \ldots, w_d) = (\sigma\vert_{\pi_1(N_{1,1})}, w_1).
\]
is a continuous map such that $\Pi(\mathcal{A}) = \mathcal{A}_{1,1}$. 
\end{lem}

\begin{proof} We first claim $\Pi(\mathcal{A}) \subset  \mathcal{A}_{1,1}$.  Recall the definitions of $\mathcal{A}$ and $\mathcal{A}_{1,1}$.  If a point $(\sigma, w_1, \ldots, w_d) \in \mathcal{A}$ satisfies $(w_1, \ldots, w_d) \neq (\infty, \ldots, \infty)$ then the extension $\sigma_{w_1, \ldots, w_d} \in MP(\hat{N}, \hat{P})$.  There is a $\pi_1$-injective pared embedding $\iota: \hat{N}_{1,1} \to \hat{N}$, such that the representation $\sigma_{w_1, \ldots, w_d} \circ \iota_* : \pi_1(\hat{N}_{1,1}) \to PSL(2, \mathbb{C})$ is conjugate to the extension of $\sigma \vert_{\pi_1(N_{1,1})}$ by $w_1$.  So this extended representation is discrete, faithful, geometrically finite, and minimally parabolic with respect to $\hat{P}_{1,1}$. Note that we are again using that finitely generated subgroups of geometrically finite groups are geometrically finite, provided the domain of discontinuity is nonempty. Thus the extension of $\sigma\vert_{\pi_1(N_{1,1})}$ by $w_1$ lies in $MP(\hat{N}_{1,1}, \hat{P}_{1,1})$, and so by the definition of $\mathcal{A}_{1,1}$, we have $\Pi(\sigma, w_1, \ldots, w_d) =  (\sigma\vert_{\pi_1(N_{1,1})} , w_1)\in \mathcal{A}_{1,1}$.  If $(w_1, \ldots, w_d) = (\infty, \ldots, \infty)$ then it follows immediately from Lemma \ref{restriction} and the definition of $\Pi$ that  $\Pi(\sigma, \infty, \ldots, \infty) = (\sigma\vert_{\pi_1(N_{1,1})}, \infty) \in \mathcal{A}_{1,1}$.

We now use Klein-Maskit combination to show that $\mathcal{A}_{1,1} \subset \Pi(\mathcal{A})$.  We begin by defining some new pared manifolds that arise as pieces of $N$ and $\hat{N}$ (see Figure \ref{pared_manifolds}). 
The annulus $\gamma_2 \times [0,1]$ divides $N$ into two pieces.  Let $N_{1,1}$ denote the closure of the piece containing $\gamma_1$ and let $N_3$ denote the closure of the remaining piece containing $\gamma_3, \ldots, \gamma_d$.  Let $\hat{N}_{1,1} = N_{1,1} - U_1$ and $\hat{N}_3  = N_3 - \cup_{i=3}^d U_i$.  Define $\hat{P}_{1,1} = \partial U_1 \cup \left( \gamma_2 \times [0,1] \right)$ and $\hat{P}_3 = \left( \cup_{i=3}^d \partial U_i \right) \cup \left( \gamma_2 \times [0,1] \right)$.  

Next define $N_2 = N - \cup_{i\neq 2} U_i$ and set $P_2$ to be the union of the $d-1$ tori $ \cup_{i\neq 2} \partial U_i$ with the annulus in $P \subset S \times \{1\}$ whose core curve is homotopic to $\gamma_2$.  Roughly speaking, $(N_2, P_2)$ is the pared manifold we will get by gluing $\hat{N}_{1,1}$ to $\hat{N}_{3}$ along $\gamma_2 \times [0,1]$.  We then obtain $(\hat{N}, \hat{P})$ from $(N_2, P_2)$ by drilling out $\gamma_2$.  Again, refer to Figure \ref{pared_manifolds}.

\begin{figure}[htbp]
\begin{center}
\psfrag{x}[][]{\footnotesize $(N,P)$}  
\psfrag{y}[][]{ \footnotesize $(\hat{N}, \hat{P})$}  
\psfrag{z}[][]{\footnotesize $(N_{1,1}, P_{1,1})$}
\psfrag{w}[][]{\footnotesize $(N_{1,1}, P_{1,1}')$}
\psfrag{t}[][]{\footnotesize $(\hat{N}_{1,1}, \hat{P}_{1,1})$}
\psfrag{u}[][]{\footnotesize $(\hat{N}_3, \hat{P}_3)$}
\psfrag{v}[][]{\footnotesize $(N_2, P_2)$}
\includegraphics[height=3.6in]{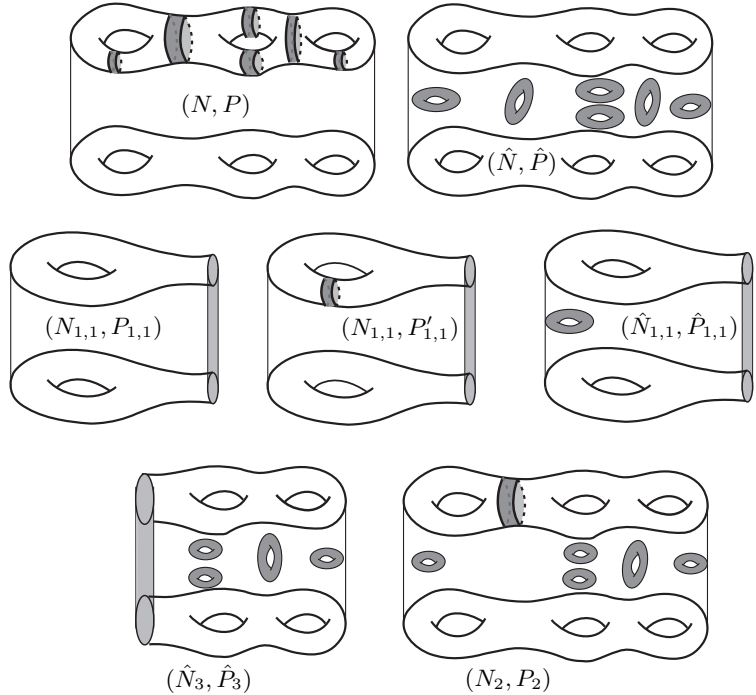} 
\caption[Pared manifolds]{\label{pared_manifolds}  Some of the pared manifolds we are using (illustrated in genus $3$).  The shaded regions indicate the paring locus.}
\end{center}
\end{figure}

For any $(\rho, w) \in \mathcal{A}_{1,1}$, $w\neq \infty$, we have $\rho_w \in MP(\hat{N}_{1,1}, \hat{P}_{1,1})$.  We can choose a representation $\hat{\rho}_1 : \pi_1(\hat{N}_{1,1}) \to PSL(2, \mathbb{C})$ in the conjugacy class $\rho_w$ such that $\hat{\rho}_1(\gamma_2) = \begin{pmatrix} 1 & 2 \\ 0 & 1 \end{pmatrix}$. Since $\rho_w$ is an extension of $\rho$, there is some representation  $\sigma_1 : \pi_1(N_{1,1}) \to \hat{\rho}_1(\pi_1(\hat{N}_{1,1}))$ representing the conjugacy class $\rho \in MP_0(N_{1,1}, P'_{1,1})$.  Note that the top and bottom of $\mathbb{H}^3/\sigma_1(\pi_1(N_{1,1}))$ are well-defined: the conformal boundary component on the top of $\mathbb{H}^3/\sigma_1(\pi_1(N_{1,1}))$ is a thrice-punctured sphere, and the bottom is a punctured torus.  For any $\hat{\rho}_1$ chosen as above, there is some $C_1$ such that $B^+ = \{ z \in \mathbb{C} \; : \; Im(z) > C_1 \}$ and $B^- = \{ z \in \mathbb{C} \; : \; Im(z) < -C_1 \}$ are precisely invariant for the subgroup $\left< \hat{\rho}_1 (\gamma_2) \right>$ in $\hat{\rho}_1(\pi_1(\hat{N}_{1,1}))$.  Without loss of generality, we can assume that $\hat{\rho}_1$ has been chosen such that $B^+ \subset \Omega(\hat{\rho}_1) \subset \Omega(\sigma_1)$ projects to the top surface of $\Omega(\sigma_1)/\sigma_1(\pi_1(\hat{N}_{1,1}))$.

Next, let $\hat{M}_3$ be a geometrically finite hyperbolic $3$-manifold homeomorphic to the interior of $\hat{N}_3$ whose only cusps are those associated to $\hat{P}_3$.  Let $h_3$ be an orientation preserving pared homeomorphism from $(\hat{N}_3, \hat{P}_3)$ to the relative compact core of $\hat{M}_3$.  The boundary $\partial \hat{N}_3 - \hat{P}_3$ has a top and bottom component, both of which are homeomorphic to a punctured genus $g-1$ surface $S_{g-1, 1}$.  We will call these $S_{g-1, 1, top}$ and $S_{g-1, 1, bot}$.  The homeomorphism $h_3$ distinguishes a top and bottom component of the relative compact core of $\hat{M}_3$ and thus distinguishes a top and bottom of the conformal boundary of $\hat{M}_3$.  Define a representation $\hat{\rho}_3 : \pi_1(\hat{N}_3) \to PSL(2, \mathbb{C})$ conjugate to $(h_3)_*$ such that \\
(a) $\hat{\rho}_3 (\gamma_2) = \begin{pmatrix} 1 & 2 \\ 0 & 1 \end{pmatrix}$, \\
(b) $B_3 = \{ z \in \mathbb{C} \; : \; Im(z) > -C_1 -1 \}$ is a precisely invariant set for the subgroup $\left< \hat{\rho}_3 (\gamma_2) \right>$ in $\hat{\rho}_3 (\pi_1(\hat{N}_3))$ \\
(c) the component of $\Omega(\hat{\rho}_3)$ containing $B_3$ projects to the bottom conformal boundary component of $\hat{M}_3$.

Let $\sigma_3 = \hat{\rho}_3 \vert_{\pi_1(S_{g-1, 1, bot})}$.  That is, $\sigma_3$ is the restriction of $\hat{\rho}_3$ to the natural inclusion of the fundamental group of the bottom surface into $\pi_1(\hat{N}_3)$.  Thus, $\mathbb{H}^3/\sigma_3(\pi_1(S_{g-1, 1}))$ has a rank-$1$ cusp associated to each of the curves $\gamma_2, \ldots, \gamma_d$, and all of these cusps are on the top since the representation $\sigma_3$ was constructed from the inclusion of the bottom surface into $\hat{M}_3$. Hence, $\sigma_3 \in MP_0(S_{g-1, 1} \times I, (S_{g-1, 1} \times\{1\}) \cap P)$.  
In other words, the paring locus is a collection of $d-1$ annuli in $S_{g-1, 1} \times \{1\}$ whose core curves are $\gamma_i \times \{1\}$ for $i = 2, \ldots, d$.

Now we can apply type I Klein-Maskit combination along the subgroup $\left< \hat{\rho}_3 (\gamma_2) \right>= \left< \hat{\rho}_1 (\gamma_2) \right>$.  See Section \ref{KM combination} for references and notation.  Note that $\pi_1(N_2) \cong \pi_1(\hat{N}_{1,1}) *_{\left< \gamma_2 \right>} \pi_1(\hat{N}_3)$.  Define a representation $\rho_2 : \pi_1(N_2) \to PSL(2, \mathbb{C})$ by setting $\rho_2(x) = \hat{\rho}_3(x)$ for all $x \in \pi_1(\hat{N}_3)$ and $\rho_2(x) = \hat{\rho}_1(x)$ for all $x \in \pi_1(\hat{N}_{1,1})$.  By construction, the representation $\rho_2$ is discrete, faithful, geometrically finite, and minimally parabolic with respect to the paring locus $P_2$.

We can also apply Klein-Maskit combination to the subgroups $\sigma_1(\pi_1(S_{1,1}))$ and $\sigma_3(\pi_1(S_{g-1, 1}))$.  Since $\pi_1(S) \cong \pi_1(S_{1,1}) *_{\left< \gamma_2 \right>} \pi_1(S_{g-1, 1})$, we can define a representation $\sigma: \pi_1(S) \to PSL(2, \mathbb{C})$ by $\sigma(x) = \sigma_1(x)$ for all $x \in \pi_1(S_{1,1})$ and $\sigma(x) = \sigma_3(x)$ for $x \in \pi_1(S_{g-1, 1})$.  This defines a discrete, faithful, geometrically finite representation whose parabolics consist precisely of the curves $\gamma_i$ in $S$.  Since $\sigma_1 \in MP_0(N_{1,1}, P_{1,1}')$ and $\sigma_3 \in MP_0(S_{g-1, 1} \times I, (S_{g-1, 1} \times \{1\}) \cap P)$, the cusps corresponding to $\gamma_1, \gamma_3, \ldots, \gamma_d$ are in the top of $\mathbb{H}^3/\sigma(\pi_1(S))$.  Moreover, since $B^-$ and $B_3$ cover the bottoms of their respective manifolds, the cusp associated to $\gamma_2$ is also in the top. It follows that $\sigma \in MP_0(N,P)$.

To summarize, we now have a representation $\sigma \in MP_0(N,P)$ such that $\sigma \vert_{\pi_1(N_{1,1})} = \sigma_1$ is conjugate to $\rho$.  The next step is to find $w_1, \ldots, w_d$ such that $\Pi(\sigma, w_1, \ldots, w_d) = (\rho, w)$.  

Since $\pi_1(\hat{N})$ is generated by $\pi_1(N_2)$ and the meridian of the second torus boundary component $\partial U_2$, we can extend the representation $\rho_2$ to a representation $\hat{\rho}: \pi_1(\hat{N}) \to PSL(2, \mathbb{C})$ by defining $\hat{\rho}$ to equal $\rho_2$ on $\pi_1(N_2)$ and sending the additional generator to $\begin{pmatrix} 1 & w_2 \\ 0 & 1\end{pmatrix}$ for $w_2 \in \mathbb{C}$.   There is some constant $C_2$ such that if $Im(w_2) > C_2$, then we can apply type II Klein-Maskit combination (see Section \ref{KM combination}).  In this case, $\hat{\rho}$ defines a geometrically finite representation of $\pi_1(\hat{N})$ which is minimally parabolic with respect to $(\hat{N}, \hat{P})$. 

On each torus boundary component of $\hat{N}$ there is a well-defined meridian.  Up to conjugation (that depends on $i$), we can assume $\hat{\rho}$ sends $\gamma_i$ to $\begin{pmatrix} 1 & 2 \\ 0 & 1 \end{pmatrix}$.  In this case, there will be some $w_i$ such that the meridian of $\partial U_i$ is sent to $\begin{pmatrix} 1 & w_i \\ 0 & 1 \end{pmatrix}$.  With this definition of $w_i$ we now have a point $(\sigma, w_1, \ldots, w_d) \in \mathcal{A}$.  

Next, we check that $w_1 = w$ and therefore $\Pi(\sigma, w_1, \ldots, w_d) = (\rho, w)$.  This follows by construction.  The extension of $\sigma$ by $(w_1, \ldots, w_d)$ is conjugate to $\hat{\rho}$.  The restriction of $\hat{\rho}$ to $\pi_1(\hat{N}_{1,1})$ is conjugate to $\rho_w$ (which is conjugate to $\hat{\rho}_1$).  Thus $w_1 = w$.  

Finally, if $(\rho, \infty) \in \mathcal{A}_{1,1}$, then we can pick any $w' \neq \infty$ such that $(\rho, w') \in \mathcal{A}_{1,1}$.  Following the same construction, we can find a point $(\sigma, w', w_2, \ldots, w_d) \in \mathcal{A}$ such that $\sigma \vert_{\pi_1(N_{1,1})}$ is conjugate to $\rho$.  Then the point $(\sigma, \infty, \ldots, \infty)$ is also in $\mathcal{A}$ and satisfies $\Pi(\sigma, \infty, \ldots, \infty) = (\rho, \infty)$. 
\end{proof}

Recall that each $\gamma_i$ was contained in a four-punctured sphere or punctured torus component of $S - \cup_{j \neq i } \gamma_j$.  Given $\sigma \in MP_0(N,P)$, one can define $\Pi_i$ similarly to $\Pi= \Pi_1$.  The first coordinate is obtained by restricting $\sigma$ to the $i$th such subsurface and the second coordinate is defined by projecting $(w_1, \ldots, w_d) \mapsto w_i$.  Lemma \ref{projection} generalizes to show that $\Pi_i (\mathcal{A}) = \mathcal{A}_{1,1}$ if $\gamma_i$ lives in a punctured torus or $\Pi_i (\mathcal{A}) = \mathcal{A}_{0,4}$ if $\gamma_i$ lives in a four-punctured sphere.  In fact, for $i >1$, we will only need the first paragraph of the proof of this Lemma which shows that $\Pi_i (\mathcal{A}) \subset \mathcal{A}_{1,1}$ or $\Pi_i (\mathcal{A}) \subset \mathcal{A}_{0,4}$.

We now get the following corollary to Lemmas \ref{lower bound on w (one dimensional)} and \ref{projection}.  

\begin{cor} \label{bound imaginary part}  For all $(\sigma, w_1, \ldots, w_d) \in \mathcal{A}$ with $(w_1, \ldots, w_d) \neq (\infty, \ldots, \infty)$, the imaginary part of $w_i$ is bounded below by
\[
Im(w_i) > 1
\] for all $i$.
\end{cor}

\subsection{$\mathcal{A}$ is Not Locally Connected} \label{A not loc conn section}

In Lemma 4.14 of \cite{B3}, Bromberg shows that there exists a point $(\sigma_z,\infty) \in \mathcal{A}_{1,1}$ at which $\mathcal{A}_{1,1}$ fails to be locally connected.  We will use Lemma \ref{projection} to extend this failure of local connectivity to $\mathcal{A}$.  Although our statement of Lemma \ref{box lemma} differs from Bromberg's statement in \cite{B3}, his proof implies this result (see also Lemma 5.6 in \cite{MagidThesis} for a detailed proof).

\begin{lem}[Bromberg \cite{B3}]  \label{box lemma} There exists a point $z_0 \in \mathcal{M}^+$, a closed rectangle $R$, and some $\delta > 0$ such that if $z$ lies in the $\delta$-neighborhood of $z_0$ in $\mathbb{C}$ then $z \in \mathcal{M}^+$ and the set
\[\mathcal{A}_z = \{w \in \hat{\mathbb{C}} \; : \; (\sigma_z,w) \in \mathcal{A}_{1,1} \}
\] satisfies \\
$(i)$ $\mathcal{A}_z \cap int(R) \neq \emptyset$, and \\
$(ii)$ The distance between $\overline{\mathcal{A}_z}$ and $\partial R$ is at least $\delta$.   \\
Moreover, we can choose $R$ such that its sides are parallel to the axes and its width is $<2$. 
\end{lem}

Let $O$ denote the $\delta$ neighborhood of $z_0$, and let $W$ be an open neighborhood of $\sigma_{z_0}$ in $MP_0(N_{1,1}, P_{1,1}')$ such that for all $\sigma_{z} \in \overline{W}$, $z \in O$.  In other words, if $\tau: \mathcal{M}^+ \to MP_0(N_{1,1}, P_{1,1}')$ is the homeomorphism $z \mapsto \sigma_z$, then $W$ is a neighborhood of $\sigma_{z_0}$ such that $\tau^{-1}(\overline{W}) \subset O$.  

By Lemma \ref{box lemma}, $\mathcal{A}_{z_0} \cap int(R) \neq \emptyset$, so let $u_{1} \in int(R)$ such that $(\sigma_{z_0}, u_1) \in  \mathcal{A}_{1,1}$.  Lemma \ref{projection} shows that $\Pi: \mathcal{A} \to \mathcal{A}_{1,1}$ is a surjection, thus there is some $(\sigma_0, u_1, \ldots, u_d) \in \Pi^{-1}(\sigma_{z_0}, u_1)$.  Note that since $(\sigma_0, u_1, \ldots,u_d) \in \mathcal{A}$, we have $\sigma_0 \in MP_0(N,P)$.  Thus, by the definition of $\mathcal{A}$, we also have $(\sigma_0, \infty, \ldots, \infty ) \in \mathcal{A}$.  

We now claim that $(\sigma_0, u_1 + 2n, \ldots, u_d + 2n) \in \mathcal{A}$ for all $n$.  The definition of $(\sigma_0, u_1, \ldots, u_d)$ belonging to the set $\mathcal{A}$ is that the extension of $\sigma_0$ by $(u_1, \ldots, u_d)$ to a representation of $\hat{N}$ lies in $MP(\hat{N}, \hat{P})$. The extension of $\sigma_0$ by $(u_1 + 2n, \ldots, u_d + 2n)$ (as in Section \ref{subsection: model space}) has the same image in $PSL(2, \mathbb{C})$ as the extension of $\sigma_0$ by $(u_1, \ldots, u_d)$.  Hence, $(\sigma_0,u_1 + 2n, \ldots, u_d + 2n) \in \mathcal{A}$ for all $n$.

Let $U$ be a neighborhood of $(\sigma_0, \infty, \ldots, \infty)$ in $\mathcal{A}$ such that for all $(\sigma, w_1, \ldots, w_d)$ in $U$, the first coordinate of $\Pi(\sigma, w_1, \ldots, w_d)$ lies in $W$.  That is, for all $(\sigma, w_1, \ldots, w_d) \in U$, $\sigma \vert_{\pi_1(N_{1,1})} \in W$.

For each $n$, let $C_n$ be the collection of components of $U$ defined by 
\[
C_n = U \cap \Pi^{-1} (W \times (R+2n))).
\]   
Equivalently, 
\[
C_n = \{ (\sigma, w_1, \ldots, w_d) \in U \; : \; w_1 \in (R+2n) \}. 
\]
Here, $R+2n$ denotes the box $R$ translated by $2n$.  

We have already produced a sequence of points $(\sigma_0, u_1 + 2n, \ldots, u_d + 2n) \in \Pi^{-1}(W \times (R+2n))$ converging to $(\sigma_0, \infty, \ldots, \infty)$; hence, $C_n$ is nonempty for all but finitely many $n$. Now we claim that $\overline{C_n}$ and $\overline{U-C_n}$ are uniformly bounded apart from each other.

\begin{lem} \label{lowerboundw} There exists some $\delta>0$ such that for all $n$, if $(\sigma , w_1, \ldots, w_d) \in \overline{C_n}$ and $(\sigma',w_1', \ldots, w_d') \in \overline{U-C_n}$ then \[
|w_1-w_1'| > \delta.
\] 
\end{lem}

\begin{proof}  For any $(\sigma, w_1, \ldots, w_d) \in \overline{U}$, we have $\sigma \vert_{\pi_1(N_{1,1})} \in \overline{W}$.
Since $\overline{W} \subset \tau^{-1}(O)$, if we let $z = \tau(\sigma \vert_{\pi_1(N_{1,1})})$, then Lemma \ref{box lemma} implies that $\overline{\mathcal{A}_z}$ and $\partial R$ are at least a distance $\delta$ apart.  Since the set $\mathcal{A}_z$ is invariant under the translation $w \mapsto w+2$, points in $\overline{\mathcal{A}_z} \cap (R+2n)$ are bounded away from points in $\overline{\mathcal{A}_z} \cap (\mathbb{C} - (R+2n))$ by a distance of at least $2\delta$, which is strictly greater than $\delta$.  
\end{proof}

Since $C_n$ is nonempty for all but finitely many $n$ and the sets $C_n$ accumulate to $(\sigma_0, \infty, \ldots, \infty)$ Lemma \ref{lowerboundw} shows that any neighborhood $U' \subset U$ containing $(\sigma_0, \infty, \ldots, \infty)$, $U'$ has infinitely many components.  Thus we have shown


\begin{prop} \label{A not loc conn} There is a point $\sigma_0 \in MP_0(N,P)$ such that $\mathcal{A}$ is not locally connected at $(\sigma_0, \infty, \ldots, \infty)$.
\end{prop}

\textit{Remark.} We do not need the closures of $C_n$ and its complement to be disjoint to conclude that $\mathcal{A}$ is not locally connected, but we will need the full strength of Lemma \ref{lowerboundw} in the following section.

\subsection{Definition of $\Phi$}  \label{subsection: definition of phi}

Now that we have defined $\mathcal{A}$ and shown that it fails to be locally connected at some point $(\sigma_0, \infty, \ldots, \infty)$, we want to construct a map $\Phi$ from a subset of $\mathcal{A}$ containing this point into $AH(S \times I)$.   In this section, we show that $\Phi$ is well-defined on some subset of $\mathcal{A}$ and in the subsequent sections, we will show that $\Phi$ is a local homeomorphism from $\mathcal{A}$ to $MP(N) \cup MP_0(N,P) \subset AH(N)$ at the point $(\sigma_0, \infty, \ldots, \infty)$.  

As in Section 3 of \cite{B3}, we construct the map $\Phi$ in two steps.  Heuristically, points in $\mathcal{A}$ with $(w_1, \ldots, w_d)= (\infty, \ldots, \infty)$ parameterize $MP_0(N,P)$, and points $(\sigma, w_1, \ldots, w_d) \in \mathcal{A}$ parameterize a subset of $MP(\hat{N}, \hat{P})$. For points $(\sigma, \infty, \ldots, \infty) \in \mathcal{A}$, the representation $\sigma \in MP_0(N,P) \subset AH(N)$, so we will define $\Phi(\sigma, \infty, \ldots, \infty)  = \sigma$.  For all other points $(\sigma, w_1, \ldots, w_d) \in \mathcal{A}$, the representation $\sigma_{(w_1, \ldots, w_d)} \in MP(\hat{N}, \hat{P})$ and so $\mathbb{H}^3/\sigma_{(w_1, \ldots, w_d)} (\pi_1(\hat{N}))$ is a marked hyperbolic manifold with $d$ rank-$2$ cusps.  For these points, we will define $\Phi(\sigma, w_1, \ldots, w_d)$ to be the marked hyperbolic manifold in $MP(N)$ obtained by filling in these cusps.  We use the filling theorem to show that $\Phi$ is well-defined on some subset of $\mathcal{A}$ and that $\Phi$ is continuous.  

Let $(\sigma,w) \in \mathcal{A}$ such that $w \neq (\infty, \ldots, \infty)$.  By the definition of $\mathcal{A}$ we have that $\sigma_{w} \in MP(\hat{N}, \hat{P})$.  Let $\hat{M}_{\sigma, w} = \mathbb{H}^3/\sigma_{w}(\pi_1(\hat{N}))$ be the corresponding geometrically finite manifold with $d$ cusps. 

Recall that $\epsilon_3$ denotes the Margulis constant in dimension 3. By Corollary \ref{multiple fillings}, there is a constant $K$ such that if 
\[
\frac{|w_i|}{\sqrt{2Im(w_i)}}> K
\] for all $i$, then we can $\beta_i$-fill the $i$th cusp ($i = 1, \ldots, d$) to get a hyperbolic manifold $M_{\sigma, w}$ with the same conformal boundary as $\hat{M}_{\sigma, w}$, and there exists a biLipschitz diffeomorphism
 \[
\phi_{\sigma, w} : \hat{M}_{\sigma,w} - \mathbb{T}_{\epsilon_3}(T) \to M_{\sigma,w} - \mathbb{T}_{\epsilon_3}(\gamma).
\]  Here $T$ denotes the union of the cusps $T_i$ and $\gamma$ denotes the union of the curves $\gamma_i$.  
 
  Define 
\[
\mathcal{A}_{K} = \{ (\sigma, w) \in \mathcal{A} \; : \;  w = (\infty, \ldots, \infty), \; \text{or} \; \frac{|w_i|}{\sqrt{2Im(w_i)}}> K \; \text{for all} \; i \}.
\]

Recall that $\sigma \in MP_0(N,P)$ can be identified with a marked hyperbolic manifold $(M_\sigma, f_\sigma)$.  Without loss of generality, we can assume that $f_\sigma$ is a smooth immersion and that
$f_\sigma(N)$ does not intersect the $\epsilon_3$-parabolic thin part of $M_\sigma$ since $f_\sigma$ is only defined up to homotopy.  As $\sigma(\pi_1(N))$ is a subgroup of $\sigma_{w}(\pi_1(\hat{N}))$ we have a covering map 
\[
\pi_{\sigma,w} : M_\sigma \to \hat{M}_{\sigma,w}
\]  Now define $$f_{\sigma,w} = \phi_{\sigma,w} \circ \pi_{\sigma,w} \circ f_\sigma. $$  Since we have assumed $f_\sigma(N)$ avoids the $\epsilon_3$-parabolic thin part of $M_\sigma$, $\pi_{\sigma,w} \circ f_\sigma (N)$ is contained in $\hat{M}_{\sigma,w} - \mathbb{T}_{\epsilon_3}(T)$.  Thus, post-composition by the filling map $\phi_{\sigma, w}$ makes sense here. 

We next claim (as in Lemma 3.6 of \cite{B3}) that $(f_{\sigma,w})_*$ is an isomorphism from $\pi_1(N)$ to $\pi_1(M_{\sigma,w})$ and therefore $(M_{\sigma,w}, f_{\sigma,w}) \in AH(N)$.  First observe that $f_\sigma$ is a homotopy equivalence so we only need to show that 
$$(\phi_{\sigma, w})_* \circ (\pi_{\sigma, w})_*: \pi_1(M_\sigma) \to \pi_1(M_{\sigma, w})$$
is an isomorphism.  
Recall $\pi_1(\hat{N}) = \left< \pi_1(N), \beta_1, \ldots, \beta_d \; \vert \; [\beta_i, \gamma_i]=1 \right>$.  By the definition of the covering map $\pi_{\sigma, w}$ and the definition of the representation $\sigma_w$,  
\[
\pi_1(\hat{M}_{\sigma, w}) = \left< (\pi_{\sigma, w})_*( \pi_1(M_\sigma)), \sigma_w (\beta_1), \ldots, \sigma_w(\beta_d) \; \vert \; [\sigma_w(\beta_i),\sigma_w(\gamma_i)]=1 \right>.  
\]
Now the filling map $(\phi_{\sigma, w})_*$ kills the meridians in $\hat{M}_{\sigma, w}$ which were precisely the group elements $\sigma_w(\beta_i)$.  Thus 
\[
(\pi_{\sigma, w})_*( \pi_1(M_\sigma)) \cap Ker((\phi_{\sigma, w})_* ) = \{1\}
\] and therefore $(\phi_{\sigma, w})_* \circ (\pi_{\sigma, w})_*$ is an isomorphism from $\pi_1(M_{\sigma})$ onto its image, which is $\pi_1(M_{\sigma, w})$. 

 Moreover, as the filling preserves the conformal boundary components of $\hat{M}_{\sigma, w}$ and the filled manifold $M_{\sigma, w}$ has no cusps, $(f_{\sigma, w})_*$ is a minimally parabolic, geometrically finite representation in $AH(N)$. 

When $w = (\infty, \ldots, \infty)$, we define $\Phi(\sigma, \infty, \ldots, \infty) = \sigma \in MP_0 (N,P)$.  

So we have 
\[
\Phi(\sigma,w) = \begin{cases} (f_{\sigma,w})_* \quad &\text{if} \; w \neq (\infty, \ldots, \infty) \\ 
\sigma \quad &\text{if} \; w = (\infty, \ldots, \infty).  \end{cases}
\]
Thus we have defined $\Phi$ on some subset $\mathcal{A}_{K} \subset \mathcal{A}$ such that $\Phi(\mathcal{A}_K) \subset MP(N) \cup MP_0(N,P)$. 

\begin{lem} \label{phi continuity} The map $\Phi$ is continuous on $\mathcal{A}_K$.  
\end{lem}

\begin{proof}  Let $(\sigma_0, w_0)$ be a point in $\mathcal{A}_K$ with $w_0 \neq (\infty, \ldots, \infty)$. Let $B$ be the component of $\left( \mathcal{A}_K - \{ (\sigma, w)  \; : \; w= (\infty, \ldots, \infty) \} \right)$ containing $(\sigma_0,w_0)$.  Clearly the correspondence $(\sigma, w) \mapsto \sigma_w$ is a continuous map from $\left( \mathcal{A}_K - \{ (\sigma, w)  \; : \; w= (\infty, \ldots, \infty) \} \right)$ to $MP(\hat{N}, \hat{P})$ and thus takes the component $B$ into one of the components $C$ of $MP(\hat{N}, \hat{P})$.  Recall from Section \ref{background deformation theory} that  $C = F^{-1}([(\hat{N}_C, \hat{P}_C), h_C])$ for some $[(\hat{N}_C, \hat{P}_C), h_C] \in A(\hat{N}, \hat{P})$.  For any point $(\hat{M}_{\hat{\rho}}, f_{\hat{\rho}}) \in C$, the map $f_{\hat{\rho}} \circ h_C^{-1}$ is homotopic to a pared homeomorphism from $(\hat{N}_C, \hat{P}_C)$ to the relative compact core of $\hat{M}_{\hat{\rho}}$, and thus we can use $f_{\hat{\rho}} \circ h_C^{-1}$ to define a marking from $\partial \hat{N}_C - \hat{P}_C$ to the conformal boundary of $\hat{M}_{\hat{\rho}}$.  The Ahlfors-Bers parameterization $\widehat{\mathcal{AB}}_C : C \to \mathcal{T}(\partial \hat{N}_C - \hat{P}_C)$ is defined by sending $(\hat{M}_{\hat{\rho}}, f_{\hat{\rho}})$ to the conformal boundary of $\hat{M}_{\hat{\rho}}$ marked by $f_{\hat{\rho}} \circ h_C^{-1}$.  Similarly, let $\mathcal{AB}: MP(N) \to \mathcal{T}(\partial N)$ be the Ahlfors-Bers parameterization of $MP(N)$.

For any $(\sigma, w) \in B$, we showed in the definition of $\Phi$ that $(f_{\sigma, w})_*$ is an isomorphism, which implies $f_{\sigma, w}$ is homotopic to a homeomorphism \cite{Waldhausen}.  Thus the $\cup h_C(\beta_i)$-Dehn filling of $(\hat{N}_C, \hat{P}_C)$ is homeomorphic to $N$, where $\cup h_C(\beta_i)$ refers to the collection of filling slopes corresponding to $\beta_1, \ldots, \beta_d$ under the homotopy equivalence $h_C : (\hat{N},\hat{P}) \to (\hat{N}_C, \hat{P}_C)$.  This filling gives us an inclusion $i_C : (\hat{N}_C, \hat{P}_C) \to N$ which defines a homeomorphism $i_C : (\partial \hat{N}_C - \hat{P}_C) \to \partial N$.  Using this homeomorphism, we can identify $\mathcal{T}(\partial \hat{N}_C - \hat{P}_C)$ with $\mathcal{T}(\partial N) \cong \mathcal{T}(S) \times \mathcal{T}(S)$.  

With this identification of the Teichm\"uller spaces of $(\partial \hat{N}_C - \hat{P}_C)$ and $\partial N$, it follows that $\Phi(\sigma, w) = \mathcal{AB}^{-1} \circ \widehat{\mathcal{AB}}_C (\sigma_w)$ for any $(\sigma, w) \in B$ since the filling map $\phi_{\sigma, w}$ extends to a conformal map from the conformal boundary of $\hat{M}_{\sigma, w}$ to the conformal boundary of $M_{\sigma, w}$.  Since the Ahlfors-Bers maps are homeomorphisms, this shows that $\Phi$ is continuous on the component $B$ of $\left( \mathcal{A}_K - \{ (\sigma, w)  \; : \; w= (\infty, \ldots, \infty) \} \right)$ containing $(\sigma_0, w_0)$.  Since $(\sigma_0, w_0)$ was arbitrary, we have that $\Phi$ is continuous on all of $\left( \mathcal{A}_K - \{ (\sigma, w)  \; : \; w= (\infty, \ldots, \infty) \} \right)$.

Next, we show $\Phi$ is continuous at points where $w = (\infty, \ldots, \infty)$.  Suppose 
$$(\sigma_i, w_{1,i}, \ldots, w_{d,i}) \to (\sigma,\infty, \ldots, \infty).$$
We claim that $\Phi(\sigma_i, w_{1,i}, \ldots, w_{d,i}) \to \Phi(\sigma,\infty, \ldots, \infty) = \sigma$.  If $(w_{1,i}, \ldots, w_{d,i}) = (\infty, \ldots, \infty)$ for all $i$, then clearly $\Phi(\sigma_i, w_{1,i}, \ldots, w_{d,i}) = \sigma_i \to \sigma$.  

Now suppose that $(w_{1,i}, \ldots, w_{d,i}) \neq (\infty, \ldots, \infty)$ for all $i$.  Let $(M_\sigma, f_\sigma)$ be the marked hyperbolic $3$-manifold corresponding to $\sigma$.  Again, assume that $f_\sigma$ is smooth.  Since $\sigma_i \to \sigma$, there is a sequence $L_i \to 1$ and smooth homotopy equivalences $g_i : M_{\sigma} \to M_{\sigma_i}$ such that $(g_i \circ f_\sigma)_* = \sigma_i$ and $g_i$ is an $L_i$-biLipschitz local diffeomorphism on a compact core of $M_\sigma$ (\textit{i.e.}, the maps $g_i$ converge to a local isometry).  If we let $f_{\sigma_i} = g_i \circ f_\sigma$, then the pullback metrics on $N$ via $f_{\sigma_i}$ converge to the pullback metrics on $N$ via $f_{\sigma}$. See p. 154 of \cite{BBES} for this geometric definition of algebraic convergence (see also p. 43 of \cite{McMullen book}).

Recall that by definition, $\Phi(\sigma_i, w_i) = (M_{\sigma,w_i}, f_{\sigma_i,w_i})$ where
$f_{\sigma_i,w_i} = \phi_{\sigma_i,w_i} \circ   \pi_{\sigma_i,w_i} \circ f_{\sigma_i}.$ (Notation: we abbreviate the $d$-tuple $(w_{1,i}, \ldots, w_{d,i})$ by $w_i$ and use a double subscript to refer to the $j$th entry, $w_{j,i}$, of $w_i$.)

Since each $w_{j,i} \to \infty$, we can find a sequence $J_i \to 1$ such that $\phi_{\sigma_i,w_i}$ is $J_i$-biLipschitz away from the $\epsilon_3$-neighborhood of the cusps of $\hat{M}_{\sigma_i,w_i}$. In particular, $\phi_{\sigma_i,w_i}$ is $J_i$-biLipschitz on $\pi_{\sigma_i,w_i}(f_{\sigma_i} (N))$.      

It follows that the limit of the pullback metrics on $N$ via the maps $f_{\sigma_i,w_i} :N \to M_{\sigma_i,w_i}$ is the same as the limit of the pullback metrics on $N$ via $\pi_{\sigma_i, w_i} \circ f_{\sigma_i}$ since $f_{\sigma_i, w_i} = \phi_{\sigma_i, w_i} \circ   \pi_{\sigma_i, w_i} \circ f_{\sigma_i}$.  The covering map is a local isometry so this limit is the limit of the pullback metrics on $N$ under $f_{\sigma_i}$.  Since $\sigma_i \to \sigma$, the limit of the pullback metrics on $N$ via $f_{\sigma_i}$ is the pullback metric on $N$ via $f_\sigma$.  To summarize, the limit of the pullback metrics on $N$ via the maps $f_{\sigma_i,w_i} :N \to M_{\sigma_i,w_i}$ is the pullback metric on $N$ via $f_{\sigma} : N \to M_\sigma$.  This convergence of metrics implies that $(f_{\sigma_i, w_i})_*$ converges to $\sigma$ in $AH(N)$ \cite{BBES}. 
\end{proof}

\textit{Remark.} 
 Lemma \ref{phi continuity} is essentially the same as Proposition 3.7 in \cite{B3}, but when there are multiple cusps we need to use the multiple cusp version of the filling theorem (Corollary \ref{multiple fillings}) which requires all of the $w$-coordinates to go to infinity.  This is the why we have defined $\mathcal{A}$ to exclude points $(\sigma, w_1, \ldots, w_d)$ where some but not all of the $w$-coordinates are $\infty$.

\subsection{An Inverse to $\Phi$}  \label{subsection: inverse to phi}

We now construct a map $\Psi$ from a subset of $MP(N) \cup MP_0(N,P)$ to $\mathcal{A}$.  For any $\sigma \in MP_0(N,P)$ and any sufficiently small neighborhood of $\sigma$ in $MP(N) \cup MP_0(N,P)$, $\Psi$ will be an inverse to $\Phi$.

Fix a representation $\sigma_0 \in MP_0(N,P)$. For $\rho$ in some neighborhood of $\sigma_0$, the definition of $\Psi$ will have two coordinates $\Psi(\rho ) = (\xi(\rho), q(\rho)) \in MP_0(N,P) \times \hat{\mathbb{C}}^d$.  We will actually begin by defining a neighborhood $V'$ of $\sigma_0$ such that for $\rho \in V'$, $\xi(\rho) \in AH(N,P)$. We will then restrict to a smaller neighborhood $V$ such that $\xi(V) \subset MP_0(N,P)$ and $(\xi(\rho), q(\rho)) \in \mathcal{A}$.   Before defining this neighborhood of $\sigma_0$ on which $\Psi$ is defined, we set up some notation and background.


Let $\mathcal{H}(N)$ denote the space of smooth, hyperbolic metrics on $N$ with the $C^\infty$-topology (see I.1.1 of \cite{CEG} for the definition of a $(PSL(2, \mathbb{C}), \mathbb{H}^3)$-structure on a manifold with boundary, and I.1.5 for a description of the space $\mathcal{H}(N)$ which is denoted $\Omega(N)$ in \cite{CEG}).   If we let $\mathcal{D}(N)$ be the space of smooth developing maps $\tilde{N}\to \mathbb{H}^3$ with the compact-$C^\infty$ topology, then $\mathcal{H}(N)$ is the quotient of $\mathcal{D}(N)$ by $PSL(2, \mathbb{C})$ acting by postcomposition.  Note that $\mathcal{H}(N)$ is still infinite dimensional since we are not identifying developing maps that differ by the lift of an isotopy.  Let $H: \mathcal{H}(N) \to AH(N)$ be the holonomy map.  Theorem I.1.7.1 of \cite{CEG} locally describes $\mathcal{H}(N)$. See Chapter I of \cite{CEG} for more details. 

\begin{thm}[Canary-Epstein-Green \cite{CEG}]  \label{neighborhood is product} Let $N_{th}$ be a thickening of $N$ (i.e., the union of $N$ with a collar $\partial N  \times I$).  Let $D_0 : \tilde{N}_{th} \to \mathbb{H}^3$ be a fixed developing map. A small neighborhood of $D_0 \vert_{\tilde{N}}$ in $\mathcal{D}(N)$ is homeomorphic to $X \times Y$ where $X$ is a small neighborhood of the obvious inclusion $N \subset N_{th}$ in the space of locally flat embeddings, and $Y$ is a neighborhood of the holonomy map $H(D_0)$ in $Hom(\pi_1(N), PSL(2, \mathbb{C}))$.  A small neighborhood of $D_0$ in $\mathcal{H}(N)$ is homeomorphic to $X \times Z$ where $Z$ is a small neighborhood of the conjugacy class of $H(D_0)$ in $R(N)$.
\end{thm}

We now let $V'$ be a neighborhood of $\sigma_0 \in V' \subset MP(N) \cup MP_0(N,P)$ that satisfies the properties (1)-(4) given below.  Roughly, $V'$ is a neighborhood on which we can define a  section $\varsigma : V' \subset AH(N) \to \mathcal{H}(N)$ and such that if $\rho \in V'$ then the length of $\rho(\gamma_i)$ is short in $M_\rho$.  The existence of such a neighborhood follows from the arguments given in Section 3.1 of \cite{B3} and Theorem \ref{neighborhood is product}, although we include justification for why we can define $V'$ with these properties after the statement of each property.


 Fix a smooth embedding $s_{\sigma_0} : N \to M_{\sigma_0}$ such that $(s_{\sigma_0})_* = \sigma_0$.   Let $g_{\sigma_0}$ be the pullback of the hyperbolic metric on $M_{\sigma_0}$.  We can choose $s_{\sigma_0}$ so that the core curves of the annuli in $P$ (\textit{i.e.}, the curves $\gamma_i \times \{1\}$) have length less than $\epsilon_3/4$ in the $g_{\sigma_0}$ metric. 

\textbf{(1)}  There exists a continuous section $\varsigma :  V' \to \mathcal{H}(N)$ to the holonomy map such that $\varsigma(\sigma_0) = g_{\sigma_0}$.  

The existence is given by Theorem \ref{neighborhood is product}.  For any $\rho \in V'$, define $g_\rho = \varsigma(\rho)$.  We emphasize that, by the definition of a section, $H(g_\rho) = \rho$.    

\textbf{(2)}   For any $\rho_1, \rho_2 \in V'$, the identity map
\[
(N, g_{\rho_1}) \xrightarrow{id} (N, g_{\rho_2})
\] is $2$-biLipschitz.  

This follows from the continuity of $\varsigma$ and the topology on $\mathcal{H}(N)$.

\textbf{(3)} For any $\rho \in V'$, there is a locally isometric immersion $s_\rho : (N, g_\rho) \to M_\rho$ where $M_\rho = \mathbb{H}^3/\rho(\pi_1(N))$ is equipped with the hyperbolic metric, such that $(s_\rho)_* = \rho$.   Moreover, there is some $\epsilon_3 > \epsilon_0 > 0$  such that $s_\rho(N)$ is contained in the $\epsilon_0$-thick part of $M_\rho$. 

The existence of $s_\rho : N \to M_\rho$ with $(s_\rho)_* = \rho$ is given by Theorem \ref{neighborhood is product}.  We now find $\epsilon_0$.  There is some $K$ such that for any point $x \in (N,g_{\sigma_0})$, there are loops $\alpha, \beta$ based at $x$ of length less than $K$ such that the group generated by $\alpha$ and $\beta$ is not virtually abelian.  For example, one can find a point $x_0$ and loops $\alpha_0$ and $\beta_0$ based at $x_0$ that generate a free group, and then let $K$ be larger than the sum of the diameter of $(N, g_{\sigma_0})$ and the lengths of $\alpha_0$ and $\beta_0$.  Since for any $\rho \in V'$, $(N, g_{\sigma_0}) \xrightarrow{id} (N, g_\rho)$ is $2$-biLipschitz by (2), at any point $x \in (N, g_\rho)$ there are loops based at $x$ of length less than $2K$ generating a free group. There exists some $\epsilon_3 > \epsilon_0 > 0$ such that for any component $\mathbb{T}_{\epsilon_0}$ of the $\epsilon_0$-thin part of any hyperbolic manifold $M$, the distance between $\partial \mathbb{T}_{\epsilon_0}$ and $\partial \mathbb{T}_{\epsilon_3}$ is at least $K$. 

Suppose $s_\rho(x) \in s_\rho(N) \cap M_\rho^{\leq \epsilon_0}$ for some $x$. Then since $s_\rho$ is a homotopy equivalence, there are loops based at $s_\rho(x)$ that generate a free group and therefore must leave the $\epsilon_3$-thin part of $M_\rho$; however, to do so they must have length greater than $2K$ contradicting that $s_\rho$ is a locally isometric immersion.  Thus there exists some $\epsilon_0$ such that $s_\rho(N)$ is contained in $\epsilon_0$-thick part of $M_\rho$ for all $\rho \in V'$.

\textbf{(4)}  Let $\epsilon_0$ be the constant in property (3).  Let $l_0$ be the constant from the drilling theorem such that the drilling map is a biLipschitz diffeomorphism outside an $\epsilon_0$-Margulis tube about the drilling. Let $l_1 = \min\{ \epsilon_0/8, l_0\}$. Then for any $\rho \in V'$ we have the length of $\gamma_i$ in $M_\rho$ is less than $l_1$, for each $i = 1, \ldots, d$.  

\bigskip \noindent \textit{Notation.} Here, the length of $\gamma_i$ in $M_\rho$ is really the length of the unique geodesic representative of $s_\rho(\gamma_i)$ in $M_\rho$. For the remainder of this section, we distinguish this geodesic representative by $s_\rho(\gamma_i)^*$.  This curve is homotopic to $s_\rho(\gamma_i \times \{t\})$ for any $t$, but its length is less than or equal to the length of $s_\rho(\gamma_i \times \{t\})$. We make this distinction since we will also be using the length of $s_\rho(\gamma_i \times \{t\})$, which is the length of $\gamma_i \times \{t\} \subset (N, g_\rho)$. 
\bigskip


Now we construct the map $\xi$ which will be the first coordinate of $\Psi$.  If $\rho \in V' \cap MP_0(N,P)$, then set $\xi(\rho) = \rho$.  Otherwise $\rho \in V' \cap MP(N)$ so let $(M_\rho, s_\rho)$ be the associated marked hyperbolic 3-manifold.  Note that by properties (1) and (3) of the neighborhood $V'$ we can use $s_\rho : N \to M_\rho$ to mark $M_\rho$.

By property (4), the length of each $s_\rho(\gamma_i)^*$ will be short in $M_\rho$ so we can drill out $s_\rho(\gamma)^* =s_\rho(\gamma_1)^* \cup \cdots \cup s_\rho(\gamma_d)^*$ and get a hyperbolic manifold $\hat{M}_\rho$. Let 
\[
\psi_\rho : M_\rho - \mathbb{T}_{\epsilon_0} (s_\rho(\gamma)^*) \to \hat{M}_\rho - \mathbb{T}_{\epsilon_0} (T)
\] be the inverse of the map $\phi$ from the drilling theorem (Theorem \ref{drilling}).  Let $\overline{M}_\rho$ be the cover of $\hat{M}_\rho$ associated to $(\psi_\rho \circ s_\rho)_* (\pi_1(N))$. Let $\overline{f}_\rho : N \to \overline{M}_\rho$ be the lift of $\psi_\rho \circ s_\rho :N \to \hat{M}_\rho$.  
Note that $s_\rho(N)$ will be contained in $M_\rho - \mathbb{T}_{\epsilon_0}(s_\rho(\gamma)^*)$ by (3), so it makes sense to compose with $\psi_\rho$.   In Lemma 3.3 of \cite{B3}, Bromberg shows that
the representation $(\overline{f}_\rho)_* : \pi_1(N) \to \pi_1(\overline{M}_\rho) \subset PSL(2, \mathbb{C})$ is in $AH(N,P)$.
See also Lemma 5.11 in \cite{MagidThesis} for a proof of the same fact with using this notation.

Thus we can define $\xi$ by 
\[
\xi(\rho) = \begin{cases} (\overline{f}_\rho)_*  \quad &\text{if} \; \rho \in MP(N) \\ \rho \quad &\text{if} \; \rho \in MP_0(N,P). \end{cases}
\]

The following Lemma is a restatement of Lemma 3.4 (see also Lemma 5.12 in \cite{MagidThesis}). 

\begin{lem}  \label{xi continuity}
The map $\xi$ is continuous at all points in $V' \cap MP_0(N,P)$.  
 \end{lem}

Since $\xi: V' \to AH(N,P)$ is continuous at points in $V' \cap MP_0(N,P)$, and $MP_0(N,P)$ is an open subset of $AH(N,P)$, we can restrict $\xi$ to a smaller neighborhood $\sigma_0 \in V \subset V' \subset MP(N) \cup MP_0(N,P)$ so that its image is contained in $MP_0(N,P)$ (Corollary 3.5 of \cite{B3}, Corollary 5.13 of \cite{MagidThesis}).  This allows us to use $\xi(\rho)$ as the first coordinate of $\Psi(\rho)$ in the definition of $\Psi$.  We now consider the second coordinate $q(\rho)$.

If $\rho \in V \cap MP_0(N,P)$, then we set $q(\rho) = (\infty, \ldots, \infty)$.  Otherwise, we consider the covering $\pi_\rho : \overline{M}_\rho \to \hat{M}_\rho$ induced by the image of the injection $(\psi_\rho \circ s_\rho)_* : \pi_1(N) \to \pi_1(\hat{M}_\rho)$.  The group $\pi_1(\hat{M}_\rho)$ is obtained from $\pi_1(\overline{M}_\rho)$ by the same construction described in Section \ref{subsection: model space}. That is, $\xi(\rho) = (\overline{M}_\rho, \overline{f}_\rho)$ corresponds to some representation $\sigma \in MP_0(N,P)$ and there is a unique $(w_1, \ldots, w_d)$ such that the extension $\sigma_{w_1, \ldots, w_d} (\pi_1(\hat{N})) = \pi_1(\hat{M}_\rho)$.    We define this to be $q(\rho) = (w_1, \ldots, w_d)$.  

Equivalently, $w_i$ is defined so that if we conjugate $(\psi_\rho \circ s_\rho)_*$ so that $\gamma_i$ is mapped to $\begin{pmatrix} 1 & 2 \\ 0 & 1 \end{pmatrix}$, then the unique nontrivial element $\beta_i \in  \pi_1(\partial U_i) \subset \pi_1(\hat{M}_\rho)$ that bounds a disk in $M_\rho$ will be $\begin{pmatrix} 1 & w_i \\ 0 & 1 \end{pmatrix}$.   

Now that we have defined $q(\rho)$, we can define $\Psi : V \to MP_0(N,P) \times \hat{\mathbb{C}}^d$ by $$\Psi(\rho) = (\xi(\rho), q(\rho))$$ for any $\rho \in V$.  Note that we have defined $q(\rho)$ so that $\Psi(\rho) \in \mathcal{A}$ for all $\rho \in V$.  Unlike $\Phi$, we only show $\Psi$ is continuous for points on the boundary of $MP(N)$.  Although Lemma \ref{continuity2} is nearly identical to Proposition 3.8 of \cite{B3}, we include a proof since this is one of instances where we must keep track of multiple cusps and therefore our setup differs from Bromberg's.   

\begin{lem}  
\label{continuity2} The map $\Psi$ is continuous on $V \cap MP_0(N,P)$. 
\end{lem}

\begin{proof}  Lemma \ref{xi continuity} shows that $\xi$ is continuous on $V \cap MP_0(N,P)$.  Now consider a sequence $\rho_i \to \sigma$ where $\sigma \in MP_0(N,P)$.  Since $\Psi(\sigma) = (\sigma ,\infty, \ldots,  \infty)$ and we know $\xi(\rho_i) = \sigma$, it suffices to show that $q(\rho_i) \to (\infty, \ldots, \infty)$.  If $\rho_i \in MP_0(N,P)$ then $q(\rho_i) = (\infty, \ldots, \infty)$ so assume that $\rho_i \in V \cap MP(N)$.  We will use the notation $q(\rho_i) = (w_{1,i}, \ldots, w_{d,i})$ and show $w_{j,i} \to \infty$ as $i \to \infty$ for $j = 1, \ldots, d$.  

Since for each $j = 1, \ldots, d$, the length of $\rho_i(\gamma_j) \to 0$ as $i \to \infty$, Proposition \ref{normalized length corollary} shows that the normalized length of the $j$th meridian, $\beta_j$, goes to infinity as $i$ goes to infinity.  The normalized length is given by 
\[
\frac{|w_{j,i}|}{\sqrt{2Im(w_{j,i})}}
\]   If $w_{j,i}$ does not go to $\infty$, then we must have $Im(w_{j,i}) \to 0$ as $i \to \infty$.  This cannot happen by Corollary \ref{bound imaginary part}.  

It follows that $q(\rho_i) \to  (\infty, \ldots, \infty) = q(\sigma)$ proving $q$ is continuous at any point $\sigma \in V\cap  MP_0(N,P)$.  Thus, $\Psi$ is continuous on $V\cap MP_0(N,P)$. 
\end{proof}

\subsection{Local Homeomorphism}  \label{subsection: local homeomorphism}

Recall, in Section \ref{subsection: definition of phi} we defined $\Phi$ on a subset $\mathcal{A}_K \subset \mathcal{A}$ and showed $\Phi$ is continuous.  
We now claim that there is some subset of $\mathcal{A}_K$ on which 
$\Phi$ is continuous and injective.  See also Propositions 3.9 and 3.10 of \cite{B3}.

\begin{lem} \label{inverses}   Let $\sigma_0 \in MP_0(N,P)$.  There is some neighborhood $U$ of $(\sigma_0, \infty, \ldots, \infty)$ in $\mathcal{A}$ such that $\Psi \circ \Phi \vert_U = id$.  In particular, $\Phi$ is injective on $U$. 
\end{lem}

\begin{proof} Let $V$ be the neighborhood of $\sigma_0$ on which $\Psi$ was defined.  By the continuity of $\Phi$, we can find a neighborhood $U'$ so that $\Phi(U')$ is contained in $V$, and for any $(\sigma, w_1, \ldots, w_d) \in U'$, $\sigma \in V$.  We now consider $\Psi \circ \Phi \vert_{U'}$.

Let $(\sigma, w_1, \ldots, w_d) \in U'$.  If $(w_1, \ldots, w_d) = (\infty,\ldots, \infty)$ then $\Phi(\sigma, \infty, \ldots, \infty) = \sigma$ and $\Psi(\sigma) = (\sigma, \infty, \ldots, \infty)$.  If $(w_1, \ldots, w_d) \neq (\infty,\ldots, \infty)$ then $\sigma_{w_1, \ldots, w_d} \in MP(\hat{N}, \hat{P})$.  Recall that the definition of $\Phi$ in this case was 
\[
\Phi(\sigma, w_1, \ldots, w_d)= (M_{\sigma, w}, f_{\sigma, w})
\] where $M_{\sigma, w}$ was the filling of $\hat{M}_{\sigma, w}= \mathbb{H}^3/\sigma_{w_1, \ldots, w_d}  (\pi_1(\hat{N}))$ and $f_{\sigma, w} = \phi_{\sigma, w} \circ \pi_{\sigma, w} \circ f_\sigma$.  Since $\sigma \in V$, the marking $f_\sigma$ is homotopic to a local isometry $s_\sigma : (N, g_\sigma) \to M_\sigma$ such that $s_\sigma(N) \subset M_{\sigma}^{\geq \epsilon_0}$ (see the four properties of the neighborhood $V'$ defined in Section \ref{subsection: inverse to phi}). Thus we can redefine the marking $f_{\sigma, w} = \phi_{\sigma, w} \circ \pi_{\sigma, w} \circ s_\sigma$ without changing the definition of $\Phi(\sigma, w)$.  Also recall that 
$\pi_{\sigma, w}$ is a covering map and therefore a local isometry, and 
   \[
   \phi_{\sigma, w} : \hat{M}_{\sigma, w} - \mathbb{T}_{\epsilon_0} (T) \to M_{\sigma, w} - \mathbb{T}_{\epsilon_0}(\gamma)
   \]
is a biLipschitz diffeomorphism. (Recall that $\Phi$ was originally defined on $\mathcal{A}_K$ so that for any $(\sigma, w) \in \mathcal{A}_K$, $\phi_{\sigma, w}$ is a biLipschitz diffeomorphism on $\hat{M}_{\sigma, w} - \mathbb{T}_{\epsilon_3} (T)$.  By possibly making $U'$ smaller, we can assume that $\phi_{\sigma, w}$ is biLipschitz on the $\hat{M}_{\sigma, w} - \mathbb{T}_{\epsilon_0} (T)$.)  Thus $f_{\sigma, w} = \phi_{\sigma, w} \circ \pi_{\sigma, w} \circ s_\sigma$ is smooth, and we let $g_{\sigma, w}'$ be the pullback metric on $N$ via $f_{\sigma, w} : N \to M_{\sigma, w}$.

By the assumption that $\Phi(U') \subset V$, we can find a homotopic marking $s_{\sigma, w} \simeq f_{\sigma, w}: N \to M_{\sigma, w}$ satisfying the properties (1)-(4) listed prior to the definition of $\Psi$.  However, we need that $f_{\sigma, w}$ is homotopic to $s_{\sigma, w}$ in $M_{\sigma, w} - \gamma$ in order to have the drilling construction in $\Psi$ be the inverse to the filling construction in $\Phi$.

    Let $W$ be a neighborhood of $g_{\sigma_0}$ in $H^{-1}(V) \subset \mathcal{H}(N)$ such that Theorem \ref{neighborhood is product} applies.  We first claim there is some $U \subset U' \subset \mathcal{A}$ such that if $(\sigma, w) \in U$, then $g_{\sigma, w}' \in W$.  There is some $J$ such that if $\phi_{\sigma, w}$ is a $J$-biLipschitz diffeomorphism and $\sigma$ is sufficiently close to $\sigma_0$, then $g_{\sigma, w}' \in W$.  So, let $U\subset U'$ be a neighborhood of $(\sigma_0, \infty, \ldots, \infty)$ in $\mathcal{A}$ such that for all $(\sigma, w_1, \ldots, w_d) \in U$, $\sigma$ is sufficiently close to $\sigma_0$, and for all $i$, $|w_i|$ is large enough so that the filling map $\phi_{\sigma, w}$ is a $J$-biLipschitz diffeomorphism. 
 
 Next we claim that there is a metric $g_{\sigma, w} \in \varsigma(V) \subset W$ and a locally isometry $s_{\sigma, w} : (N, g_{\sigma, w}) \to M_{\sigma, w}$ such that $f_{\sigma, w}$ is homotopic to $s_{\sigma, w}$ as a map into $M_{\sigma, w} -\gamma$.  This claim follows from product structure of $W$ described in Theorem \ref{neighborhood is product}.   More precisely, let $N_{th}$ be a thickening of $N$.  Then we can extend the local isometry $f_{\sigma, w} : (N, g_{\sigma, w}') \to M_{\sigma, w}$ to a local isometry $f_{\sigma, w, th} : (N_{th}, g_{\sigma, w, th}') \to M_{\sigma, w}$, where $g_{\sigma, w, th}'$ is a hyperbolic metric on $N_{th}$ that restricts to $g_{\sigma, w}$ on $N\subset N_{th}$.  Then there exists a locally flat embedding $i : N \to N_{th}$ isotopic to the identity such that $s_{\sigma, w}  = f_{\sigma, w, th} \circ i$.  Thus $s_{\sigma, w}$ and $f_{\sigma, w}$ are homotopic as maps inside $f_{\sigma, w, th} (N_{th}) \subset M_{\sigma}$.  Since $f_{\sigma, w}(N) \subset M_{\sigma, w} - \mathbb{T}_{\epsilon_0} (\gamma)$, we can assume that the neighborhood $W$ in $\mathcal{H}(N)$ is small enough so that $f_{\sigma, w, th}(N_{th}) \subset M_{\sigma, w} - \gamma$. Thus, $f_{\sigma, w}$ and $s_{\sigma, w}$ are homotopic in $M_{\sigma, w} - \gamma$.

It now follows from the definitions that 
\[
\Psi(\Phi(\sigma, w_1, \ldots, w_d)) = \Psi(M_{\sigma, w},f_{\sigma, w}) = \Psi(M_{\sigma, w},s_{\sigma, w}) = (\sigma, w_1, \ldots, w_d)
\]
 See Lemma 5.15 in \cite{MagidThesis}. 
\end{proof}

\begin{lem}  \label{inverses2} Let $\rho \in MP(N) \cup MP_0(N,P)$.  If $\Phi \circ \Psi$ is defined at $\rho$ then $\Phi \circ \Psi (\rho) = \rho$. 
\end{lem}
\begin{proof}  If $\rho \in V \cap MP_0(N,P)$ then clearly $\Psi(\rho) = (\rho, \infty, \ldots, \infty)$ and $\Phi(\Psi(\rho)) = \rho$.  If $\rho \in V \cap MP(N)$, then recall that we can choose the marking $s_\rho : N \to M_\rho$ and define $\hat{M}_\rho$ to be the $\gamma$-drilling of $M_\rho$.  Then we let $\overline{M}_\rho$ be the cover of $\hat{M}_\rho$ associated to $(\psi_\rho \circ s_\rho)_*(\pi_1(N))$.  If $\Psi(\rho) = (\sigma,w_1, \ldots, w_d)$ then $\overline{M}_\rho = M_\sigma$, $\overline{f}_\rho  \simeq f_\sigma$, $\hat{M}_{\rho} = \hat{M}_{\sigma,w}$, and $M_\rho = M_{\sigma,w}$.  (To see why $\hat{M}_\rho = \hat{M}_{\sigma, w}$ see Proposition 3.10 and Section 6 of \cite{B3}.  It only follows from the definitions that $\hat{M}_{\sigma , w}$ covers $\hat{M}_{\rho}$; however, Bromberg proves this cover is trivial.)  Thus 
\[
\pi_{\sigma,w} \circ f_\sigma \simeq \pi_{\rho} \circ \overline{f}_\rho = \psi_\rho \circ s_\rho
\] since $\overline{f}_\rho$ was the lift of $\psi_\rho \circ s_\rho$.  But then 
\[
f_{\sigma, w} = \phi_{\sigma,w} \circ \pi_{\sigma,w} \circ f_\sigma \simeq \phi_{\sigma,w} \circ \psi_\rho \circ s_\rho = \psi_\rho^{-1} \circ \psi_\rho \circ s_\rho  =s_\rho.
\]   It follows that when we apply $\Phi$ to $(\sigma,w_1, \ldots, w_d)$ we get $(M_{\sigma,w}, f_{\sigma,w}) = (M_\rho, s_\rho) = \rho$. 
\end{proof}

\begin{thm} \label{loc homeo} Let $\sigma_0 \in MP_0(N,P)$.  
The map $\Phi$ is a local homeomorphism from $\mathcal{A}_K$ to $MP(N) \cup MP_0(N,P)$ at $(\sigma_0, \infty, \ldots, \infty)$. 
\end{thm}

\begin{proof} It follows from Lemma \ref{phi continuity} that $\Phi$ is continuous and from Lemma \ref{inverses} that $\Phi$ is injective on some neighborhood $U$ of $(\sigma_0, \infty, \ldots, \infty)$.  

Certainly $\Phi(U)$ contains $\sigma_0$.  We claim that $\Phi(U)$ contains some neighborhood $V$ of $\sigma_0$ in $MP(N) \cup MP_0(N,P)$.    Suppose no such neighborhood exists.  Then we can find a nested sequence of neighborhoods $V_i$ whose intersection is $\sigma_0$ and a sequence $\rho_i \in V_i$ such that $\rho_i \notin \Phi(U)$.  Since $\rho_i \to \sigma_0$, and Lemma \ref{continuity2} says that $\Psi$ is continuous at $\sigma_0$ we have $\Psi(\rho_i) \to \Psi(\sigma_0) = (\sigma_0, \infty, \ldots, \infty)$.  It follows that $\Psi(\rho_i) \in U$ for all sufficiently large $i$; however, this contradicts Lemma \ref{inverses2} which says that $\Phi(\Psi(\rho_i)) = \rho_i \notin \Phi(U)$ for sufficiently large $i$. 

Hence, there is some neighborhood $V$ of $\sigma_0$ contained in $\Phi(U)$.  Since $\Phi$ is continuous, $\Phi^{-1}(V)$ is a neighborhood of $(\sigma_0, \infty, \ldots, \infty)$ in $\mathcal{A}$ such that $\Phi \vert_{\Phi^{-1}(V)} : \Phi^{-1}(V) \to V$ is a continuous bijection.  The inverse map is given by $\Psi$, which is continuous on $V \cap MP_0(N,P)$ by Lemma \ref{continuity2} and on $V \cap MP(N)$ by invariance of domain. Hence $\Phi$ is a local homeomorphism at $\sigma_0$.   
\end{proof}

\bigskip \noindent \textit{Remark.}  Since the point $\sigma_0 \in MP_0(N,P)$ that we fixed in the beginning of Section \ref{subsection: inverse to phi} and used throughout Sections \ref{subsection: inverse to phi} and \ref{subsection: local homeomorphism} was arbitrary, we have actually shown that $\Phi$ is a local homeomorphism at any $\sigma \in MP_0(N,P)$. 
\bigskip

\subsection{$MP(N) \cup MP_0(N,P)$ is not locally connected} \label{not locally connected section}

In Proposition \ref{A not loc conn}, we saw that there was a point $\sigma_0 \in MP_0(N,P)$ such that $\mathcal{A}$ is not locally connected at $(\sigma_0, \infty, \ldots, \infty)$.  By Theorem \ref{loc homeo}, $\Phi$ is a local homeomorphism from $\mathcal{A}$ to $MP(N) \cup MP_0(N,P)$ at $(\sigma_0, \infty, \ldots, \infty)$.  Hence, $MP(N) \cup MP_0(N,P)$ is not locally connected at $\Phi(\sigma_0, \infty, \ldots, \infty)= \sigma_0 \in MP_0(N,P)$.   Thus we have shown

\begin{thm} \label{sigma_0} There exists $\sigma_0 \in MP_0(N,P)$ such that $MP(N) \cup MP_0(N,P)$ is not locally connected at $\sigma_0$.  
\end{thm}

By the Density Theorem (Theorem \ref{density theorem}), $AH(N)$ is the closure of $MP(N) \cup MP_0(N,P)$.  Of course, it does not follow directly from this that $AH(N)$ is not locally connected at $\sigma_0$.  In order to conclude anything about the closure, we need more quantitative control over the components of a neighborhood $U$ of $(\sigma_0, \infty, \ldots, \infty)$ in $\mathcal{A}$, and what happens to these components under the map $\Phi$.  By Lemma \ref{lowerboundw}, there is lower bound to the distance between some of the components of $U$.  In the next Section, we will use the filling theorem to show that this implies $\overline{\Phi(U)}$ has infinitely many components.

\section{$AH(S \times I)$ is not locally connected} \label{closure not locally connected section}

In this section, we prove Theorem \ref{main theorem} by contradiction. If one assumes $AH(S \times I)$ is locally connected, then one may use the filling theorem (Theorem \ref{filling}) and Lemma \ref{lowerboundw} to derive a contradiction.  Recall that for a point $(\sigma, w_1, \ldots, w_d) \in \mathcal{A}$ at which $\Phi$ is defined, one obtains a hyperbolic manifold $M_{\sigma, w}$ by filling the $d$ cusps of $\hat{M}_{\sigma, w} =  \mathbb{H}^3/\sigma_{w}(\pi_1(\hat{N}))$. The manifold $M_{\sigma, w}$ does not depend on the order in which the cusps are filled so we can fill the first cusp last.  Let $M'_{\sigma, w}$ denote the manifold with a single rank-2 cusp obtained by filling all but the first cusp of $\hat{M}_{\sigma, w}$.  Equivalently, $M'_{\sigma, w}$ is the $\gamma_1$-drilling of $\Phi(\sigma, w_1, \ldots, w_d) \in MP(N)$.  Eventually we will use the $w_1$ coordinate to estimate the complex length of $\gamma_1$ in $\Phi(\sigma, w_1, \ldots, w_d)$.  As an intermediate step, Lemma \ref{first approximation} bounds the change in the geometry of the first cusp while we perform the other $d-1$ fillings.

Let $q_1$ be the first coordinate of the map $q$ in the definition of $\Psi$.  That is, 
\[
q_1: V \cap MP(N) \to \mathcal{T}(T^2) 
\] is defined so that if $\Psi(\eta) = (\sigma, w_1, \ldots, w_d)$, then $q_1(\eta) = w_1$.    This is a Teichm\"uller parameter for the first cusp in $\hat{M}_{\sigma, w}$ in the sense that $\sigma_w$ is conjugate to a representation that sends $\gamma_1$ to $\begin{pmatrix} 1& 2 \\ 0 & 1 \end{pmatrix}$ and the meridian $\beta_1$ of $\partial U_1$ to $\begin{pmatrix} 1 & q_1(\eta) \\ 0 & 1 \end{pmatrix}$. 

Now define $r_1: V \cap MP(N) \to \mathcal{T}(T^2)$ so that $r_1(\eta)$ is the Teichm\"uller parameter of the cusp of $M'_{\sigma, w}$.  That is, after $d-1$ cusps have been filled, we can conjugate $\pi_1(M'_{\sigma, w})$ so that the remaining cusp is marked by 
\[
\gamma_1 \mapsto \begin{pmatrix}1 & 2 \\ 0 & 1 \end{pmatrix} \quad \text{and} \quad \beta_1 \mapsto \begin{pmatrix} 1 & r_1(\eta) \\ 0 & 1 \end{pmatrix}. 
\] 

The drilling theorem can be used to show that $q_1$ and $r_1$ are close in the following sense.

\begin{lem} \label{first approximation}  Let $\delta > 0$, $\kappa > 0$. 
There is some $l_0>0$ such that for any $\eta \in MP(N)$ with $\min\{ Im(q_1(\eta)), Im(r_1(\eta)) \}<\kappa$ and
$$\sum_{i = 1}^d l(\eta(\gamma_i)) < l_0,$$
then
$$|q_1(\eta) - r_1(\eta) | < \frac{\delta}{4}.$$
\end{lem} 

\begin{proof}  For any $\eta \in MP(N)$, let $M_\eta = \mathbb{H}^3/\eta(\pi_1(N))$, let $M_{\sigma, w}'$ denote the $\gamma_1$-drilling of $M_\eta$, and let $\hat{M}_{\sigma, w}$ denote the $\cup_{i=2}^d \gamma_i$-drilling of $M_{\sigma, w}'$. 

The drilling theorem says that there exists $l_1$ such that if the length of $l(\eta(\gamma_1)) < l_1$ then there is a $2$-biLipschitz map 
\[
M_\eta - \mathbb{T}_{\epsilon_3} (\gamma_1) \to M_{\sigma, w}' - \mathbb{T}_{\epsilon_3}(T_1).
\]  
This implies the lengths of $\gamma_2, \ldots, \gamma_d$ do not double when we drill $\gamma_1$. 

Choose some $\varepsilon > 0$ such that $\varepsilon e^\varepsilon < \frac{\delta}{4\kappa}$.   There exists some $J > 1$ such that if $X_1, X_2$ are two points in $\mathcal{T}(T^2)$ and $\phi: X_1 \to X_2$ is a $J$-biLipschitz diffeomorphism, then $d_{\mathcal{T}(T^2)} (X_1, X_2) < \varepsilon$.

By the drilling theorem, there is some $l_2$ such that if $\sum_{i=2}^d l_{M_{\sigma, w}'}(\gamma_i) < l_2$ then there exists a $J$-biLipschitz diffeomorphism
 \[
\phi: M_{\sigma, w}' - \cup_{i=2}^d \mathbb{T}_{\epsilon_3} (\gamma_i) \to \hat{M}_{\sigma, w} - \cup_{i=2}^d \mathbb{T}_{\epsilon_3}(T_i).  
\] 

Now choose any $0 < l_0 < \min\{ l_1, \frac{l_2}{2}\}$. If $\sum_{i=1}^d l(\eta(\gamma_i)) < l_0$ then $l(\eta(\gamma_1))< l_0 < l_1$.  This implies the lengths of $\gamma_2, \ldots, \gamma_d$ do not double as we do the first drilling.  Thus, 
\[
\sum_{i=2}^d l_{M_{\sigma, w}'}(\gamma_i)< \sum_{i=2}^d 2 l(\eta(\gamma_i))  < 2l_0 <  l_2.  
\]  

Now since $\sum_{i=2}^d l_{M_{\sigma, w}'}(\gamma_i)< l_2$, there exists a $J$-biLipschitz diffeomorphism
 \[
\phi: M_{\sigma, w}' - \cup_{i=2}^d \mathbb{T}_{\epsilon_3} (\gamma_i) \to \hat{M}_{\sigma, w} - \cup_{i=2}^d \mathbb{T}_{\epsilon_3}(T_i)
\]  when we drill $\gamma_2, \ldots, \gamma_d$.  As in the proof of Corollary \ref{multiple fillings} (see also the remarks following Theorem \ref{drilling}), we can assume that $\phi$ restricts to a $J$-biLipschitz diffeomorphism on $T_1$ that takes torus cross-sections of the first cusp in $M_{\sigma, w}'$ to torus cross-sections of the first cusp in $\hat{M}_{\sigma, w}$ (Theorem 6.12 of \cite{BB}).    
Since the Teichm\"uller metric for $\mathcal{T}(T^2)$ agrees with the hyperbolic metric for the upper-half plane model of $\mathbb{H}^2$, this implies 
\[
d_{\mathcal{T}(T^2)} (q_1(\eta), r_1(\eta))   = d_{\mathbb{H}^2}(q_1(\eta), r_1(\eta)) < \varepsilon.
\]    
See also Theorem 7.2 of \cite{B2}.

  Since either $Im(q_1(\eta))< \kappa$ or $Im(r_1(\eta)) < \kappa$, \[
|q_1(\eta) - r_1(\eta) | < \kappa e^\varepsilon \left( d_{\mathbb{H}^2}(q_1(\eta), r_1(\eta))  \right)  < \kappa \varepsilon e^\varepsilon  <    \frac{\delta}{4}.
\]
 \end{proof}

With Lemma \ref{first approximation} providing some control on $r_1$ based on $q_1$, we are now ready to prove Theorem \ref{main theorem}, which we restate here for convenience.

\textbf{Theorem \ref{main theorem}.}  \textit{Let $S$ be a closed surface of genus $g \geq 2$.  Then $AH( S \times I)$ is not locally connected.}

\begin{proof} Let $(\sigma_0, \infty, \ldots, \infty) \in \mathcal{A}$ be the point that we described in Section \ref{A not loc conn section} where we found $\mathcal{A}$ is not locally connected.  Recall this was a point such that $\sigma_0 \vert_{\pi_1(N_{1,1})} = \sigma_{z_0}$ where $z_0$ was the point described in Lemma \ref{box lemma}.  We will show $AH(S\times I)$ is not locally connected at $\sigma_0$.  First we claim there exists a neighborhood $U$ of $(\sigma_0, \infty, \ldots, \infty)$ with the following properties:

\textbf{(1)}  There is a neighborhood $V$ of $\sigma_0$ in $MP(N) \cup MP_0(N,P)$ such that $\Phi\vert_U : U \to V$ is a homeomorphism.  Such a neighborhood exists by Theorem \ref{loc homeo}. 

\textbf{(2)} For any $(\eta, w_1, \ldots, w_d) \in U$, $\eta \vert_{\pi_1(N_{1,1})}$ lies in the neighborhood $W$ of $\sigma_0 \vert_{\pi_1(N_{1,1})}$ that we defined in Section \ref{A not loc conn section}.
Recall $W$ is a neighborhood of $\sigma_{z_0}$ in $MP_0(N_{1,1}, P_{1,1}')$ such that for all $\sigma_z \in \overline{W}$, the coordinate $z$ lies in the $\delta$-neighborhood $O$ of $z_0$ (see Lemma \ref{box lemma}).

\textbf{(3)} Recall from Section \ref{A not loc conn section} that $C_n = \{ (\eta, w_1, \ldots, w_d) \in U \; : \; w_1 \in R+2n\}$.  Then there exists $0<\delta < 1$ such that for any $(\eta, w_1, \ldots, w_d) \in \overline{C_n}$ and any $(\eta', w_1', \ldots, w_d') \in \overline{U-C_n}$, we have $|w_1-w_1'| > \delta$ for any $n$.  This follows from property (2) and Lemma \ref{lowerboundw}.  Note that the quantity $\delta$ is the radius of the neighborhood $O$ of $z_0$ from Lemma \ref{box lemma}.  

\textbf{(4)}  Let $\delta>0$ be the constant from (3).  Let $\kappa > 80(2\pi)^2$ be some constant such that for any $(\eta, w_1, \ldots, w_d) \in \overline{C_n}$, we have $Im(w_1) < \kappa-1$.  Then for any $\eta \in \Phi(U) \cap MP(N) = V \cap MP(N)$, we have $|r_1(\eta) - q_1(\eta) | < \frac{\delta}{4}$ or $\min \{ Im(r_1(\eta)) , Im(q_1(\eta)) \} \geq \kappa$. 

We can assume $U$ satisfies property (4) for the following reason.  By Lemma \ref{first approximation}, given any $\delta, \kappa>0$ there exists some $l_0$ such that if $\sum_{i=1}^d l(\eta(\gamma_i)) < l_0$ then either $|r_1(\eta) - q_1(\eta) | < \frac{\delta}{4}$ or $\min \{ Im(r_1(\eta)) , Im(q_1(\eta)) \} \geq \kappa$.  Since $\Phi\vert_U : U \to V$ is a homeomorphism, and $V$ is a neighborhood of $\sigma_0$ where $\sigma_0(\gamma_1), \ldots, \sigma_0(\gamma_d)$ are parabolic, we can make $U$ small enough so that $\sum_{i=1}^d l(\eta(\gamma_i)) < l_0$ for any $\eta \in \Phi(U) \cap MP(N)$. One can check that shrinking $U$ does not change properties (1), (2), and (3).  

Again, since making $U$ smaller does not affect the above properties, we can assume for all $(\eta, w_1, \ldots, w_d) \in U$, $|w_1| > 81(2\pi)^2$.  Since $w_1 = q_1(\Phi(\eta, w_1, \ldots, w_d))$, it follows from (3) and (4) that $U$ satisfies:

\textbf{(5)} For any $(\eta, w_1, \ldots, w_d) \in U$, $|r_1(\Phi(\eta, w_1, \ldots, w_d))| > 80(2\pi)^2$.

Now that we have set up a neighborhood $U$ of $(\sigma_0, \infty, \ldots, \infty)$ in $\mathcal{A}$, suppose $AH(N)$ was locally connected at $\Phi(\sigma_0, \infty, \ldots, \infty) = \sigma_0$.   Then we claim that, for all but finitely many $n$, 
\[
\overline{\Phi(C_n)} \cap \overline{\Phi(U-C_n)} \neq \emptyset.
\]

To prove the claim, let $V_{AH}$ be a neighborhood (in $AH(N)$) of $\sigma_0$ contained inside $\overline{\Phi(U)}$.  Note that the closure of $\Phi(U)$  contains such a neighborhood of $\sigma_0$ in $AH(N)$ by the Density Theorem (Theorem \ref{density theorem}).  If $AH(N)$ is locally connected, then there exists a connected neighborhood $\sigma_0 \in V_{conn} \subset V_{AH}$.  Recall in Section \ref{A not loc conn section} we found a sequence $(\sigma_0, u_1 + 2n, \ldots, u_d + 2n) \in C_n$ converging to $(\sigma_0, \infty, \ldots, \infty)$. Thus $V_{conn} \cap \Phi(C_n)$ and $V_{conn} \cap \Phi(U-C_n)$ are nonempty for all sufficiently large $n$.

If the closures of $\Phi(C_n)$ and $\Phi(U-C_n)$ were disjoint then we could form a separation of $V_{conn}$.   Thus we must have 
\[
\overline{\Phi(C_n)} \cap \overline{\Phi(U-C_n)} \neq \emptyset
\] for all but finitely many $n$.    

Now let 
\[
\rho \in \overline{\Phi(C_n)} \cap \overline{\Phi(U-C_n)}. 
\] for some sufficiently large $n$.  We will determine $n$ later, but for now notice that there are only finitely many $n$ for which this intersection is empty.

Although $\rho$ is not in the image of $\Phi$, we can find sequences 
\[
\rho = \lim_{m\to \infty} \eta_{m} = \lim_{m\to\infty} \eta_{m}'
\] where $\eta_{m} \in \Phi(C_n)$ and $\eta_{m}' \in \Phi(U-C_n)$ are representations in $MP(N)$.  

Up to subsequence, we can assume that $q_1(\eta_m)$ and $q_1(\eta_m')$ converge, so we define $w_1$ and $w_1'$ by 
\[
w_1 = \lim_{m\to \infty} q_1(\eta_{m}) \] and \[  w_1' = \lim_{m\to\infty} q_1(\eta'_{m}). 
\]

Equivalently, $w_1$ and $w_1'$ are the second coordinates of $\lim_{m \to \infty} \Psi(\eta_{m}) \in \overline{C_n}$ and $\lim_{m \to \infty} \Psi(\eta_{m}') \in \overline{U-C_n}$.  Note that $w_1 \in R+2n$ since $\eta_m \in \Phi(C_n)$ for all $m$. 
Also, by passing to further subsequences if necessary, we define $\zeta_1$ and $\zeta_1'$ by
\[
\zeta_1 = \lim_{m\to \infty} r_1(\eta_{m}) \] and \[  \zeta_1' = \lim_{m\to\infty} r_1(\eta'_{m}). 
\]

By property (3) of $U$, 
 there is some $0 < \delta <1$ such that
\[
|w_1 - w_1'|> \delta.
\]  

Note that $\kappa$ was chosen so that $Im(w_1) < \kappa-1$ since $w_1 = \lim q_1(\eta_m)$ and $\eta_m \in \Phi(C_n)$.  Thus by property (4) of the neighborhood $U$, we have
\[
|\zeta_1 - w_1| \leq \frac{\delta}{4}. 
\]
If we also have $\min\{ Im(w_1'),  Im( \zeta_1') \} < \kappa$, then $|\zeta_1' - w_1'| \leq \frac{\delta}{4}$ and thus 
\begin{align}\label{lower bound r}
\left| \zeta_1- \zeta_1' \right| \geq \delta - \frac{\delta}{4} - \frac{\delta}{4} =  \frac{\delta}{2}.
\end{align}
Otherwise, we have $\min\{ Im(w_1'),   Im(\zeta_1')\}  \geq \kappa$. But since $Im(w_1)< \kappa - 1$ and $|w_1 -\zeta_1| \leq \frac{\delta}{4}$, we must have  $Im(\zeta_1) < \kappa -1 + \frac{\delta}{4}$.  Thus 
\[
\left| \zeta_1 - \zeta_1' \right|  \geq \left| Im(\zeta_1) - Im(\zeta_1') \right|  >  \kappa - \left(\kappa-1 + \frac{\delta}{4}\right) = 1 - \frac{\delta}{4}  > \frac{\delta}{2}
\] so inequality (\ref{lower bound r}) still holds.

Next we will use the complex length estimates in the filling theorem to produce a contradiction to (\ref{lower bound r}).  Consider the complex length, $\mathcal{L}(\rho(\gamma_1))$.  We can estimate the complex length of $\rho(\gamma_1)$ in two ways, corresponding to each of the two sequences $\eta_{m}$ and $\eta_{m}'$.  

For any $\eta \in V$, parts $(ii)$ and $(v)$ of the filling theorem (Theorem \ref{filling}) can be used to estimate $\mathcal{L}(\eta(\gamma_1))$. If we let 
\[
L^2_{\eta} = \frac{|r_1(\eta)|^2}{2Im(r_1(\eta))} \quad \text{and} \quad A^2_{\eta} = \frac{|r_1(\eta)|^2}{2Re(r_1(\eta))},
\] then the filling theorem gives us the following estimates on $\mathcal{L}(\eta(\gamma_1)) = l(\eta(\gamma_1))+ i \theta(\eta(\gamma_1))$.
\[
\left| l(\eta(\gamma_1)) - \frac{2\pi}{L_{\eta}^2} \right| \leq   \frac{8(2\pi)^3}{L_{\eta}^4 - (16)(2\pi)^4} \quad \text{and} \quad 
\left| \theta(\eta(\gamma_1)) - \frac{2\pi}{A_{\eta}^2} \right| \leq \frac{5(2\pi)^3}{(L_{\eta}^2 - 4(2\pi)^2)^2}.
\]  
If $\rho = \lim_{m \to \infty} \eta_m = \lim_{m\to \infty} \eta_m'$, then we get the following two sets of estimates on  $\mathcal{L}(\rho(\gamma_1)) = l(\rho(\gamma_1))+ i \theta(\rho(\gamma_1))$.  
Let 
\[
L^2 = \lim_{m\to \infty} L^2_{\eta_m}, \; A^2 = \lim_{m\to \infty} A^2_{\eta_m}, \; (L')^2 = \lim_{m\to \infty} L^2_{\eta_m'}, \; (A')^2 = \lim_{m\to \infty} A^2_{\eta_m'}
\]
Then
\[
\left| l(\rho(\gamma_1)) - \frac{2\pi}{L^2} \right| \leq   \frac{8(2\pi)^3}{L^4 - (16)(2\pi)^4} \quad \text{and} \quad 
\left| \theta(\rho(\gamma_1)) - \frac{2\pi}{A^2} \right| \leq \frac{5(2\pi)^3}{(L^2 - 4(2\pi)^2)^2},
\]  and 
\[
\left|  l(\rho(\gamma_1))- \frac{2\pi}{(L')^2} \right| \leq   \frac{8(2\pi)^3}{(L')^4 - (16)(2\pi)^4} \hspace{0.11in} \text{and} \hspace{0.11in} 
\left| \theta(\rho(\gamma_1))- \frac{2\pi}{(A')^2} \right| \leq \frac{5(2\pi)^3}{((L')^2 - 4(2\pi)^2)^2}.
\] 
Recall that by property ($5$) of the neighborhood $U$, we have $|r_1(\eta_m)|, |r_1(\eta_m')| > 80(2\pi)^2$ for all $m$. So after passing to the limit, $| \zeta_1|, | \zeta_1' |\geq  80(2\pi)^2$ and therefore $L^2, (L')^2 \geq 40(2\pi)^2$. In particular, $L^2, (L')^2 > 8(2\pi)^2$, which together with the triangle inequality implies
\begin{align} \label{L2eta}
\left| \frac{2\pi}{L^2} - \frac{2\pi}{(L')^2} \right| \leq  16(2\pi)^3 \left(  \frac{1}{L^4}  +  \frac{1}{(L')^4} \right)
\end{align} and 
\begin{align}  \label{A2eta}
\left| \frac{2\pi}{A^2} - \frac{2\pi}{(A')^2} \right| \leq 20(2\pi)^3  \left( \frac{1}{L^4} + \frac{1}{(L')^4} \right).
\end{align}
Next, we combine the inequalities (\ref{L2eta}) and (\ref{A2eta}) to show that $\zeta_1$ and $\zeta_1'$ are close.  The following lemma provides a way of doing this.  We refer the reader to Lemma 6.2 of \cite{MagidThesis} for a detailed proof of this estimate, which is an intricate, but not deep, calculation.

\begin{lem} \label{lengths close implies w close}  Let $z_1, z_2 \in \mathbb{C}$, $|z_i| \geq (80)(2\pi)^2$,  and set
\[
L_i^2= \frac{|z_i|^2}{2Im(z_i)} \quad \text{and} \quad  A_i^2 = \frac{|z_i|^2}{2Re(z_i)}.
\]
Suppose 
\[
\left| \frac{2\pi}{L_1^2} - \frac{2\pi}{L_2^2} \right| \leq  16(2\pi)^3 \left(  \frac{1}{L_1^4}  +  \frac{1}{L_2^4} \right)
\]
\[
\left| \frac{2\pi}{A_1^2} - \frac{2\pi}{A_2^2} \right| \leq 20(2\pi)^3  \left( \frac{1}{L_1^4} + \frac{1}{L_2^4} \right).
\] Then
\[
|z_1 - z_2| < 560 (2\pi)^2 \frac{Im(z_1)}{|z_1|}.
\]

\end{lem}

Setting $z_1=  \zeta_1$ and $z_2 =\zeta_1'$, the inequalities (\ref{L2eta}) and (\ref{A2eta}), together with Lemma \ref{lengths close implies w close}, imply
\begin{align} \label{upper bound r}
\left| \zeta_1 - \zeta_1' \right| < 560 (2\pi)^2 \frac{Im( \zeta_1 )}{| \zeta_1|}.
\end{align}

By combining the lower bound from (\ref{lower bound r}) and the upper bound from (\ref{upper bound r}), we find that \[
\frac{\delta}{2} <560 (2\pi)^2 \frac{Im( \zeta_1 )}{| \zeta_1|}.
\]

Recall that the constant $\kappa$ was chosen in property (4) of $U$ so that $Im(q_1(\eta_m)) < \kappa - 1$ for any $\eta_m \in \Phi(C_n)$.  Thus $|r_1(\eta_m) - q_1(\eta_m)| < \frac{\delta}{4}$ for all $m$.  It follows that $Im(\zeta_1)$ is bounded above by a quantity that is independent of $n$:
\[
Im \left( \zeta_1 \right) \leq (\kappa - 1) + \frac{\delta}{4}  < \kappa.
\]
Since the Maskit slice  $\mathcal{M}^+$ is invariant under horizontal translations by $2$, we can assume for any point $w \in R$, $Re(w) > -2$ (see Lemma \ref{box lemma}).
 Since $\zeta_1$ lies in a closed $\frac{\delta}{4}$-neighborhood of $R+2n$, we have $| \zeta_1 | \geq  2n -2 - \frac{\delta}{4} > 2n - 3$.  It follows that 
\begin{align} \label{kappa}
\frac{\delta}{2} < 560 (2\pi)^2 \frac{Im(\zeta_1)}{| \zeta_1|} < \frac{ 560 (2\pi)^2 \kappa}{2n - 3}.  
\end{align}

Since $\kappa$ is independent of $n$, there are only finitely many $n$ that satisfy (\ref{kappa}).  Hence, for any $\rho \in  \overline{\Phi(C_n)} \cap \overline{\Phi(U-C_n)}$ with $n > \frac{560(2\pi)^2 \kappa}{\delta} + \frac{3}{2}$, inequality (\ref{kappa}) produces a contradiction. However, our assumption that $AH(N)$ was locally connected implied that $\overline{\Phi(C_n)} \cap \overline{\Phi(U-C_n)}$ is non-empty for all but finitely many $n$.  

It follows that $AH(N)$ is not locally connected at the point $\sigma_0 \in MP_0(N,P)$.
\end{proof}

\footnotesize



\begin{thebibliography}{widest-label} 
\bibitem
{AM} W. Abikoff, B. Maskit, ``Geometric decompositions of Kleinian groups,'' \textit{American Journal of Mathematics}, \textbf{99} 
(1977), 687-697.




\bibitem
{Agol}  I. Agol, ``Tameness of hyperbolic 3-manifolds,'' preprint (2004). Available at \url{http://front.math.ucdavis.edu/math.GT/0405568}.

\bibitem
{Ahlfors} L. Ahlfors, L. Bers, ``Riemann's mapping theorem for variable metrics,'' \textit{Annals of Mathematics}, \textbf{72} (1960), 385-404. 



\bibitem
{ACM} J. W. Anderson, R. D. Canary, D. McCullough, ``The topology of deformation spaces of Kleinian groups,"  \textit{Annals of Mathematics}, \textbf{152} (2000), 693-741.


\bibitem{BP} R. Benedetti, C. Petronio, \textit{Lectures on Hyperbolic Geometry,} Springer (1992).


\bibitem
{Bers2} L. Bers, ``Simultaneous uniformization,'' \textit{Bull. of the AMS}, \textbf{66} (1960), 98-103.

\bibitem
{Bers} L. Bers, ``Spaces of Kleinian groups,'' Maryland Conference in Several Complex Variables I. Springer-Verlag Lecture Notes in Math, \textbf{155} (1970), 9-34. 





\bibitem
{Bowditch} B. Bowditch, ``Geometrical finiteness for hyperbolic groups,'' \textit{J. Funct. Anal.}, \textbf{113} (1993), 245-317.

\bibitem
{Brock} J. Brock, ``Boundaries of Teichm\"uller spaces and end-invariants for hyperbolic 3-manifolds,'' \textit{Duke Mathematics Journal}, \textbf{106} (2001), 527-552.  


\bibitem
{BB} J. Brock, K. Bromberg, ``On the density of geometrically finite Kleinian groups,'' \textit{Acta Mathematica}, \textbf{192} (2004), 33-93.

\bibitem{BBES} J. Brock, K. Bromberg, R. Evans, J. Souto, ``Tameness on the boundary and Ahlfors' measure conjecture,'' \textit{Publ. Math. Inst. Hautes \'Etudes Sci.}, \textbf{98} (2003) 145-166.


\bibitem
{BCM2} J. Brock, R. D. Canary, Y. Minsky, ``The classification of Kleinian surface groups, II: The ending lamination conjecture,'' preprint (2004), Available at \url{http://front.math.ucdavis.edu/math.GT/0412006}.

\bibitem
{BCM} J. Brock, R. D. Canary, Y. Minsky, ``The classification of finitely generated Kleinian groups,'' in preparation.


\bibitem
{Bromberg Thesis} K. Bromberg, ``Rigidity of hyperbolic 3-manifolds with geometrically finite ends,"  Ph.D. Thesis, University of California at Berkeley (1998). 

\bibitem
{B1} K. Bromberg, ``Rigidity of geometrically finite hyperbolic cone-manifolds,'' \textit{Geom. Dedicata}, \textbf{105} (2004), 143-170. 

\bibitem
{B2} K. Bromberg, ``Hyperbolic cone-manifolds, short geodesics, and Schwarzian derivatives,'' \textit{Journal of the AMS}, \textbf{17} (2004), 783-826. 

\bibitem{Br4} K. Bromberg, ``Projective structures with degenerate holonomy and the Bers density conjecture,'' \textit{Annals of Mathematics}, \textbf{166} (2007), 77-93.

\bibitem
{B3} K. Bromberg, ``The space of Kleinian punctured torus groups is not locally connected," Preprint (2008). Available at \url{http://www.math.utah.edu/~bromberg/papers/nonlocal.pdf}.

\bibitem
{BH}  K. Bromberg, J. Holt, ``Self-bumping of deformation spaces of hyperbolic 3-manifolds," \textit{Journal of Differential Geometry}, \textbf{57} (2001), 47-65.

\bibitem
{BS} K. Bromberg, J. Souto, ``The density conjecture: A prehistoric approach,'' in preparation.

\bibitem{BM} R. Brooks, J. P. Matelski, ``Collars in Kleinian groups,'' \textit{Duke Mathematical Journal}, \textbf{49} (1982), 163-182.

\bibitem
{CG}   D. Calegari, D. Gabai, ``Shrinkwrapping and the taming of hyperbolic 3-manifolds,'' \textit{Journal of the AMS}, \textbf{19} (2006), 385-446.


\bibitem
{survey} R. D. Canary, ``Introductory bumponomics: The topology of deformation spaces of hyperbolic 3-manifolds,'' preprint (2008). Available at \url{http://www.math.lsa.umich.edu/~canary/}.


\bibitem
{CEG} R. D. Canary, D. B. A. Epstein, P. Green, ``Notes on notes of Thurston,'' \textit{Analytical and Geometric Aspects of Hyperbolic Space}, Cambridge University Press (1987), 3-92.

\bibitem
{CM} R. D. Canary, D. McCullough, \textit{Homotopy Equivalences of 3-Manifolds and Deformation Theory of Kleinian Groups},  Memoirs of the American Mathematical Society, \textbf{172} (2004). 


\bibitem
{Chuckrow} V. Chuckrow, ``On Schottky groups with applications to Kleinian groups,'' \textit{Annals of Mathematics}, \textbf{88} (1968), 47-61.



\bibitem
{CS} M. Culler, P. Shalen, ``Varieties of group representations and splittings of 3-manifolds,'' \textit{Annals of Mathematics}, \textbf{117} (1983), 109-146.

\bibitem
{HK1} C. D. Hodgson and S. P. Kerckhoff, ``Rigidity of hyperbolic cone-manifolds and hyperbolic Dehn surgery,'' \textit{Journal of Differential Geometry}, \textbf{48} (1998), 1-59.

\bibitem
{HK2} C. D. Hodgson and S. P. Kerckhoff, ``Universal bounds for hyperbolic Dehn surgery,'' \textit{Annals of Mathematics}, \textbf{162} (2005), 367-421. 





\bibitem
{Jorgensen} T. J\o rgensen, ``On discrete groups of M\"obius transformations,'' \textit{American Journal of Mathematics}, \textbf{98} (1976), 739-749.



\bibitem
{Kap} M. Kapovich, \textit{Hyperbolic Manifolds and Discrete Groups,}  Progress in Mathematics, Birkh\"auser, \textbf{183} (2001). 

\bibitem
{KS} L. Keen, C. Series, ``Pleating coordinates for the Maskit embedding of the Teichm\"uller space of punctured tori,''  \textit{Topology}, \textbf{32} (1993), 719-749.     


\bibitem
{KLO}  I. Kim, C. Lecuire, K. Ohshika, ``Convergence of freely decomposable Kleinian groups,'' preprint (2007).
 Available at \url{http://arxiv.org/abs/0708.3266}.



 \bibitem
 {KlSo}
G. Kleineidam, J. Souto, ``Algebraic convergence of function groups,'' \textit{Comment. Math. Helv.}, \textbf{77} (2002), 244-269. 





\bibitem
{Kojima} S. Kojima, ``Deformations of hyperbolic 3-cone manifolds,'' \textit{Journal of Differential Geometry}, \textbf{49} (1998), 469-516.


\bibitem
{Kra} I. Kra, ``On spaces of Kleinian groups,'' \textit{Comm. Math. Helv.}, \textbf{47} (1972), 53-69.

\bibitem
{Kra2} I. Kra, ``Horocyclic coordinates for Riemann surfaces and moduli spaces, I: Teichm\"uller and Riemann spaces of Kleinian groups,'' \textit{Journal of the AMS}, \textbf{3} (1990), 499-578.


\bibitem
{KuSh} R. Kulkarni, P. Shalen, ``On Ahlfors' finiteness theorem,'' \textit{Adv. Math.}, \textbf{76} (1989), 155-169.


\bibitem
{Lecuire} C. Lecuire, ``An extension of the Masur domain,'' \textit{Spaces of Kleinian Groups}, ed. by Y. Minsky, M. Sakuma and C. Series, London Math. Soc. Lecture Notes, \textbf{329} (2006), 49-73.

\bibitem
{Magid} A. D. Magid, ``Examples of relative deformation spaces that are not locally connected,'' \textit{Math. Annalen}, \textbf{344} (2009), 877-889.

\bibitem
{MagidThesis} A. D. Magid, ``Deformation spaces of Kleinian surface groups are not locally connected,'' Ph.D. Thesis, University of Michigan, 2009.  Available at \url{http://www.math.umd.edu/~magid/papers.html}.


\bibitem
{Marden open} A. Marden, ``The geometry of finitely generated Kleinian groups,'' \textit{Annals of Mathematics}, \textbf{99} (1974), 383-462.

\bibitem
{Marden} A. Marden, \textit{Outer Circles: An Introduction to Hyperbolic 3-Manifolds,}  Cambridge University Press (2007). 

\bibitem
{Maskit original} B. Maskit, ``Construction of Kleinian groups,'' \textit{Proc. of the Conf. on Complex Analysis}, Minneapolis (1964); Springer-Verlag (1965), 281-296.


\bibitem
{Maskit paper} B. Maskit, ``On Klein's combination theorem,'' \textit{Transactions of the AMS}, \textbf{120} (1965), 499-509.

\bibitem
{Maskit paper 2} B. Maskit, ``On Klein's combination theorem II,'' \textit{Transactions of the AMS}, \textbf{131} (1968), 32-39.


\bibitem
{Maskit} B. Maskit, ``Self-maps of Kleinian groups,'' \textit{American Journal of Math.}, \textbf{93} (1971), 840-856.

\bibitem
{McCullough} D. McCullough, ``Compact submanifolds of $3$-manifolds with boundary,'' \textit{Quart. J. Math. Oxford}, \textbf{37} (1986), 299-307.


\bibitem{McMullen book}  C. McMullen, \textit{Renormalization and 3-Manifolds which Fiber over the Circle,} Annals of Mathematical Studies, \textbf{142}, Princeton University Press (1996). 


\bibitem
{McMullen} C. McMullen, ``Complex earthquakes and Teichm\"uller theory,'' \textit{Journal of the AMS}, \textbf{11} (1998), 283-320.





\bibitem
{Minsky} Y. N. Minsky,  ``The classification of punctured-torus groups,"  \textit{Annals of Mathematics}, \textbf{149} (1999), 559-626.


\bibitem
{Morgan}  J. Morgan, ``On Thurston's uniformization theorem for three-dimensional manifolds,'' \textit{The Smith Conjecture}, Academic Press (1984), 37-125.

\bibitem
{NS} H. Namazi, J. Souto, ``Non-realizability and ending laminations,'' in preparation. 


\bibitem
{Ohshika2} K. Ohshika, ``Ending laminations and boundaries for deformation spaces of Kleinian groups,'' \textit{Journal of the London Math. Soc.}, \textbf{42} (1990), 111-121.

\bibitem
{Ohshika}  K. Ohshika, ``Realising end invariants by limits of minimally parabolic, geometrically finite groups,'' preprint (2009). Available at \url{http://arxiv.org/abs/math/0504546}.

\bibitem
{Purcell} J. Purcell, ``Cusp shapes under cone deformation,'' \textit{Journal of Differential Geometry}, \textbf{80} (2008), 453-500.  


\bibitem
{Scott} G. P. Scott, ``Finitely generated 3-manifold groups are finitely presented,'' \textit{Journal of the London Math. Soc.(2)}, \textbf{6} (1973), 437-440.

\bibitem
{Sullivan} D. Sullivan, ``Quasiconformal homeomorphisms and dynamics II: Structural stability implies hyperbolicity for Kleinian groups,'' \textit{Acta Mathematica}, \textbf{155} (1985), 243-260.

\bibitem
{Th1} W. P. Thurston, ``The geometry and topology of three-manifolds,'' Princeton University Course Notes (1980). Available at \url{http://www.msri.org/publications/books/gt3m/}.

\bibitem
{Th2} W. P. Thurston, ``Three-dimensional manifolds, Kleinian groups and hyperbolic geometry,'' \textit{Bull. Amer. Math. Soc. (N.S.)}, \textbf{6} (1982), 357-381.


\bibitem
{Th4} W. P. Thurston, ``Hyperbolic structures on 3-manifolds, II: Surface groups and 3-manifolds which fiber over the circle,'' preprint (1986), available at \url{http://front.math.ucdavis.edu/9801.5045}.


\bibitem
{Waldhausen} F. Waldhausen, ``On irreducible 3-manifolds which are sufficiently large,'' \textit{Annals of Mathematics}, \textbf{87} (1986), 56-88.

\bibitem
{Wright} D. J. Wright, ``The shape of the boundary of Maskit's embedding of the Teichm\"uller space of once punctured tori,'' preprint (1990). 
\end{thebibliography}

\textsc{University of Maryland \\ Department of Mathematics \\ 1301 Mathematics Building \\ College Park, MD 20742-4015}

\end{document}